\documentclass{preprint}
\usepackage[utf8x]{inputenc}
\usepackage[english]{babel}
\usepackage[osf]{newtxtext}

\usepackage{amsmath,amssymb,xcolor, tikz, graphicx, mathtools}
\usepackage{tensor}

\usepackage[toc,section=section]{glossaries}

\usepackage{xifthen}
\usepackage{makecell, booktabs}
\usepackage{mhequ,mhenvs}
\usepackage{orcidlink}

\usepackage{microtype}

\usepackage{comment}
\usepackage{hyperref}

\DeclareMathAlphabet{\mathbf}{OT1}{ptm}{bx}{n}
\DeclareMathAlphabet{\mathrm}{OT1}{ptm}{m}{n}
\DeclareMathAlphabet{\mathbb}{U}{jkpsyb}{m}{n}

\usetikzlibrary{arrows.meta,calc,decorations.markings,math,patterns, shapes}
\tikzstyle{terminator} = [rectangle, draw, text centered, rounded corners, minimum height=2em]
\tikzstyle{process} = [rectangle, draw, text centered, minimum height=2em]
\tikzstyle{decision} = [diamond, draw, text centered, minimum height=2em]
\tikzstyle{data}=[trapezium, draw, text centered, trapezium left angle=60, trapezium right angle=120, minimum height=2em]
\tikzstyle{connector}  = [draw, very thick]
\tikzstyle{connectorG} = [draw, green!70!gray, very thick]
\tikzstyle{connectorR} = [draw, red!70!gray, very thick]
\tikzstyle{arrow} = [very thick,->,>=stealth]

\newtheorem{assumption}[lemma]{Assumption}

\DeclareSymbolFont{timesoperators}{T1}{ptm}{m}{n}
\SetSymbolFont{timesoperators}{bold}{T1}{ptm}{b}{n}
\makeatletter
\renewcommand{\operator@font}{\mathgroup\symtimesoperators}
\makeatother

\DeclareMathOperator{\Tr}{Tr}

\newcommand{\eqdef}{\stackrel{\mbox{\tiny\rm def}}{=}}

\def\dash{\leavevmode\unskip\kern0.18em--\penalty\exhyphenpenalty\kern0.18em}
\def\slash{\leavevmode\unskip\kern0.15em/\penalty\exhyphenpenalty\kern0.15em}

\colorlet{darkblue}{blue!90!black}
\colorlet{darkgreen}{green!50!black}
\let\comm\comment

\newcommand\bpi{\boldsymbol{\pi}}

\newcommand\myl{\raisebox{0.13em}{\mbox{\tiny $\boldsymbol{<}$}}}
\newcommand\myg{\raisebox{0.13em}{\mbox{\tiny $\boldsymbol{>}$}}}
\newcommand\myc{\raisebox{0.1em}{\mbox{\small $\circ$}}}
\newcommand\myleq{\raisebox{0.13em}{\mbox{\tiny $\leqslant$}}}
\newcommand\mygeq{\raisebox{0.13em}{\mbox{\tiny $\geqslant$}}}

\def\one{\mathord{\mathbf{1}}}
\def\a{{\mathord{\mathbf a}}}

\newcommand{\RR}{\mathbb{R}}
\newcommand{\NN}{\mathbb{N}}
\newcommand{\ZZ}{\mathbb{Z}}
\newcommand{\TT}{\mathbb{T}}
\newcommand{\CC}{\mathbb{C}}
\newcommand{\PP}{\mathbb{P}}
\newcommand{\EE}{\mathbb{E}}

\newcommand{\mf}[1]{\mathfrak{#1}}

\newcommand*{\ud}{\mathrm{\,d}}

\newcommand{\mX}{\mathcal{X}}
\newcommand{\mF}{\mathcal{F}}
\newcommand{\mA}{\mathcal{A}}
\newcommand{\mL}{\mathcal{L}}
\newcommand{\mP}{\mathcal{P}}
\newcommand{\mM}{\mathcal{M}}
\newcommand{\mN}{\mathcal{N}}
\newcommand{\mS}{\mathcal{S}}
\newcommand{\mB}{\mathcal{B}}
\newcommand{\mD}{\mathcal{D}}
\newcommand{\mC}{\mathcal{C}}
\newcommand{\mZ}{\mathcal{Z}}
\newcommand{\mO}{\mathcal{O}}

\newcommand{\mQ}{\mathcal{Q}}
\newcommand{\mH}{\mathcal{H}}

\newcommand{\tC}{\texttt{C}}
\newcommand{\tD}{\texttt{D}}

\renewcommand{\l}{\ell}

\def\K{\mathbf{K}}
\def\CE{\mathcal{E}}

\newcommand{\ve}{\varepsilon}
\newcommand{\vr}{\varrho}
\newcommand{\vt}{\vartheta}

\DeclareMathOperator{\supp}{supp}
\let\f\frac
\let\eps\varepsilon

\makeglossaries

\newglossarystyle{tbl}
{%
     {\begin{longtable}{l>{\raggedright}p{0.65\textwidth} } \toprule}%
     {\bottomrule \end{longtable}}%
  \renewcommand*{\glsgroupheading}[1]{}%
  \ifglsnogroupskip
  \else
  \fi
}

\setglossarystyle{tbl}

\newglossaryentry{con}
{
        name= ~Concentrated~,
        symbol={$ \| (\Pi_{M_{t}-1}^{\myc} + \Pi_{M_{t}-2}^{\myc}) u_{t} \| \geqslant \frac{1}{4}
\| \Pi^{\myl}_{M_{t}-2} u_{t} \| $},
	description={The system is concentrated if mass
accumulates close to the energy median.}
}
\newglossaryentry{dil}
{
        name= ~Diluted~,
        symbol={$ \| (\Pi_{M_{t}-1}^{\myc} + \Pi_{M_{t}-2}^{\myc}) u_{t} \| < \frac{1}{4} \|
\Pi^{\myl}_{M_{t}-2} u_{t} \| $},
	description={If the system is not concentrated, then it is diluted.}
}
\newglossaryentry{low}
{
        name= \phantom{A}~\ensuremath{\Pi_{L}^{\myl}},
        symbol={ $ \Pi^{\myl}_{L} \varphi^{\alpha} = \sum_{| k | \leqslant
L_{\alpha}} \langle \varphi^{\alpha}, e_{k} \rangle e_{k}$},
	description = {Projection on frequencies lower than $ L $.}
}
\newglossaryentry{high}
{
        name= \phantom{A}~\ensuremath{\Pi_{L}^{\myg}},
	description = {Projection on frequencies strictly higher than $ L +1 $.},
        symbol={ $ \Pi^{\myg}_{L} \varphi^{\alpha} = \sum_{| k |
>(L+1)_{\alpha}}
 \langle \varphi^{\alpha}, e_{k} \rangle e_{k} $}
}
\newglossaryentry{central}
{
        name= \phantom{A}~\ensuremath{\Pi_{L}^{\myc}},
        symbol={ $ \Pi_{L}^{\myc}  \varphi^{\alpha} = \sum_{L_{\alpha} < | k |
\leqslant (L+1)_{\alpha}}  \langle \varphi^{\alpha}, e_{k} \rangle e_{k}$},
	description = {Projection on frequencies of order $ L$.},
}
\newglossaryentry{energy median}
{
        name= \phantom{A}~\ensuremath{M(u)},
        symbol={ $\inf_{M \in \NN^{+}} \{  \|
\Pi^{\mygeq}_{M} u \| \leqslant \| \Pi^{\myl}_{M} u \| \} $ },
	description = {Energy median.},
}
\newglossaryentry{eigenvalue}
{
        name= \phantom{A}~\ensuremath{\zeta_{k}},
        symbol={ $ \zeta_{k} = | k |^{2 a}$ },
	description = {Eigenvalue of the operator $ (- \Delta)^{a} $ associated
to the Fourier mode $ k \in \ZZ^{d} $},
}
\newglossaryentry{gap}
{
        name= \phantom{A}~\ensuremath{\Delta_L},
        symbol={ $\zeta_{L+1} - \zeta_{L}$ },
	description = {Gaps between the eigenvalues of $ (- \Delta)^{a} $ in
dimension $ d=1 $, used in any dimension.},
}
\newglossaryentry{w}
{
        name= \phantom{A}~\ensuremath{w (L; t, x) },
        symbol={ $\Pi^{\myg}_{L} u_{t}(x) / \| \Pi^{\myl}_{L}
u_{t} \| $ },
	description = {No central frequencies.},
}
\newglossaryentry{wbold}
{
        name= \phantom{A}~\ensuremath{w^{\mygeq}(L; t, x)},
        symbol={ $\Pi^{\mygeq}_{L} u_{t}(x) / \|
\Pi^{\myl}_{L} u_{t} \| $ },
	description = {Centered around $ N_{t} = M_{t} -1 $,
unless stated otherwise.},
}
\newglossaryentry{sbold}
{
        name= \phantom{A}~\ensuremath{\sigma^{\mygeq}_{\beta}(L, t_0)},
        symbol={ $ \inf\{ t \geqslant t_{0}  \; \colon \; \|
\Pi^{\mygeq}_{L} u_{t}\| \geqslant \beta \| \Pi^{\myl}_{L}
u_{t} \|\}\wedge (t_{0} +1) $ },
	description = {$ \sigma^{\mygeq} = \sigma^{\mygeq}_{3/2}$.},
}
\newglossaryentry{taul}
{
        name= \phantom{A}~\ensuremath{\tau^{\myl} (L,t_{0}) },
        symbol={ $\inf \{ t \geqslant t_{0}  \ \colon \ \|
\Pi_{ L}^{\myl}  u_{t} \| \geqslant 2 \| \Pi_{
L}^{\myg}  u_{t} \|  \} \wedge (t_{0} + 1)$ },
	description = {Lower exit time. $ \tau^{\myl} (t_{0}) = \tau^{\myl} (
N_{t_{0}}, t_{0})$.},
}
\newglossaryentry{numin}
{
        name= \phantom{A}~\ensuremath{\nu_{\rm{min}}},
        symbol={ $ \min_{\alpha = 1, \dots, m} \nu^{\alpha}$ },
	description = {aa},
}
\newglossaryentry{numax}
{
        name= \phantom{A}~\ensuremath{\nu_{\rm{max}}},
        symbol={ $ \max_{\alpha = 1, \dots, m} \nu^{\alpha}$ },
	description = {aa},
}
\newglossaryentry{la}
{
        name= \phantom{A}~\ensuremath{L_{\alpha}},
        symbol={ $( \nu^{\alpha})^{- \frac{1}{2 \a}} L $ },
	description = {aa},
}

\begin{document}

\title{Spectral gap for projective processes of linear SPDEs}

\author{Martin Hairer$^{1,2}$ \orcidlink{0000-0002-2141-6561} and Tommaso
Rosati$^3$ \orcidlink{0000-0001-5255-6519}}
\institute{EPFL, Switzerland, \email{martin.hairer@epfl.ch}
\and Imperial College London, UK, \email{m.hairer@imperial.ac.uk}  
\and University of Warwick, UK, \email{t.rosati@warwick.ac.uk} }

\maketitle

\begin{abstract}
  This work studies the angular component $ \pi_{t} = u_{t} / \|
  u_{t} \| $ associated to the solution $ u $
  of a vector-valued linear hyperviscous SPDE on a $ d $--dimensional torus
  \begin{equ}
    \ud  u^{\alpha}  =- \nu^{\alpha} (- \Delta)^{\a } u^{\alpha} \ud t  + (u
\cdot \ud W)^{\alpha} \;,
\quad \alpha \in \{ 1, \dots, m \} 
  \end{equ}
  for $ u \colon \TT^{d} \to \RR^{m} $,  $ \a  \geqslant  1 $ and a sufficiently
smooth and non-degenerate noise $
  W $. We provide conditions for existence, as well as uniqueness
  and spectral gaps (if $ \a  > d/2$) of
  invariant measures for $ \pi $ in the projective space.
  Our proof relies on the  introduction of a novel Lyapunov functional for $
  \pi_{t}$, based on the study of dynamics of the ``energy median'': the energy
level $M$ at which projections of $u$ onto frequencies with energies less or more than $M$ have about
equal $L^2$ norm. This technique is applied
  to obtain \dash in an infinite-dimensional setting without order preservation
\dash lower bounds on top Lyapunov exponents of the equation, and their
  uniqueness via Furstenberg--Khasminskii formulas.  \\[.4em]
\end{abstract}

\setcounter{tocdepth}{2}

\tableofcontents

\section{Introduction}

We consider, for fixed $ d, m \in \NN $, a parameter $ \a  \geqslant  1 $ and $
\nu^{\alpha} >0 $ for all $
\alpha  \in \{ 1, \dots, m \} $, the solution $ u = (u^{\alpha})_{\alpha =1}^{m} \colon [0,
\infty) \times \TT^{d} \to \RR^{m} $ ($ \TT^{d} $ being the $ d $-dimensional
torus) to the vector-valued linear stochastic PDE 
\begin{equ}\label{eqn:main}
\ud u^{\alpha} = - \nu^{\alpha} (- \Delta)^{\a } u^{\alpha} \ud t + (u
\cdot \ud W)^{\alpha }\;, \ \
u^{\alpha }(0, x)= u_{0}^{\alpha } (x)\;, \ \ \forall \alpha  \in \{ 1, \dots, m \} \;,
\end{equ}
for $ (t, x) \in [0, \infty) \times \TT^{d}$ and $ u_{0} \in L^{2}_{\star} =
L^{2} \setminus \{ 0 \} $. The noise $ W $ is chosen white in time, translation
invariant and
sufficiently smooth in space for classical solution
theories to apply. 
In this setting, the multiplicative ergodic theorem guarantees that for
every $ u_{0} \in L^{2} $ the Lyapunov exponent
\begin{equ}
  \lambda (u_{0}) = \lim_{t \to \infty} \frac{1}{t} \log{ \| u_{t} \|} \in \RR
  \cup \{ \pm \infty \}
\end{equ}
exists, is deterministic, and under mild assumptions on the noise satisfies $
\lambda (u_{0}) < \infty $. The converse bound, namely $ \lambda (u_{0}) > - \infty $,
is however unknown in general. The aim of the present work is to prove that under
relatively weak non-degeneracy conditions there exists a uniform lower
bound on $ \lambda (u_{0}) $ over all initial conditions $
u_{0}$:
\begin{equ}[eqn:lower-bd-lyap]
  \inf_{u_{0} \in L^{2}_{\star}} \lambda(u_{0}) > - \infty \;.
\end{equ}
In addition, under stronger assumptions we show that the Lyapunov
exponent does not depend at all on the initial condition: $ \lambda (u_{0}) =
\lambda (v_{0}), $ for all \( u_{0}, v_{0} \in L^{2}_{\star} \). In this case we 
prove Furstenberg--Khasminskii type formulas for the exponent.

Both results build on the study of the angular component $ \pi_{t} = u_{t} / \| u_{t} \| $ of the
solution to \eqref{eqn:main}, which is at the heart of the present work. While
in finite dimensions the link between properties of the Lyapunov exponent and
ergodic properties of the process $ \pi_{t} $ has been extensively used
(see for example the monograph \cite{MR600653}), in
infinite dimensions such an approach has remained mostly inaccessible. The
main difficulty is that the process
$ \pi_{t} $ no longer takes values in a compact state space, so the proof
of existence of invariant measures requires special care. To the best of
our knowledge, the only infinite-dimensional case that has been treated to some
extent is the order-preserving one \cite{Order,Monotone}, namely when the dynamic of the linear equation
preserves both cones of positive and negative functions. The most prominent
example in this setting is \eqref{eqn:main} with $ \a  =  1$ and $m=1 $, which is linked via the Cole--Hopf transform to
both the KPZ and Burgers' equation. Here, under the assumption $ \pi_{0} \geqslant 0 $, geometric ergodicity of $ \pi_{t} $ \dash
in fact, even a pathwise contraction
and a one force one solution principle \dash can be established, building on
suitable generalizations of the Krein--Rutman theorem,
see for example the seminal work by Sinai \cite{Sinai1991Buergers} and many
later works and the references therein
\cite{Boritchev16Burgulence,DunlapGrahamRyzhik21Burgers,RosatiSynchro}. Yet,
even in the case $ \a  =1$ and $m =1 $, no proof of \eqref{eqn:lower-bd-lyap}, which
allows for arbitrary initial data, seems to
be available, although a simpler proof than ours appears very plausible. Indeed for
sufficiently non-degenerate noise one would expect the dynamic of $ u $ to
eventually be trapped in either the cone of positive or that of negative
functions, even if the initial condition has no definite sign, which then reduces to the known case.

Therefore, although in the regime $ \a  > 1 $ the local solution theory of \eqref{eqn:main}
is simplified by the additional smoothing, this is an interesting regime
from our perspective since no order preservation property is satisfied and 
our study of $ \pi_{t} $ cannot rely on previous methods. Even more
interesting is the case $ \a  = 1 $ but $ m \geqslant 2 $ which is relevant in the study of fluid dynamics. The difficulty we 
are faced with is not just technical: it is easy to see
that our results simply do not hold in the deterministic case $ W = 0$, since
the Laplacian does not have a bounded spectrum, contradicting
\eqref{eqn:lower-bd-lyap}. Indeed, a 
sufficiently non-degenerate noise is required, in opposition to the mentioned
order preserving case, with positive or negative initial data, where the
noise is not necessary and if $ W = 0 $ the result reduces to the
Krein--Rutman theorem.

The linchpin on which our argument hinges is a novel Lyapunov functional
for the process $ \pi_{t} $, which can be obtained under rather mild regularity and
non-degeneracy assumptions on the noise $ W $.
Its definition is based on the analysis of the dynamics of energy levels of the
angular process \( \pi_{t} \). 
The main issue that has to be overcome is that in the deterministic dynamic 
($ W = 0 $) every eigenfunction of the Laplacian is a fixed point for the deterministic
angular dynamic, so that $ \pi_{t} $ can get stuck in high-frequency
states and, for large times, converges to the eigenfunction associated to
its smallest non-zero Fourier mode.
On the other hand, every eigenfunction but the one associated to the top
eigenvalue is a saddle point, the unstable directions being given by all
eigenfunctions with strictly smaller wave number. One therefore expects
that, provided the noise is sufficiently non-degenerate, the process is unlikely to get trapped by these
critical points  and bounds such as \eqref{eqn:lower-bd-lyap} become more plausible.

Our approach to making this heuristic rigorous is to measure the high-frequency
state in which the process $ \pi_{t} $ finds itself through the ``energy median''
$ M(\pi_{t}) $, a level at which roughly half of its $ L^{2} $ energy lies in
frequencies both above and below $ M(\pi_{t}) $. We use the energy median to
distinguish between large scales, where we have to deal with dynamical phenomena such
as the one just explained, and small scales, where we expect to be able to exploit
the strong dissipativity of the  equation.

Eventually, the construction of the Lyapunov functional builds on small scale
(or high-frequency) hyperviscous regularity estimates, together with a drift
towards low frequencies for the energy median. The proof of
the latter result requires to distinguish two cases. On the one hand a
\emph{diluted} case, in which the energy is spread out, also in frequencies
distant from the level $ M(\pi_{t}) $: here the negative drift is simply a
consequence of dissipation.
On the other hand a \emph{concentrated}
case, in which most of the energy is to be found around frequency $
M(\pi_{t}) $: this is where the deterministic analysis alone
cannot be sufficient and we solve a control problem to prove that a non-degenerate
noise rapidly pushes the system out of high-frequency concentrated states.

Having constructed the Lyapunov functional, the proof of the
uniform lower bound on the Lyapunov exponent follows by a bootstrap argument
which delivers sufficiently good regularity estimates for $ \pi $.
Furthermore, uniqueness of the
invariant measures follows from Harris' theorem, when viewing \( \pi \) as a \emph{projective}
process, that is identifying $ \pi $ and $ - \pi $. Here
the proof requires some regularity for the law of $ \pi_{t} $ near the constant
eigenfunction $ \pi\equiv 1/ \sqrt{m}  $ (the constant is chosen to have unit
$ L^{2} $ norm), which imposes among others the much
stronger requirement $ \a > d/2 $ to make use of the Bismut--Elworthy--Li formula:
see Remark~\ref{rem:on-assumption} for a further discussion of this point.

To conclude this introduction, let us discuss the relevance of our results. The
study of ergodicity for projective processes is fundamental to obtain a
precise control on top Lyapunov exponents of SPDEs, from properties such as
finiteness, uniqueness, Furstenberg--Khasminskii formulas and continuous dependence on the parameters of the
equation \cite{rosati2021lyapunov}, to central limit theorems for the sample Lyapunov exponent
\cite{DuGuKo21Fluct},
to precise estimates on the stability or instability of nonlinear equations.
Especially the study of nonlinear SPDEs close
to an invariant manifold still presents many open challenges. In the order
preserving case, a recent work
\cite{MuellerKhoshnevisan20Phase} 
considers the stability or instability of a nonlinear equation similar to
\eqref{eqn:main} with $ \a  = 1 $ close to the fixed point $ u \equiv 0 $. The
authors show that if the noise is either very weak or very strong, then the equation is locally unstable (respectively stable). Beyond the
order preserving case, very little is known, see for example
\cite{BianchBlomker21SwiftHohenberg} for quartic ($ \a  =2 $) equations in a
small noise regime, where the argument of the authors relies on a
transformation that reduces the problem to the study of an order preserving
system. Similarly relevant is the recent work \cite{gess2022lyapunov} (related
to a classical result \cite{crauel1998additive} for finite-dimensional systems): here
the particular structure of the Allen--Cahn equation can be used to prove the
negativity of the Lyapunov exponent, although so far there is not proof of a
lower bound to the exponent.

In the finite dimensional case, much more precise tools are available. Building on spectral
gaps for projective processes one can implicitly construct Lyapunov functionals
for nonlinear problems close to unstable invariant manifolds
\cite{Hairer09Heat}: this allows to establish and quantify the stability or
instability of an invariant manifolds in terms of the sign of the top Lyapunov
exponent, as should be expected. If extended to infinite dimensions, these tools can in
principle be used to address open problems, such as the \textit{non}-uniqueness of invariant measures for the
Navier--Stokes equations under degenerate forcing, as opposed to uniqueness
when the noise satisfies minimal non-degeneracy conditions
\cite{HairerMattingly06}, see for example \cite{HairerCotiZelati21Lorenz} for
 a result with the same underlying motivation, but in the case of a 3D Lorenz system.

A series of nice results which are related in both motivation and flavour to this article was recently
obtained by Bedrossian, Blumenthal, and Punshon-Smith.
They obtained two types of results for the stochastically forced Navier--Stokes equations.
On one hand, they study in \cite{BedrossianBlumenthalSmith22LowerBoundLyap,
  BedrossianBlumenthalSmith22Lagrangian,
BedrossianBlumenthalSmith22BatchlorSpectrum} the behaviour of passive tracers advected by the corresponding
random velocity field. They show that as long as the forcing is sufficiently non-degenerate for the
strong Feller property to hold, these exhibit ``Lagrangian chaos'', namely tracers started nearby
separate exponentially fast.
On the other hand, they show in \cite{BedrossianSmith21chaos} that its finite-dimensional Galerkin truncations exhibit ``Eulerian chaos'',
namely that the top Lyapunov exponent for the linearised (with respect to initial data) equation 
is strictly positive at high enough Reynolds number.
Extending the latter result to
infinite dimensions presents many fundamental challenges.
Even establishing the continuity of the Lyapunov exponent
with respect to the size of the finite-dimensional approximation, or a 
lower bound in the spirit of \eqref{eqn:lower-bd-lyap} remain open problems.
Once more, this is not out of purely technical reasons
since in the deterministic Navier--Stokes system, results concerning enhanced dissipation
show that arbitrarily large exponential, or even super-exponential decay is
possible \cite{BedrossianMasmoudiVicol16InvDampEnhancedDiss, BedrossianCoti17Enhanced}.

In conclusion, this article introduces a new approach to
study Lyapunov exponents and projective processes for
SPDEs beyond the order preserving case. The long-term goal of these methods is to
tackle some of the problems described above.

\subsection*{Acknowledgments}
This article was written in vast majority while TR was employed at
Imperial College, London.
Support from the Royal Society through MH's research
professorship, grant RP\textbackslash R1\textbackslash 191065, is
gratefully acknowledged.
TR is very grateful to Alex Blumenthal and Sam Punshon-Smith for many inspiring
discussions.

\subsection*{Notations} 

We write $ \NN = \{ 0, 1, 2,3 \ldots \}, \ \NN^{+} = \NN \setminus \{ 0
\} $, and $ \TT^{d} =
\RR^{d}/\ZZ^{d} $ for the $ d $-dimensional Torus. We denote by $ \| \cdot \| $
the
\( L^{2}(\TT^{d}; \CC^{n}) \) norm $ \| \varphi
\|^{2} =  \int_{\TT^{d}} |\varphi|^{2}(x)  \ud x
$, when $ n $ is clear from context and we write $ L^{2}_{\star}
(\TT^{d}; \CC^{n}) = L^{2}(\TT^{d}; \CC^{n})\setminus \{ 0 \} $. For $ x  \in \CC^{n} $ we write $ | x | $ for its
Euclidean norm. We also write for $ \varphi, \psi \in L^{2}(\TT^{d};
\CC^{n}) $:
\begin{equation*}
\begin{aligned}
\langle \varphi, \psi \rangle = \sum_{\alpha=1}^{n} \int_{\TT^{d}} \varphi^{\alpha} (x)
\psi^{\alpha}(x) \ud x \;.
\end{aligned}
\end{equation*}
We will control high frequency regularity, for
$ L \in \NN $ and $ \gamma \geqslant 0 $, via the
(semi-)norms
\begin{equs}
  \| \varphi \|_{H^{\gamma}(\TT^{d}; \RR^{n})}^2 & =  \sum_{\alpha=1}^{n}\sum_{k \in \ZZ^{d}} (1 + | k |)^{2} |
  \hat{\varphi}^{\alpha} (k) |^{2} \;, \\
  \| \varphi \|_{H^{\gamma}_{L}}^2 & = \sum_{\alpha =1}^{n} \sum_{| k | >
L_{\alpha} } (1+| k |- L_{\alpha})^{2 \gamma} |
  \hat{\varphi}^{\alpha}(k) |^{2}  \;, \qquad L_{\alpha} =
(\nu^{\alpha})^{- \frac{1}{2 \a }} L\;,
\end{equs}
where $ (\nu^{\alpha})_{\alpha}  $ are viscosity coefficients.
Further, to simplify the notation we write
\begin{equ}
  \| \varphi \|_{H^{\gamma}_{L}} = \| \varphi \|_{\gamma, L} \;, \qquad \|
  \varphi \|_{H^{\frac{1}{2}}_{L}} = \| \varphi \|_{L} \;.
\end{equ}
For any set $ \mX $ and functions $ f, g \colon
\mX \to \RR $ we write $ f \lesssim g $ if there exists a constant $ c >0 $
such that $ f(x) \leqslant c g(x)$ for all $ x \in \mX $.
The (lack of) dependence of $c$ on additional parameters will hopefully either be clear
from context or explicitly specified. For a complex number $ z \in \CC $, we
write
$\mf{Re} (z)$, $\mf{Im}(z)$,
for its real and imaginary parts.
We write $ \zeta_{k} $ for the  eigenvalue of $  (- \Delta)^{\a } $ associated to mode
$ k \in \ZZ^{d}$, and $ \Delta_{M} $ for the one-dimensional gaps:
\begin{equation*}
\begin{aligned}
\gls{eigenvalue} = | k |^{2 \a } \;, \qquad \gls{gap} = \zeta_{L+1}- \zeta_{L} \;,
\qquad \forall k \in \ZZ^{d} \;, \quad L \in \NN \;.
\end{aligned}
\end{equation*}
Throughout the article we will consider the solution $
(u_{t})_{t \geqslant 0} $ to a stochastic PDE on a probability space $
(\Omega, \mF, \PP) $ and we write $
(\mF_{t})_{t \geqslant 0} $ for the right-continuous filtration generated by the Wiener process
driving $u$. We use the following notation for conditional expectations, given any stopping time $ t_{0} $
\begin{equ}
\EE_{t_{0}}[ \varphi] = \EE[\varphi \vert \mF_{t_{0}}]\;, \quad
\PP_{t_{0}}(\mA) = \PP( \mA \vert \mF_{t_{0}})\;.
\end{equ}

\section{Preliminaries and main results}
We start with some general considerations, which hold for a wide class of
matrix-valued noises $ \ud W $. First, for any $ (t,x) \in
[0, \infty) \times \TT^{d} $ we write in coordinates $ W_{t} (x) =
(W^{\alpha}_{\beta, t}(x))_{\alpha, \beta =1}^{m} \in (\RR^{m})^{\otimes 2}  $.
Hence in \eqref{eqn:main} we find the matrix multiplication
\begin{equ}[e:dotproduct]
(u \cdot \ud W_{t})^{\alpha} = \sum_{\beta} u^{\beta} \ud W^{\alpha}_{\beta, t} 
\;.
\end{equ}
Now, for the sake of this introductory section, let us consider a generic
\emph{spatially homogeneous} noise $ W $. By this, we mean that
\begin{equation*}
\ud \langle W^{\alpha}_{\beta}(x), W^{\alpha^{\prime}}_{\beta^{\prime}}
(y) \rangle_{t} = \Lambda^{\alpha, \alpha^{\prime}}_{\beta, \beta^{\prime}}
(x-y)\ud t \;, \qquad \forall x, y \in \TT^{d} \;,
\end{equation*}
for a tensor-valued function $ \Lambda \colon \RR \to (\RR^{m})^{\otimes 4} $.
In Fourier coordinates, such a noise can be written as
\begin{equ}
\ud W^{\alpha}_{\beta, t } =  \sum_{k \in
\ZZ^{d}}  e_{k} \ud B^{\alpha, \beta}_{k, t} \;,
\end{equ}
with $ e_{k}(x) = \exp{ ( \iota k \cdot x )} $ for $ x \in
\TT^{d} $ and with  $$ \{
B^{\alpha, \beta}_{k, t}  \; \colon \; k \in \ZZ^{d} \;, \alpha, \beta \in
\{ 1, \dots, m \} \;, t \geqslant 0 \} $$ a collection of rescaled complex
Brownian motions with quadratic
covariation of the form
\begin{equ}[e:noisefc]
 \ud \langle 
B^{\alpha, \beta}_{k} ,
B^{\alpha^{\prime}, \beta^{\prime}}_{k^{\prime}}
\rangle_{t} = \Gamma^{\alpha, \alpha^{\prime}}_{\beta,
\beta^{\prime}, k} \delta_{k,- k^{\prime}}\ud t \;, \qquad B^{\alpha,
\beta}_{-k, t} = \overline{B}^{\alpha, \beta}_{k, t}\;, 
\end{equ}
such that
\begin{equ}[e:def-Lambda]
\Lambda^{\alpha, \alpha^{\prime}}_{\beta, \beta^{\prime}}
(x) = \sum_{k \in \ZZ} \Gamma^{\alpha, \alpha^{\prime}}_{\beta, \beta^{\prime},
k} e_{k}(x)\;.
\end{equ}
For the time being, we will refrain from stating more precise assumptions on
the noise coefficients $ \Gamma $ appearing above: the assumptions will be
provided in the upcoming section. Instead we proceed with some heuristic
arguments, assuming that the noise is sufficiently smooth for \eqref{eqn:main}
to be well posed. Throughout this work we will decompose the solution as $ u_{t} =
r_{t} \pi_{t} $, with 
\begin{equ}[e:defrpi]
r_{t} = \| u_{t} \|\;, \qquad \pi_{t} = \frac{u_{t}}{
\| u_{t} \|}\;.
\end{equ}
We will refer to the first as the \textit{radial} and to the second as the
\textit{angular} component of the solution $ u_{t} $. 

The natural state
space for the angular component is the infinite-dimensional sphere 
\begin{equ}
S = \{
\varphi \in L^{2}(\TT^{d}; \RR^{m})  \ \colon \ \| \varphi
\|_{L^{2}(\TT^{d}; \RR^{m})} = 1 \} \;,
\end{equ}
but on
this space the process may in principle have multiple invariant measures. For example, in the case
$ \a  =1 $, one has at least two invariant measures since the equation preserves the cones of positive and
negative functions. If $ \a  \geqslant 1 $ and the noise is sufficiently non-degenerate,
then it follows from our results that there are at most two invariant measures,
since any invariant measure must contain either $ e_{0} $ or $ - e_{0} $ in its
support and the angular component is strong Feller at these points, see the
proof of Theorem~\ref{thm:uniq}. In general, we cannot expect less than two
invariant measures, since in the case $
\a =1, m=1$, the angular component $ \pi_{t} $ preserves the cones of positive and
negative functions. On the other hand, if $ \a  > 1 $, since the
cones of positive and negative functions are no longer preserved, one might expect
exactly one invariant measure for the process $ \pi_{t} $, but this falls
beyond the scope of this work.

Motivated by the case $ \a  =1, m=1 $, the natural space to check for
uniqueness of the invariant measure is therefore
the infinite dimensional projective space $ \mathbf{P} $, which can be viewed
as the Hilbert manifold obtained by
quotienting the sphere with the antipodal map. We define
\begin{equ}
  \mathrm{Ap} \colon S \to S\;,
  \qquad \mathrm{Ap}(\varphi) = - \varphi\;, \qquad \mathbf{P} =
  S / \mathrm{Ap} \;,
\end{equ}
and we denote by $ [\varphi] \in \mathbf{P} $ the equivalence class of $ \varphi \in S $
in the projective space:
\begin{equ}
  \left[ \; \cdot \; \right] \colon S \to \mathbf{P} \;, \qquad [\varphi] = [-
  \varphi] \;.
\end{equ}
We observe that due to the linearity of \eqref{eqn:main}, the process $
( [\pi_{t}] )_{t \geqslant 0}
$ is a Markov process on $ \mathbf{P} $ which we call the
\emph{projective} process associated to $ u $. Before we move on to state our
main results, let us perform some formal calculations for $ r_{t} $ and $
\pi_{t} $, which establish the link between them and the Lyapunov exponent $
\lambda(u_{0}) $.

\subsection{Formula for the top Lyapunov exponent}\label{sec:fk-derivation}
Control over the projective process opens the road to a detailed analysis of
Lyapunov exponents. In this section we cover but one aspect in which control
over the projective dynamic is helpful, by formally establishing
Furstenberg--Khasminskii type formulas for the Lyapunov exponent.
We start by observing that we can formally write an SPDE for $ \pi_{t} $. For this
purpose it is
useful to rewrite the equation for $ u_{t} $ in Stratonovich form:
\begin{equ}[eqn:main-strat]
  \ud u  = \mL  u \ud t - \frac{1}{2}  u \cdot \Tr(\Lambda) + u  \cdot\circ \ud W_{t}\;,
\end{equ}
where we have defined
\begin{equ}
( \mL u)^{\alpha} = - \nu^{\alpha} (-
\Delta)^{\a } u^{\alpha}\;, \qquad \Tr (\Lambda) \in
(\RR^{m})^{\otimes 2}\;, \qquad \Tr (\Lambda)^{\alpha}_{\beta} =
\sum_{\gamma = 1}^{m} \Lambda^{\gamma, \alpha}_{\beta, \gamma}(0) \;,
\end{equ}
and $ u \cdot \Tr(\Lambda) = \sum_{\beta} u^{\beta} \Tr
(\Lambda)^{\alpha}_{\beta} $.
Regarding the radial process, setting $$ Q_{\mL}(u,u) = \Big\langle
u , \mL u  - \frac{1}{2} u \cdot \Tr(\Lambda) \Big\rangle \;,$$ a simple application of the chain rule yields
\begin{equ}[e:r]
\ud r =  r Q_{\mL}(\pi, \pi) \ud t +
r \langle \pi,  \pi \cdot \circ \ud W \rangle  \;.
\end{equ}
This in turn leads to an expression for the projective dynamic:
\begin{equ}[e:pi]
\ud \pi  =   \Big( \mL \pi - \frac{1}{2} \pi \cdot \Tr(\Lambda)-
Q_{\mL}(\pi, \pi) \pi\Big) \ud t +  \pi \cdot  \circ \ud W - \langle
\pi,  \pi \cdot \circ \ud W \rangle \pi  \;.
\end{equ}
We observe that if $ m =1 $, then the equation for $ \pi $ does not depend on
$ \Lambda $ (other than through the noise), as should be expected since $\Tr(\Lambda) $ is scalar in this case.

These calculations allow us to derive a heuristic
formula for the top Lyapunov exponent.  If we write the equation for $  \log( r_{t} ) $ in Itô form
\begin{equs}
\ud \log(  r ) & =  Q_{\mL}(\pi,
\pi) \ud t +
\langle \pi, \pi \cdot \circ \ud W \rangle\\
&  =  \langle \pi , \mL \pi \rangle + \frac{1}{2} C(\pi, \Lambda)  \ud t +
\langle \pi, \pi \cdot \ud W \rangle \;,\qquad \label{eqn:differential-fk}
\end{equs}
where the It\^o--Stratonovich corrector can be obtained via \eqref{e:pi} and
is given by
\begin{equ}
C(\pi, \Lambda)  = \langle \pi, \pi \cdot
\Tr^{\mathrm{u}} (\Lambda) \rangle -2 \langle \pi^{\otimes 2}, \Lambda
\pi^{\otimes 2} \rangle \;. 
\end{equ}
Here we have defined 
\begin{equs}
\Tr^{\mathrm{u}} (\Lambda)^{\alpha}_{\beta} & =
\sum_{\gamma=1}^m \Lambda^{\gamma, \gamma}_{\alpha, \beta}(0) \;,\\
 \langle \pi^{\otimes 2}, \Lambda \pi^{\otimes 2} \rangle & =  \sum_{\alpha, \beta,
\gamma, \eta}\int_{(\TT^{d})^{2}} \pi^{\alpha}(x)\pi^{\gamma}(x)
\Lambda^{\alpha, \beta}_{\gamma, \eta}(x -y)
\pi^{\beta}(y)\pi^{\eta} (y)  \ud x \ud y \;. \label{e:quartic}
\end{equs}
For the sake of clarity, let us observe that if $ m=1 $ and if the noise is
space-independent, that is $\Lambda (x) \equiv \Lambda(0) $, then the
corrector reduces to $ C(\pi, \Lambda) =
- \Lambda^{2} \in (- \infty, 0) $. In other words, the corrector we obtain is
an infinite-dimensional and vector-valued generalisation of the corrector
between It\^o and Stratonovich versions of geometric Brownian motion, the first one with
Lyapunov exponent $ - 1/2 $, the second one with Lyapunov exponent zero.

Indeed, we see that in the limit $t\to \infty$, if we consider $ \frac{1}{t} \log{r_{t}} $, the
martingale term disappears, being roughly of order $ \sqrt{t} $. Assuming
that $ [\pi_{t}] $ is uniquely ergodic, and since the drift above is quadratic
in $ \pi $ and therefore depends only
on the projective class of $ \pi $, this calculation suggests the identity
\begin{equ}[eqn:fk]
  \lambda =  \EE_{\infty} \bigg[ \langle \pi, \mL \pi \rangle + \frac{1}{2}
C(\pi, \Lambda )\bigg] \in \RR \cup \{ - \infty \} \;,
\end{equ}
where $ \EE_{\infty} $ stands for expectation under the stationary law of $
[\pi_{t}] $. One of the aims of this article is to rigorously
derive \eqref{eqn:fk}, by obtaining suitable conditions for the existence of a
unique invariant measure for the projective process. Another
aim is to prove that $ \lambda > - \infty $ by obtaining suitable regularity
estimates.

\subsection{The Lyapunov functional and the energy median}\label{sec:lyapunov}

The fundamental tool to achieve our objective is to introduce a Lyapunov
functional for the projective process. Its construction rests on a large- and
small- scale separation, which is achieved by separating frequencies according to
whether they lie above or below the energy median, which we will shortly
define. In the vector-valued setting, if the viscosity coefficients $
\nu^{\alpha}$ are not identical, we must adjust low or high
frequency projections so as to match the dissipation rate. Namely, for given $
L >0$ we project on modes that dissipate at, above, or below level $
\zeta_{L} = - L^{2\a } $. Since at frequency $ k $ and in the entry $
\alpha $ the dissipation rate is given by $ \nu^{\alpha} | k |^{\mathbf{2 a}} $, the
threshold for the projection for the $ \alpha $-th entry will be $$ | k
| = \frac{L}{ (\nu^{\alpha})^{\frac{1}{2\a }}}
\eqdef L_{\alpha}\;.$$ In other words, rather
than projecting on frequencies less than a given frequency level, we project on eigenspaces with
eigenvalues less than a fixed threshold.

\begin{definition}\label{def:proj}
For every $ L \in \NN $ and $ \varphi \in L^{2}(\TT^{d}; \RR^{m}) $ we define respectively the 
low- and high-frequency projections, written in components for $ \alpha \in \{ 1, \dots , m \} $
\begin{equ}
\gls{low}\varphi^{\alpha}  = \sum_{| k | \leqslant L_{\alpha}} \langle
\varphi^{\alpha},
e_{k} \rangle e_{k}\;, \qquad  \gls{high}\varphi^{\alpha} = \sum_{| k | >
(L+1)_{\alpha}} \langle
\varphi^{\alpha}, e_{k} \rangle e_{k}\;,
\end{equ}
as well as the central frequency projection
\begin{equ}
\gls{central} \varphi^{\alpha} = \sum_{ L_{\alpha} < | k | \leqslant
(L+1)_{\alpha}} \langle \varphi^{\alpha}, e_{k} \rangle e_{k}\;.
\end{equ}
In addition we consider
\begin{equ}
\Pi_{L}^{\myleq} = \Pi_{L}^{\myl} + \Pi_{L}^{\myc} = \Pi_{L+1}^{\myl}\;, \qquad
\Pi_{L}^{\mygeq} = \Pi_{L}^{\myc} + \Pi_{L}^{\myg} = \Pi_{L-1}^{\myg} \;.
\end{equ}
Finally, we define the energy median $M \colon L^{2}(\TT^{d}; \RR^{m}) \to \NN$ by
\begin{equ}
\gls{energy median} = \inf \{ M \in \NN \;, M \geqslant 1   \ :\ \|
\Pi_{M}^{\mygeq} u \| \leqslant \| \Pi_{M}^{\myl} u\|  \}\;. 
\end{equ}
\end{definition}

\begin{remark}\label{rem:dimension}
  In one dimension, there is no eigenvalue in the interval $
  (L, L+1) $. In this case, at least if all $ \nu^{\alpha} $ are identical, some of the upcoming technical steps do
simplify significantly, while in higher dimensions we must treat separately modes in between two
separated shells (in the picture below, for $ d =2 $ and $ L = 5 $, the modes that fall in
the gray area). When we consider separated shells we can make
use of gaps between eigenvalues of $\mL $, which is of great advantage. This
motivates the distinction (superfluous in one dimension) in
Section~\ref{sec:dissip} between $ w^{\myg}$ and  $ w^{\mygeq}$.
  \end{remark}
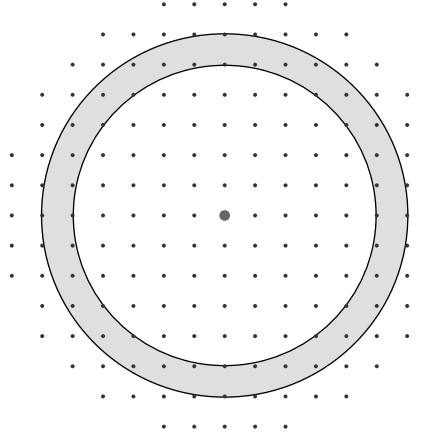
\begin{center}
\begin{tikzpicture}
\draw[line width = 1pt] (0,0) circle (5*0.40);
\draw[line width = 1pt] (0,0) circle (6*0.40);
\fill[gray!25, even odd rule] (0,0) circle (5*0.40) (0,0) circle (6*0.40);
\foreach \x in {-7,-6,...,7}
	\foreach \y in {-7,-6,...,7}
		{
		\pgfmathtruncatemacro{\z}{(0.4*\x)^2+(0.4*\y)^2}
		\ifthenelse{\z<9}{\fill[black!80] (0.40*\x, 0.40*\y) circle (0.75pt);}{} 
		}
\fill[black!60] (0,0) circle (2pt);
\end{tikzpicture}
\end{center}
The second ingredient in the construction of the Lyapunov functional is a
control on high frequency regularity. For this reason we introduce for $ L \in \NN $ the shifted Sobolev
spaces $ H^{\gamma}_{L} $ (these are Banach spaces only if considered as subsets of the functions
$ \varphi \in L^{2} $ such that $ \Pi_{L}^{\myg} \varphi = \varphi $) defined by the seminorm 
\begin{equ}[eqn:shifted-sobolev]
\| \varphi \|_{H^{\gamma}_{L}}^2 = \sum_{\alpha =1}^{m}\sum_{| k | >
L_{\alpha} +1} (1+| k |- L_{\alpha})^{2 \gamma} |
\hat{\varphi}^{\alpha} (k) |^{2}  \;,
\end{equ}
and to simplify the notation we write
\begin{equ}
 \| \varphi \|_{\gamma, L} = \| \varphi \|_{H^{\gamma}_{L}}  \;, \qquad 
 \| \varphi \|_{L} = \| \varphi \|_{\frac{1}{2}, L}  \;.
\end{equ}
\begin{remark}\label{rem:def-norm}
  The choice of the regularity parameter equal to $ 1/2 $ simplifies
  certain computations since it is ``linear'' in $ L $:
  \begin{equ}
    \| \varphi \|^{2}_{\frac{1}{2} , L} = \| \Pi_{L}^{\myg} \varphi
    \|^{2}_{H^{\frac{1}{2}}} - \sum_{\alpha=1}^{m} L_{\alpha} \|
\Pi_{L}^{\myg} \varphi^{\alpha} \|^{2}\;.
  \end{equ}
In particular, in this setting we observe that we can bound for any $ k_{0} \in \NN $ 
\begin{equ}
  \| \pi_{t} \|_{H^{\frac{1}{2}}}^{2} \leqslant 2
\nu_{\mathrm{min}}^{-\frac{1}{2 \a }}  ( M(\pi_{t}) + k_{0}) +1 + \|
  \pi_{t} \|_{M(\pi_{t})+ k_{0}}^{2} \;,
\end{equ}
with $ \nu_{\mathrm{min}}= \min_{\alpha} \nu^{\alpha} $.
\end{remark}
Now, a natural first candidate for a Lyapunov functional to $ t \mapsto
\pi_{t}  $ could be the map
\begin{equ}[e:naivelyap]
S \ni \pi \mapsto \exp \big( \| \pi \|_{H^{\gamma}}^{2} \big) \;,
\end{equ}
for some $ \gamma > 0 $. This is a reasonable choice for a Lyapunov functional,
but our analysis is not sufficient to prove that it does indeed satisfy
the Lyapunov property.
As we will see, a
crucial point of our argument is to control
the evolution of level sets of the energy of $ \pi_{t} $, such as the energy
median, which we will use as
thresholds to distinguish between large and small scales.
Therefore we will replace our first guess for the Lyapunov functional by 
\begin{equ}[eqn:def-G-lyap]
 \mf{G} \colon S \to [0, \infty] \;, \quad \mf{G} (\pi) = \exp \left( \kappa_{0} M (\pi) + \| \pi
  \|_{M(\pi) +k_{0}}^{2} \right)\;,
\end{equ}
for two parameters $ \kappa_{0}>0, k_{0} \in \NN $ to be fixed later on. 
For fixed $ \kappa_{0}, k_{0} $, there exists constants $
c_{1} (\kappa_{0}, k_{0}) < c_{2} (\kappa_{0}, k_{0}) $ such that for every $
\pi \in S $:
\begin{equ}
c_{1} \| \pi \|_{H^{1/2}}^{2} \leqslant  \kappa_{0} M (\pi) + \| \pi
  \|_{M(\pi) +k_{0}}^{2} \leqslant c_{2} \| \pi \|^{ 2}_{H^{1/2}} \;.
\end{equ}
In particular, if we could prove a \emph{super-Lyapunov} property for the
functional $ \mf{G} $, namely that
\begin{equs}
\EE[ \mf{G}( \pi_{t}) ] \leqslant c \exp \left( c e^{- \lambda t} \kappa_{0} M
(\pi_{0}) + c e^{- \lambda t} \| \pi_{0}
  \|_{M(\pi_{0}) +k_{0}}^{2} \right) \;,
\end{equs}
for some $ c , \lambda > 0 $, 
then we would be able to deduce that also the map in \eqref{e:naivelyap}
satisfies the Lyapunov property. Unfortunately, proving such super-Lyapunov
property lies beyond our capacities and we restrict to proving the Lyapunov
property only. The main issue lies in the analysis of the median
process $ M( \pi_{t}) $: for the high-frequency regularity $ \| \pi_{t} \|_{M (\pi_{t}) +
k_{0}} $, which behaves very much in analogy to the Sobolev norm of a typical
parabolic SPDE, we can prove a bound very crudely of the type $ \EE[ \exp \left(
 \| \pi_{t} \|_{M (\pi_{t}) +
k_{0}}^{2} \right) ] \leqslant  c \exp \left( c e^{- \lambda k_{0} t} \|
\pi_{0} \|_{M (\pi_{0}) + k_{0}}^{2}\right)  $. For the median however we can
only prove a bound along the lines of $ \EE [ \exp( \kappa_{0} M (\pi_{t}))]
\leqslant c( e^{- \lambda \kappa_{0} t }\exp( \kappa_{0} M (\pi_{0})) +1) $.
Namely, we show that $ M $ is upper bounded, at least for large values,
by a diffusion with a constant drift towards low frequencies. It is this upper
bound \dash which we believe to be sub-optimal but whose improvement would require a much
deeper analysis of energy level dynamics \dash which leads to the restriction
mentioned above. This discussion also
explains the presence of the parameters
$ \kappa_{0}, k_{0} $, which are used to increase to our convenience the
contraction constants $ e^{- \lambda \kappa_{0} t} $ and $ e^{- \lambda k_{0}
t} $.

Before we move on, let us remark that sometimes we will 
call $ \mf{G} $ the \emph{full} Lyapunov functional, in opposition to a
\emph{skeleton} functional $ \mf{F} $ which we construct in
Section~\ref{sec:skeleton} and which is at the
heart of our technical analysis. Now we are ready to state the main
achievements of this paper.

\subsection{Main results}

Our main results neatly divide in two distinct theorems, with respectively
weaker and stronger non-degeneracy assumptions on the noise: one concerning the
existence of invariant measures and lower bounds for the Lyapunov
exponent(s), and one regarding the uniqueness of invariant measures and
Furstenberg--Khasminskii formulas for the Lyapunov exponent. We start with the
first result.

\subsubsection{Existence of invariant measures}

Our proof of existence of invariant measures amounts to proving that $
\mf{G} $ in \eqref{eqn:def-G-lyap} is a Lyapunov functional for $
\pi_{t} $. This is not true for any noise, but requires a certain
non-degeneracy. The sufficient property we identify for the desired Lyapunov property to hold
guarantees that the noise
counteracts concentration in high-frequency states: if a substantial amount of
the energy of $ \pi_{t} $ finds itself in a frequency shell at level $ M
$, so that say $ \| (\Pi^{\myc}_{M-2} + \Pi^{\myc}_{M-1}) \pi_{t} \| \geqslant \beta_{0} $ for some
$ \beta_{0} > 0 $ with no control on how much energy lies in lower frequencies,
then the noise must nevertheless shift some energy to lower frequencies arbitrarily quickly,
provided that $ M $ is sufficiently large. The fact that we consider two shells
($ M -2 $ and $ M-1 $) is out of purely technical reasons, for later
convenience. In this case, we say that $ W
$ induces \emph{high-frequency stochastic instability}.

\begin{definition}\label{def:high-freq-insta}
The noise $ W $ appearing in \eqref{eqn:main} is said to induce
high-frequency stochastic instability if 
for every $ \mf{t} >0 $ and $ \ve \in (0, 1) $  there exists a $ \K (\mf{t}, \ve) \in \NN $ such that for all
$ M \geqslant \K$ and all $ u_{0} \in L^{2}_{\star} $, if
\begin{equ}[e:idata]
\| \Pi^{\mygeq}_{M} u_{0} \| \leqslant 2 \| \Pi^{\myl}_{M} u_{0} \|   \;,
\qquad \| (\Pi^{\myc}_{M-1} + \Pi^{\myc}_{M-2}) u_{0} \| \geqslant \frac{1}{4} \|
\Pi^{\myl}_{M-2} u_{0} \| \;,
\end{equ}
then
\begin{equs}
\PP \left( \| (\Pi^{\myc}_{M-1} + \Pi^{\myc}_{M-2}) u_{t} \| < \frac{1}{4} \| \Pi^{\myl}_{M-2} u_{t}
\| \text{ for some } t \in [0, \mf{t}]\right) \geqslant 1- \ve \;.
\end{equs}
\end{definition}
The exact value of the constants $ 2 $ and $ 1/4 $ is irrelevant but chosen in
harmony with later thresholds. The assumption \eqref{e:idata} on the initial data is stating
that $ M $ is related to the energy median of $ u_{0}$ (it bounds another energy level
set) and that the energy of $ \pi_{0} = u_{0} / \| u_{0} \| $ is concentrating
in a shell of level $ M $. The claim is then that the noise shifts energy to lower
frequencies \dash no matter how small $ \| \Pi^{\myl}_{M-2} u_{0} \| $ might be.
It is easy to see that if there is no noise ($  W = 0$) the condition is not satisfied,
which is why this is a non-degeneracy assumption on the noise.
Our main result can then be formulated in terms of this stochastic instability.
Here we make use of the Fourier coefficients appearing in \eqref{e:noisefc}.

\begin{theorem}\label{thm:lyap-func}
  Assume that $ \a  \geqslant 1 $, $ \nu^{\alpha} > 0 $ for all $ \alpha
\in \{ 1, \dots, m \} $, that the noise $ W $ appearing in \eqref{eqn:main} induces a
high-frequency stochastic instability, and that the noise coefficients $ \Gamma
$ in \eqref{e:noisefc} satisfy for every $ \alpha, \alpha^{\prime} , \beta, \beta^{\prime}  \in \{ 1, \dots, m \} $
\begin{equ}[e:reg-assu]
\exists \gamma_{0} > \frac{d}{2} +1 \;, C > 0  \quad \text{ such that } \quad |
\Gamma^{\alpha, \alpha^{\prime}}_{\beta, \beta^{\prime}, k}  | \leqslant C( 1+ | k
|)^{- 2 \gamma_{0}} \;, \quad \forall k \in \ZZ^{d} \;. 
\end{equ}
Then the following hold
\begin{enumerate}
\item For every $ u_{0} \in L^{2}_{\star} (\TT^{d}; \RR^{m}) $, almost surely
the solution $ u_{t} $ to \eqref{eqn:main} satisfies $ u_{t} \neq 0 $ for all
$ t \geqslant 0 $. The angular component $ \pi_{t} = u_{t}/ \|
u_{t}\| $ is hence defined for any $ \pi_{0} \in S$, almost surely, for all $t
\geqslant 0 $, and it is a Markov process.
\item For any $ \mf{c} \in (0, 1) $ there exist
  $ \kappa_{0}, k_{0}, J, t_{\star} >0 $
  such that for $ \mf{G} $ as in \eqref{eqn:def-G-lyap} and all $
\pi_{0} \in S $
\begin{equ}
  \EE_{t} \left[ \mf{G}(\pi_{t + t_{\star}}) \right] \leqslant
  \mf{c} \cdot \mf{G}(\pi_{t})  + J \;.
\end{equ}
\end{enumerate}
\end{theorem}
Before we move on, let us comment on the assumptions of this theorem.
\begin{remark}\label{rem:assuthm1}
The assumption $ \a \geqslant 1 $ is slightly unnatural: the natural threshold
for our analysis would be $ \a = 1/2 $, since it is at this point that gaps $
\Delta_{L} $ of the (1D) Laplacian are no longer increasing in $ L $. In our
analysis, the restriction $ \a \geqslant 1 $ appears in the proof of
Proposition~\ref{prop:higher-regularity} when treating systems with different
viscosity coefficients $ (\nu^{\alpha})_{\alpha=1}^{m} $ (if all $
\nu^{\alpha} $ were identical, we would be able to treat all $ \a > 1/2 $).

The requirement $ \gamma_{0} > d/2 +1$ appearing in \eqref{e:reg-assu} is
instead a mild and mostly
technical regularity assumption. It is needed for example to obtain the energy estimates in
Proposition~\ref{prop:reg-high-freq}.
\end{remark}
The proof of the Lyapunov property is the content of
Section~\ref{sec:prf-main-result}, the existence and Markov property of the angular component is
the content of Section~\ref{sec:basic}, see Lemma~\ref{lem:nonzero}.
We observe that (see Remark~\ref{rem:def-norm}) for some
constants $ c_{1}, c_{2} > 0 $ the functional $ \mf{G} $ satisfies
\begin{equ}
  \mf{G}(\pi) \geqslant c_{1} \exp \big( c_{2} \| \pi \|^{2}_{H^{1/2}} \big) \;,
\end{equ}
which suffices to see that the next result implies tightness of the process $
\pi_{t} $ in $ S $. 

Of course, the notion of high-frequency stochastic instability is not very
practical, so it is desirable to have some easy-to-check conditions implying it. As it
turns out, there is a bonanza of fairly mild non-degeneracy conditions that can be imposed on the noise
coefficients $ \Gamma $ to enforce this stochastic instability.
One possible condition which is very weak yet easy to state is the following.

\begin{assumption}\label{assu:noise}
Assume that the noise coefficients $
\Gamma^{\alpha, \alpha^{\prime}}_{\beta, \beta^{\prime}, k} $ satisfy
\eqref{e:reg-assu} and are diagonal:
\begin{equ}
\Gamma^{\alpha, \alpha^{\prime}}_{\beta, \beta^{\prime}, k} =
\Gamma^{\alpha}_{\beta, k} \delta_{\alpha, \alpha^{\prime}} \delta_{\beta,
\beta^{\prime}} \;,
\end{equ}
for some $ \{ \Gamma^{\alpha}_{\beta, k} \}_{\alpha, \beta, k} \subseteq [0,
\infty) $ (with a slight abuse of the letter $ \Gamma $).
Further, assume that
for every $ \beta \in \{ 1, \dots, m \} $ there exists an $ \alpha (\beta) \in \{ 1, \dots,
m \} $ such that the viscosity coefficients satisfy $ \nu^{\beta} \geqslant
\nu^{\alpha} $ and the noise coefficients satisfy that there exists a finite set
$ \mA $ and a $ \K_{0} \in \NN $ such that
$$ \mA \subseteq \supp (\Gamma^{\alpha}_{\beta}) = \{ k  \; \colon \;
\Gamma^{\alpha}_{\beta, k} > 0 \} \;,$$
and for which
\begin{equation}\label{e:lsnd}
\begin{aligned}
\one_{\mA} * \one_{B(M)} \geqslant \one_{B( M+b )} \;, \qquad \forall M \geqslant
\K_{0} \;,
\end{aligned}
\end{equation}
where  $ B(M) = \{ k \in \ZZ^{d}  \; \colon \; | k
| \leqslant M \} $ and $b=3 \nu_{\mathrm{min}}^{- 1/ 2
\a  }$.
\end{assumption}
Here, in the large scale non-degeneracy assumption we have written $
\one_{A} $ for the characteristic function of a subset $ A \subseteq
\ZZ^{d} $ and we used the discrete convolution $ (f * g) (k) =
\sum_{l \in \ZZ^{d}} f (k-l) g(l) $. Before we proceed, let us conclude
with some remarks on the setting above.

\begin{remark}\label{rem:a-grt-1}
We observe that there are many possible choices for the coefficients $
\Gamma^{\alpha}_{\beta,k} $ for which \eqref{e:lsnd} is satisfied. 
\begin{enumerate}
\item For example, by Lemma~\ref{lem:geom}, if $ \nu^{\alpha} $ and $ \Gamma^{\alpha}_{\beta, k} $
are such that for any $ \beta \in \{ 1, \dots, m \} $ there exists an $
\alpha(\beta) $ with $ \nu^{\beta} \geqslant \nu^{\alpha} $ and  
\begin{equation*}
\begin{aligned}
 \Gamma^{\alpha}_{\beta, k}  > 0
\;, \qquad \forall k \in \ZZ^{d}  \; \colon \;  | k | \leqslant  
\eta(d , \nu) \;,
\end{aligned}
\end{equation*}
where $ \eta(d, \nu) $ is an arbitrary constant such that $ \eta(d , \nu) >  3
\nu_{\mathrm{min}}^{- 1/2 \a} \sqrt{d}  $ is satisfied, then
\eqref{e:lsnd} is satisfied.
\item The condition $ \one_{\mA} * \one_{B(M)} \geqslant \one_{B (M+b)}  $ could be
relaxed to $ \one_{\mA} * \one_{B(M)} \geqslant \one_{B(M+ \ve)} $, for arbitrary $
\ve > 0 $ by considering shells of smaller width throughout the work: we work
with shells of width one only to lighten the burden of notation. For the very same
reason, we also expect that the factor $ 3
\nu_{\mathrm{min}}^{- 1/2 \a} \sqrt{d} $ appearing in the previous point can be
replaced by $ 1 $ in any dimension.
\end{enumerate}
\end{remark}
As already anticipated, this is sufficient to deduce high-frequency stochastic
instability for the noise $ W $. 
\begin{proposition}\label{prop:hf-insta}
Under Assumption~\ref{assu:noise} the noise $ W $ induces high-frequency
stochastic instability, in the sense of Definition~\ref{def:high-freq-insta}.
\end{proposition}
The proof of this proposition can be found in Section~\ref{sec:insta}.
One can observe that at least heuristically, in view of \eqref{eqn:fk} the
functional \( \mf{G} \) alone is not sufficient to obtain the bound $ \lambda (u_{0}) > - \infty
$, as this result requires an a priori estimate on $ \EE[ \| \pi_{t}
\|_{H^{\a }}^{2}] $ (as opposed to the $ H^{1/2} $ bound provided by $
  \mf{G} $), which we obtain separately by a bootstrap argument.

\begin{corollary}\label{cor:finite-lyap}
  Under the same assumptions as in Theorem~\ref{thm:lyap-func}, there exist $
s_{\star}$ and $C > 0 $ such that uniformly over all $ \pi_{0} \in S $ 
\begin{equ}
  \sup_{t \geqslant s_{\star}}\EE \left[ \| \pi_{t}
\|^{2}_{H^{\a }}  \right] <
  C \mf{G}(\pi_{0})\;,
\end{equ}
Hence,  in particular $\inf_{u_{0} \in L^{2}_{\star}} \lambda (u_{0}) > - \infty$.
\end{corollary}

\begin{proof}
  The uniform estimate on the $ H^{\a } $ norm follows from
  Proposition~\ref{prop:higher-regularity}. As for the last claim, let us prove
  a uniform lower bound over $ u_{0} \in H^{1/2} \setminus \{ 0 \}$, which is sufficient since for every
  $ u_{0} \in L^{2}_{\star} $ we have that \( u_{t} \in H^{1/2} \) for any $ t > 0
  $. By \eqref{eqn:differential-fk} we find that $\PP$-almost surely
\begin{equs}
  \lambda (u_{0}) = \lim_{t \to \infty} \frac{1}{t} \bigg(
      & \int_{s_{\star}}^{t}  \langle \pi_{s}, \mL \pi_{s} \rangle  +
\frac{1}{2} C(\pi_{s}, \Lambda)  \ud s 
+ \int_{s_{\star}}^{t} \langle \pi_{s}, \pi_{s} \cdot \ud W_{s} \rangle  \bigg) \;.
\end{equs}
We observe that the quadratic variation of the last term is given by
\begin{equation*}
\begin{aligned}
\int_{s_{\star}}^{t} \langle \pi^{\otimes 2}_{s}, \Lambda
\pi^{\otimes 2}_{s} \rangle  \ud s \;,
\end{aligned}
\end{equation*}
as defined in \eqref{e:quartic}. Further we can estimate, using the assumption
of Theorem~\ref{thm:lyap-func}
\begin{equation*}
\begin{aligned}
|\Lambda^{\alpha, \alpha^{\prime}}_{\beta, \beta^{\prime}} (x) | & =
\left\vert \sum_{k} \Gamma^{\alpha, \alpha^{\prime}}_{\beta, \beta^{\prime} ,
k} e_{k}(x-y) \right\vert 
 \leqslant  \sum_{k} | \Gamma^{\alpha, \alpha^{\prime}}_{\beta, \beta^{\prime} ,
k} | \lesssim \sum_{k} (1 + | k |)^{- d -2} < \infty \;,
\end{aligned}
\end{equation*}
meaning that $ \| \Lambda \|_{\infty} \lesssim 1 $. Therefore, the
quadratic variation is bounded by
\begin{equation*}
\begin{aligned}
\int_{s_{\star}}^{t} \langle \pi^{\otimes 2}_{s}, \Lambda
\pi^{\otimes 2}_{s} \rangle  \ud s \leqslant \| \Lambda
\|_{\infty} t \;,
\end{aligned}
\end{equation*}
since $ \pi_{s} \in S $, so that
from the law of the iterated logarithm we find that $\PP$-almost surely
  \begin{equ}
    \lim_{t \to \infty} \frac{1}{t} \int_{s_{\star}}^{t} \langle
\pi_{s}, \pi_{s} \cdot \ud W_{s} \rangle  =0 \;.
  \end{equ}
  Hence the Lyapunov exponent is bounded by
  \begin{equs}
    \lambda (u_{0}) & = \lim_{t \to \infty} \frac{1}{t} 
       \int_{s_{\star}}^{t} \langle \pi_{s}, \mL \pi_{s} \rangle   + \frac{1}{2} C(\pi_{s}, \Gamma)  \ud s  \\
    & \geqslant \EE \bigg( \limsup_{t \to \infty} \frac{1}{t}
\int_{s_{\star}}^{t} \langle \pi_{s}, \mL \pi_{s} \rangle \ud s \bigg) - \|
{\Tr^{\mathrm{u}}(\Lambda)} \| - \| \Lambda \|_{\infty} \;,
    \label{eqn:lwr-bd-lyap}
  \end{equs}
  since the left-hand side is deterministic \dash take the expected value of
the first line and replace the limit by a limsup.

  Since $ \langle \pi_{s}, \mL \pi_{s} \rangle $ is a negative random variable, Fatou's lemma yields
  \begin{equs}
  \EE \left[ \limsup_{t\to \infty} \frac{1}{t} \int_{s_{\star}}^{t} \langle \pi_{s}, \mL \pi_{s} \rangle\ud s \right]  & = - \EE \left[ \liminf_{t\to \infty} \frac{1}{t} \int_{s_{\star}}^{t}  -\langle \pi_{s}, \mL \pi_{s} \rangle\ud s \right]  \\
    & \geqslant -  \liminf_{t\to \infty} \frac{1}{t} \int_{s_{\star}}^{t}  - \EE [Q_{\mL}(\pi_{s},
    \pi_{s}) ] \ud s  \\
    & \geqslant -  \left( \max_{\alpha} \nu^{\alpha} \right) \liminf_{t\to \infty} \frac{1}{t}
\int_{s_{\star}}^{t}  \EE \left[ \| \pi_{s} \|^{2}_{H^{\a }} \right] \ud s\;.
  \end{equs}
  The first statement of our result furthermore guarantees that for every $ n \in
\NN^{+} $ we have the upper bound
  \begin{equs}
     \liminf_{t \to \infty} \frac{1}{t} \int_{n s_{\star}}^{t} \EE \|
    \pi_{s} \|^{2}_{H^{\a }} \ud s & \leqslant C \EE\mf{G}(\pi_{(n-1) s_{\star}})
     \leqslant C \left( \mf{c}^{n-1} \mf{G}(\pi_{0}) + J (\sum_{i =
    0}^{n-2} \mf{c}^{i})  \right) \;,
  \end{equs}
  where we used the uniform bound on $ \EE \| \pi_{s}
  \|_{H^{\a }}^{2} $ and the contraction property for
  $ \mf{G} $ from Theorem~\ref{thm:lyap-func}. Passing to the limit $ n \to \infty $ we deduce that
  \begin{equ}
    \liminf_{t \to \infty} \frac{1}{t} \int_{n s_{\star}}^{t} \EE \|
    \pi_{s} \|^{2}_{H^{\a }} \ud s  \leqslant C J \frac{1}{1 - \mf{c}} \;,
  \end{equ}
  yielding the desired uniform bound.
\end{proof}

\subsubsection{Uniqueness of invariant measures}
We can next exploit our control on $ \pi_{t} $ to derive 
unique ergodicity for the projective process $ [\pi_{t}] $: here we require some strong
non-degeneracy assumptions on the noise, and a condition on the hyperviscosity
parameter and the dimension.

\begin{assumption}\label{assu:non-deg-uniq}
  The parameter $ \a $ in \eqref{eqn:main} satisfies
  $\a  \in [1,\infty)\cap (d/2,\infty)$ and the noise coefficients $
\Gamma^{\alpha, \alpha^{\prime}}_{\beta, \beta^{\prime}, k} $ satisfy
\eqref{e:reg-assu} and are diagonal:
\begin{equ}
\Gamma^{\alpha, \alpha^{\prime}}_{\beta, \beta^{\prime}, k} =
\Gamma^{\alpha}_{\beta, k} \one_{\{ \alpha = \alpha^{\prime} \}} \one_{\{ \beta =
\beta^{\prime}\}} \;,
\end{equ}
for some $ \{ \Gamma^{\alpha}_{\beta, k} \}_{\alpha, \beta, k} \subseteq [0,
\infty) $. Furthermore, there exist
 constants $ 0 < c < C < \infty $ and $ \gamma_{0} > d/2+1 $ such that for all $
\alpha, \beta  \in \{ 1, \dots, m \} $: 
\begin{equ}[e:boundalphak]
    c (1 + | k |)^{- 2 \gamma_{0}} \leqslant  \Gamma^{\alpha}_{\beta, k}  \leqslant C (1 + | k
    |)^{ - 2 \gamma_{0}}\;, \qquad \forall k \in \ZZ^{d} \;.
  \end{equ}
\end{assumption}

\begin{remark}\label{rem:on-assumption}
The condition above, and in particular the requirement $ \a  >d/2 $, 
 is very restrictive and we believe it is far
from optimal.
It is
      used to guarantee that the Jacobian of the solution $ u_{t} $ takes
values in the image of the noise operator, which
allows to establish the strong Feller property via
      the Bismut--Elworthy--Li formula and a localisation argument. This could
      be avoided by using different approaches, such as asymptotic strong
      Feller \cite{HairerMattingly06} or asymptotic couplings
\cite{AsCoupling}, but extending our approach to make use of such techniques
lies beyond the scope of this paper. 
Similarly, the matching lower and upper
bounds in \eqref{e:boundalphak} are imposed to us by our strategy of proof.
\end{remark}

Under Assumption~\ref{assu:non-deg-uniq}, which implies
Assumption~\ref{assu:noise}, we are able to derive a spectral gap
for $ [\pi_{t}] $ as stated in Theorem~\ref{thm:uniq} below. This implies uniqueness of the Lyapunov
exponent over all initial conditions, along with a Furstenberg--Khasminskii
type formula \cite{Furstenberg,MR600653} for it and continuous
dependence on the parameters of the model: in our setting, we chose for simplicity
only the parameter $ \a  $ appearing as a power in the Laplacian,
although one could as well choose the noise strength coefficients $ \Gamma $.

In the statement of the following theorem we denote by $ \mP_{t} ([\pi_{0}], \cdot) $
the law of $ [\pi_{t}] $ started in $ [\pi_{0}] $ (the functional $
\mf{G}(\cdot) $ does not depend on the choice of representative for $ [\pi_{0}] $).
Moreover, for a measurable space $ (\mX, \mF) $, we
denote with $ \| \cdot \|_{\mathrm{TV}, \mX} $ the total variation norm of a
signed measure \(\mu\) over $ \mX $ (scaled by a factor $ 1/2 $ for later
convenience) by
\begin{equ}[e:tv]
 \| \mu \|_{\mathrm{TV}, \mX} =  \frac{1}{2} | \mu|(\mX)  \;.
\end{equ}

\begin{theorem}\label{thm:uniq}
  Under Assumption~\ref{assu:non-deg-uniq}
  there exists a unique invariant measure $ \mu_{\infty} $ for $ ([\pi_{t}])_{t \geqslant 0} $ on $
  \mathbf{P} $. In addition, there
  exist $ C , \gamma > 0 $ such that for $ \mf{G} $ as in
  Theorem~\ref{thm:lyap-func} and any $ \pi_{0} \in S \cap H^{1/2} $:
  \begin{equ}
    \| \mP_{t} ( [\pi_{0}] , \cdot) - \mu_{\infty} \|_{\mathrm{TV},
    \mathbf{P}} \leqslant C
    e^{- \gamma t} \mf{G}([\pi_{0}]) \;.
  \end{equ}
  In addition
  \begin{enumerate}
    \item For any initial condition $u_{0} \in L^{2}_{\star}(\TT^{d};
\RR^{m} )$,
      $ \PP$-almost surely
      \begin{equ}
        \lim_{t \to \infty} \frac{1}{t} \log{\|u_{t}\|} = \lambda\;,
      \end{equ}
      with $ \lambda \in \RR $ given by \eqref{eqn:fk}. 
    \item If we consider $ \a  \mapsto
  \lambda(\a ) $ as a function of the parameter $ \a  \geqslant 1 $ appearing in
  \eqref{eqn:main}, then the map $\lambda $ is
  continuous on the interval $ (d/2, \infty) \cap [1, \infty)$.
  \end{enumerate}
\end{theorem}

\begin{proof}
  The proof of the spectral gap is the content of Section~\ref{sec:uniqueness},
  and together with Theorem~\ref{thm:lyap-func} and
  Corollary~\ref{cor:finite-lyap} it immediately implies \eqref{eqn:fk}. As for the last statement, it
  follows from the observation that all the estimates of this work hold locally
  uniformly over $ \a  \in [1, \infty) $. In particular, for any compact interval
$ [\eta, \vt] =I \subseteq (d/2, \infty) \cap [1, \infty) $, provided that $ | \eta - \vt
| < 1/2 $, and for any $ n \in \NN $, we find by
Proposition~\ref{prop:higher-regularity} the uniform bound
\begin{equation}\label{e:unif-a}
\begin{aligned}
\sup_{\mathbf{ a} \in I} \EE_{\mu_{\a }} \left[ \| \pi \|^{n}_{H^{\vt}}  \right] < \infty
\;,
\end{aligned}
\end{equation}
where $ \mu_{\a } $ denotes the invariant measure for the process $ t \mapsto
\pi_{t} $ for a given parameter $ \a  > d/2  $. 
A first consequence of the uniform bound \eqref{e:unif-a} is that the
measures $ \mu_{\a } $ are continuous in $ \a  $ with respect to weak convergence
as probability measures on $ S $. Indeed \eqref{e:unif-a} implies that any
sequence of measures $ \{ \mu_{\a _{k}} \}_{k \in \NN } $ for  $ \{
\a _{k} \}_{k} \subseteq
I  $ with $ \a _{k} \to \a  \in I$ is tight and, writing $\mP_t^{\a}$ for the Markov semigroup
with given parameter $\a$, it is not hard to verify via Lemma~\ref{lem:cont-ic} that 
the map $(\a, \mu) \mapsto \mP_t^{\a}\mu$ is jointly continuous for any fixed
$t>0$. 
This implies that any weak limit point must be
invariant under the evolution of $ \pi_{t} $ associated to the limiting value of $
\a  $, and therefore equal $\mu_\a$. Finally, weak continuity of $ \mu_{\a } $ together with \eqref{e:unif-a} and
the Furstenberg--Khasminskii formula \eqref{eqn:fk} imply
that $ \a  \mapsto \lambda (\a ) $ is continuous.
\end{proof}

\section{The skeleton Lyapunov functional and dissipation estimates}
\label{sec:skeleton}

One of the issues that appear when dealing directly with the energy median $ (M
(\pi_{t}))_{t \geqslant 0} $ is that we have no control on the frequency of its
jump times. In fact, we expect $ M( \pi_{t}) $ to jump very rapidly to low
frequencies, provided it starts from
a sufficiently high frequency level. Such jumps are in principle
``good'' for us (since they imply that energy is not shifting to high frequencies), but every time the
energy median jumps to lower frequencies, the high-frequency regularity also jumps,
but upwards, so that it is unclear whether the Lyapunov functional is
decreasing. In addition, the energy median can be very irregular in
time since, due to the noise, it can easily accumulate
infinitely many jumps in a finite time interval while rapidly oscillating
between two neighbouring frequency levels.

To avoid all these issues, our strategy is to introduce a piecewise constant process $
(M_{t})_{t \geqslant 0} $, which we call the \emph{skeleton median process},
which behaves similarly to $ M( \pi_{t}) $, but is such that the time intervals between successive 
jumps are of order one.
For this, we will construct a suitable sequence of stopping times
$0 = T_{0} < T_1 < \dots$
with intervals of length of order one, meaning that $ T_{i +1} - T_{i} $ has
uniform bounds on both positive and negative moments. The skeleton energy median
is then a process that remains constant on every interval $ [T_{i}, T_{i +1}) $ and
updates its value only at the stopping times $ T_{i} $. For the sake of the
present discussion, one can for instance think of the stopping times as being deterministic
times
\begin{equation*}
\begin{aligned}
T_{i} \sim \eta \cdot i \;,
\end{aligned}
\end{equation*}
for some $ \eta > 0 $ and of the skeleton energy median to be defined as the
true median at times $ T_{i} $:
\begin{equation*}
\begin{aligned}
M_{t} \sim M (\pi_{T_{i}})\;, \qquad \forall t \in [T_{i}, T_{i+1}) \;.
\end{aligned}
\end{equation*}
Unfortunately, there are several reasons why this simple construction does not 
quite work for our analysis. First, our study of the median relies on bounds on
the dynamic of the relative energy processes which are addressed in
Proposition~\ref{prop:drft-OU}. These bounds do not hold over deterministic time
intervals, but only up to certain stopping times.  Second, the energy median may exhibit
large negative jumps over fixed time intervals, which is not desirable for the reason mentioned above, 
so we enforce the deterministic bound $ M_{T_{i +1}} - M_{T_{i}}
\geqslant -1 $, as well as the bound $ M_{T_{i}} \geqslant M (\pi_{T_{i}}) $ for all $ i \in \NN
$. 
Overall, although the stopping times $ T_{i} $ will not be deterministic,
we will have a deterministic upper bound on the increment $ T_{i+1} -
T_{i} $ as well as a bound on all inverse moments of such increments.

We refer to Lemma~\ref{lem:consistency} below for a
list of desirable properties satisfied by the skeleton median process.
Similarly, since the exact definitions of both the stopping times $ T_{i} $ and of the skeleton
process are quite intricate, they are deferred to Section~\ref{sec:st-def},
after we have introduced all the required tools.

Now we introduce another feature of our analysis, namely that we distinguish
between ``concentrated'' and ``diluted'' states of the process $ u_{t} $. Given a 
frequency level $ L \in \NN $
and a function $ u \in L^{2} $, we introduce the marker $
\mf{m}(L, u) \in \{ \tC, \tD \} $ given by
\begin{equ}[eqn:def-dil-n]
\mf{m}(L, u) = \begin{cases} 
\texttt{C} \quad (\text{\gls{con}}) \quad & \text{ if } \quad \|
(\Pi_{L}^{\myc} + \Pi_{L-1}^{\myc} ) u \| \geqslant \frac{1}{4}  \| \Pi_{L-1}^{\myl} u \|\;, \\
\texttt{D} \quad (\text{\gls{dil}}) \quad & \text{ if } \quad \|
(\Pi_{L}^{\myc} + \Pi_{L-1}^{\myc} )u \| < \frac{1}{4} \| \Pi_{L-1}^{\myl} u \|\;. \end{cases} 
\end{equ}
The value $ \frac{1}{4} $ is as usual arbitrary and can be replaced with any value in
$ (0, 1) $. 
In our setting, assuming that the skeleton median process is given, 
we are interested in whether the system is concentrated or diluted about level $
M_{t}-1 $. Therefore, we set
\begin{equation} \label{e:mft-def}
\begin{aligned}
\mf{m}_{t} = \mf{m} (M_{t}-1, u_{t}) \in \{ \tC, \tD \} \;.
\end{aligned}
\end{equation}
We introduce this marker because we exploit different mechanisms depending on whether
the system is in a diluted ($ \mf{m}_{t} = \tD $) or a concentrated ($
\mf{m}_{t} = \tC $) state. In both cases, for any $ i \in \NN $, provided that $
M_{T_{i}} $ is sufficiently large, the quantity $\|
\Pi_{M_{T_{i}}}^{\myg} u_{t} \|/\|
\Pi_{M_{T_{i}}}^{\myl} u_{t} \| $ is
very likely to decrease (if it is finite) for $ t \in [ T_{i}, T_{i+1}) $ as a consequence of the
dissipative nature of our equation. This
is the content of Proposition~\ref{prop:drft-OU}, which shows that the
evolution of this ratio is bounded from above by an Ornstein--Uhlenbeck process
on certain time intervals. 
At time $ T_{i+1} $, once the small scale energy captured by $ \|
\Pi_{M_{T_{i}}}^{\myg} u_{t}\|$ is likely to have dissipated, we update the
value of $ M_{T_{i}} $ to the new median, which we now expect to be strictly smaller
than the original one. But if the energy is
\textit{concentrated} in a shell of width two about level $ M_{T_{i}} -1$, namely if  $
(\Pi_{M_{T_{i}}-1}^{\myc} + \Pi_{M_{T_{i}}-2}^{\myc}) \pi_{t} $ contains a significant 
fraction of energy \dash in
the most extreme case without any energy in lower modes \dash then the
deterministic dynamic predicts that the energy will just further
concentrate at level $ M_{T_{i}}$ and dissipation alone is not sufficient
to guarantee that $ M_{T_{i+1}} < M_{T_{i}} $.
Instead it is the effect of the non-degenerate
noise \dash which pushes energy to modes surrounding $ M_{T_{i}} $ in a time of
order one \dash combined with dissipation which guarantees that the
median actually drifts to lower levels.

Eventually, we will compare the dynamics of $ M_{t} $ for large values of
$ M_{t} $ to a random walk with drift towards lower frequencies. Hence we have
to control how far the process jumps to high frequencies in the unlikely event
that this occurs. Here, once more, we make use of the dissipation, by obtaining suitable
regularity estimates for 
\begin{equ}[eqn:def-h]
  w_{t} = \frac{\Pi^{\myg}_{M_{t}}u_{t}}{\| \Pi^{\myl}_{M_{t}} u_{t}
  \|}=\frac{\Pi^{\myg}_{M_{t}} \pi_{t}}{\| \Pi^{\myl}_{M_{t}} \pi_{t} \|} \;.
\end{equ}
We use this approach to study, for $ \kappa_{0},
\kappa > 0, k_{0} \in \NN $, the functional
\begin{equ}[eqn:def-lyap-fun]
  \mf{F} (\kappa, \bpi_{t}) = \exp \left( \kappa_{0}  M_{t} +
    \kappa \| w_{t} \|_{M_{t}+ k_{0}}^{2} \right) \;,
\end{equ}
which is now a functional of the enhanced process
\begin{equ}
  \bpi_{t} = (M_{t}, \pi_{t}) \in \NN \times S\;,
\end{equ}
and should be interpreted as a skeleton version of the functional $
\mf{G} $ defined in \eqref{eqn:def-G-lyap}. While the parameters $ \kappa_{0} $ and
$ k_{0} $ will be fixed later on, $ \kappa $ will be allowed to vary.
Note that $ \mf{F} $ cannot be expressed as a
functional of $ \pi_{t}$ alone since the process $ M_{t} $ depends on the past:
one of the key technical results of this work is then
Theorem~\ref{thm:lyap-func-discrete}, which shows that $ \mf{F} $ satisfies
a  Lyapunov-type property.

In the next section we introduce one the main
building blocks of our analysis: a dissipation
estimate relating the relative energy $ \|
\Pi^{\myg}_{M_{T_{i}}} u_{t} \| / \| \Pi^{\myl}_{M_{T_{i}}} u_{t} \| $ to
an Ornstein--Uhlenbeck process on suitable time intervals.

\subsection{Dissipation estimates}\label{sec:dissip}

Let us start by considering any stopping time $ t_{0} \geqslant
0 $ and an $ \mF_{t_{0}} $--adapted random variable $ L \in \NN $. 
To state our dissipation estimates we need two additional
ingredients. First the dissipation is captured by the gap between different
eigenvalues, which satisfies
\begin{equ}
\Delta_{L} = \zeta_{L + 1} - \zeta_{L} \in (0, \infty), \qquad \lim_{L \to
\infty} \Delta_{L} = \infty \;, \qquad \Delta_{L} \sim L^{2 \a -1}\;.
\end{equ}
Then, for any $ u_{0} \in L^{2}_{\star}(\TT^{d}; \RR^{m}) $ and $ u_{t} $ the solution to
\eqref{eqn:main} we define the following functions for all \( t \geqslant t_{0}
\), assuming $ \| \Pi^{\myl}_{L} u_{t_{0}} \| > 0 $ almost surely:
\begin{equ}[eqn:def-w]
\gls{w}  = \frac{\Pi_{L}^{\myg}  u_{t}(x)}{\| \Pi_{
L}^{\myl}  u_{t} \|} \;, \qquad
\gls{wbold} = \frac{\Pi_{L}^{\mygeq}  u_{t}(x)}{\| \Pi_{
L}^{\myl}  u_{t} \|}\;.
\end{equ}
Note that nothing prevents the denominator
from vanishing in finite time. For this reason, we will only ever track these functions up to 
suitable stopping times which avoids this. In particular, we define for any $
\beta > 0 $:
\begin{equs}[eqn:tau-generic-n]
\gls{taul} & = \inf \{ t \geqslant t_{0}  \ \colon \ \|
w (L; t, \cdot) \| \leqslant 1/2 \} \wedge (t_{0} + 1) \;, \\
 \sigma^{\mygeq}_{\beta} (L,t_{0}) & = \inf \{
  t \geqslant t_{0}  \; \colon \; \| w^{\mygeq} (L; t , \cdot) \| \geqslant
\beta \} \wedge (t_{0} +1) \;.
\end{equs}
Stopping after a time interval of length at most one will later allow us to enforce
a deterministic upper bound on $T_{i+1} - T_i$. Further, the value $ 1/2
$ is of course
again somewhat arbitrary, and we use the letter $ \sigma $ rather than $ \tau
$ for the second stopping time since $ \tau^{\mygeq} $ will be reserved for the
particular threshold $ \beta=2 $ later on. As for the variable $ L $, we will
use these definitions mostly with $ L \in \{ M_{t_{0}}, M_{t_{0}} -2 \} $,
where $ M_{t} $ is the skeleton median process constructed in
Section~\ref{sec:st-def}.
Finally, it will be convenient to rewrite \eqref{eqn:main} in Fourier
coordinates.

\begin{remark}\label{rem:fourier}
  Let $ u_{t} $ be the solution to \eqref{eqn:main} with $ u_{0} \in
L^{2}_{\star} (\TT^{d}; \RR^{m}) $. Then for any $ k \in
  \ZZ^{d}$ the process $\hat{u}_{t}^{k}= \langle u_{t}, e_{k} \rangle $ satisfies
\begin{equ}
\ud \hat{u}_{t}^{k} = - \nu \zeta_{k} \hat{u}_{t}^{k} \ud t + \sum_{l \in
\ZZ^{d}}  \hat{u}_{t}^{k- l} \cdot \ud B^{l}_{t}\;.
\end{equ}
Recall that $ u $ is vector-valued, so here we have defined $ \nu =
(\nu^{\alpha})_{\alpha =1}^{m} $ and for $ \varphi \in \RR^{m} $ we define the
component-wise product
\begin{equ}
\nu \varphi = (\nu^{\alpha} \varphi^{\alpha})_{\alpha =1}^{m} \in
\RR^{m} \;.
\end{equ}
The quadratic covariation matrix between $ \hat{u}_{t}^{k}$ and $
\hat{u}_{t}^{l} $ is given by
\begin{equ}
\ud \langle \hat{u}^{k} , \hat{u}^{l} \rangle_{t} = \sum_{m \in \ZZ^{d}}
\hat{u}_{t}^{k - m} \otimes \hat{u}_{t}^{l+m} \cdot \Gamma_{m} \ud t
\eqdef C_{k, l}(u_{t})
\ud t\;. 
\end{equ}
Here the resulting covariance is a matrix $ (C_{k , l}^{\alpha,
\beta})_{\alpha, \beta} \in (\RR^{m})^{\otimes 2} $, defined by
\begin{equation*}
\begin{aligned}
C_{k , l}^{\alpha, \beta} (u) = \ud \langle \hat{u}^{\alpha, k},
\hat{u}^{\beta, l} \rangle  = \sum_{m} \sum_{\gamma, \eta} \hat{u}^{\gamma, k - m}
\hat{u}^{\eta , k + m} \Gamma^{\alpha, \beta}_{\gamma, \eta, m } \;,
\end{aligned}
\end{equation*}
where we can interpret the inner sum as a dot product between the tensor $
\hat{u}^{k-m} \otimes \hat{u}^{l+m} $ and the 4-tensor $ (\Gamma^{\alpha,
\beta}_{\gamma, \eta, m} )_{\alpha, \beta, \gamma, \eta} $, with $ m $ fixed.
Now, if we define
\begin{equation*}
\begin{aligned}
\overline{\Gamma}_{m} = \max_{\alpha, \beta, \gamma, \eta} | \Gamma^{\alpha,
\beta}_{\gamma, \eta, m}  | \;,
\end{aligned}
\end{equation*}
then the covariation $ C_{k,l}(u) $ can be estimated by the
Cauchy--Schwarz inequality, for every $ \alpha, \beta  $:
\begin{equ}[eqn:bd-cov]
| C_{k, l}^{\alpha, \beta}(u) | \leqslant \sum_{m}
\overline{\Gamma}_{m} | \hat{u}^{k + m} | | \hat{u}^{k - m} | \leqslant \| u
\|_{\l^{2}_{k}} \| u \|_{\l^{2}_{l}} \;, \end{equ}
with 
\begin{equ}
\| u \|_{\l^{2}_{k}}^{2} = \sum_{m \in \ZZ^{d}} \overline{\Gamma}_{k+ m} |
\hat{u}^{m} |^{2}\;.
\end{equ}
\end{remark}
Now we are ready to state our dissipation estimates.

\begin{proposition}\label{prop:drft-OU}
Under the assumptions of Theorem~\ref{thm:lyap-func}, consider any stopping time $ t_{0}$ and any $
\mF_{t_{0}}$-adapted random variable $ L $ with values in $ \NN $. Then fix $
u_{0} \in L^{2}_{\star} (\TT^{d}; \RR^{m}) $, let $ u_{t} $ be the solution
to \eqref{eqn:main} and assume that almost surely $ \| \Pi_{L}^{\myl}
u_{t_{0}} \| > 0 $.
Further, let $ w$ and $ w^{\mygeq} 
$ be as in \eqref{eqn:def-w} and  $ 0 < E < \beta $ two non-negative constants.
Then there exists a (deterministic) increasing function $ \beta \mapsto R(\beta) \in (0, \infty) $
such that the following holds if $ \| w^{\mygeq} (L; t_{0}, \cdot) \| \leqslant
E $, uniformly over all $ u_{0} $:
\begin{enumerate}
\item For all $ t \in [t_{0}, \sigma_{\beta}^{\mygeq} (L,t_{0}) ] $ we can bound
\begin{equ}
  \ud \| w(L; t, \cdot) \|^{2} \leqslant -2 \nu_{\mathrm{min}} \Delta_{L} \|
w(L; t, \cdot) \|^{2} \ud t + R(\beta) \ud t + \ud \mM_{t}\;,\label{eqn:mid-point}
\end{equ}
where $ \nu_{\mathrm{min}} = \min_{\alpha} \nu^{\alpha} > 0 $ and  $ \mM_{t} $ is a continuous, square integrable martingale on $ [t_{0},
\sigma_{\beta}^{\mygeq}(L,t_{0})] $ with $ \mM_{t_{0}} = 0 $ and
quadratic variation bounded by $ \ud \langle \mM \rangle_{t} \leqslant
R(\beta ) \ud t$.
\item Similarly, for all $ t \in [t_{0}, \sigma_{\beta}^{\mygeq} (L,
t_{0})] $ we can bound:
\begin{equ}
  \ud \| w^{\mygeq} (L; t, \cdot) \|^{2} \leqslant R(\beta) \ud t  + \ud
  \mathcal{N}_{t}\;,
\end{equ}
where $ \mathcal{N}_{t} $ is a continuous, square integrable martingale on $ [t_{0},
\sigma^{\mygeq}_{\beta}(L,t_{0}) ] $ with $ \mathcal{N}_{t_{0}} = 0 $ and
quadratic variation bounded by $ \ud \langle \mathcal{N} \rangle_{t}  \leqslant
R(\beta) \ud t $.
\end{enumerate}
\end{proposition}

\begin{proof}
By the strong Markov property of the solution $ u $ we may assume that $ t_{0} = 0 $. 
To show the first point, we use It\^o's formula to find
\begin{equs}
\ud \frac{\| \Pi_{ L}^{\myg}  u_{t} \|^{2}}{\| \Pi_{L}^{\myl} 
u_{t} \|^{2}} & = 2 \left[ \frac{1}{\| \Pi_{L}^{\myl}  u_{t} \|^{2}}\langle
\Pi_{L}^{\myg}  u_{t}, \mL  u_{t}\rangle - \frac{\| \Pi_{L}^{\myg}  u_{t} \|^{2}}{\| \Pi_{
L}^{\myl}  u_{t} \|^{4}} \langle \Pi_{L}^{\myl}  u_{t}, \mL u_{t}\rangle
\right] \ud t \\ & \quad + \ud \mM_{t} + \frac{1}{2} \sum_{k, l \in
\ZZ^{d}} \partial_{ \hat{u}^{k} \hat{u}^{l}} \Big( \frac{A_{t}}{B_{t}} \Big) : \ud
\langle \hat{u}^{k}, \hat{u}^{l} \rangle_{t}\;. \label{e:ItoFirst}
\end{equs}
Here the last line contains the quadratic covariation term and we have defined
$$ A_{t} = \| \Pi_{L}^{\myg}  u_{t} \|^{2}\;, \qquad B_{t} = \|
\Pi_{L}^{\myl}  u_{t} \|^{2}\;,$$ so that $  \partial_{ \hat{u}^{k} \hat{u}^{l}}
\Big( \frac{A_{t}}{B_{t}} \Big) $ is the matrix, over the indices $ \alpha,
\beta \in \{ 1, \dots , m \} $
\begin{equation*}
\begin{aligned}
  \left( \partial_{ \hat{u}^{k} \hat{u}^{l}} \Big( \frac{A_{t}}{B_{t}} \Big)
\right)^{\alpha, \beta} = \partial_{ \hat{u}^{\alpha, k} \hat{u}^{\beta,l}} \Big( \frac{A_{t}}{B_{t}} \Big)
\end{aligned}
\end{equation*}
  and for two matrices \( G, F \in
(\RR^{m} )^{\otimes 2} \) we have $ F : G = \sum_{\alpha, \beta}
F^{\alpha, \beta} G^{\alpha, \beta} $.
Then, via Remark~\ref{rem:fourier}, we can estimate the quadratic variation
term by

\begin{equs}
\sum_{k, l \in \ZZ^{d}} & \partial_{ \hat{u}^{\alpha, k} \hat{u}^{\beta, l}} \Big(
\frac{A_{t}}{B_{t}} \Big) \ud \langle \hat{u}^{k} , \hat{u}^{l}
\rangle_{t}^{\alpha, \beta} \\
 \lesssim &\sum_{|k| > (L+1)_{\alpha}} \frac{| C^{\alpha,
\beta}_{k , -k}(u_{t}) |}{B_{t}}\one_{\{ \alpha = \beta \}} +
\sum_{|l| \leqslant L_{\alpha}}   \frac{A_{t}}{B_{t}^{2}} |C_{l, -l}^{\alpha, \beta}(u_{t})| \one_{\{\alpha = \beta\}} \\
&  + \sum_{| k | \leqslant L_{\alpha}, | l | \leqslant L_{\beta}}
\frac{A_{t}}{B_{t}^{3}} | \hat{u}_{t}^{l} | | \hat{u}^{k}_{t} | |C^{\alpha, \beta}_{k,
l}(u_{t})|  + \sum_{|k| > (L+1)_\alpha, |l| \leqslant L_\beta} \frac{| \hat{u}_{t}^{k}
| | \hat{u}_{t}^{l} |}{B^{2}_{t}} | C_{k, l}^{\alpha , \beta}(u_{t}) |\;.
\end{equs}
Next we use the bound in \eqref{eqn:bd-cov}. We observe that by
\eqref{e:reg-assu}, since $ \sum_{m} \overline{\Gamma}_{m} < \infty $,
we find that $ \sum_{k \in \ZZ^{d}} \| u
\|_{\l^{2}_{k}}^{2} \lesssim \| u \|^{2} $ and therefore
\begin{equs}
\sum_{l} | C^{\alpha, \beta}_{l , -l} (u) | & \lesssim \| u \|^{2} \;, \\
\sum_{l, k} | \hat{u}^{k}  | | \hat{u}^{l}  | | C^{\alpha, \beta}_{k, l}
(u)|& \lesssim \Big( \sum_{l} | \hat{u}^{l} | \| u
\|_{\l^{2}_{l}} \Big)^{2} \lesssim \| u \|^{4} \;.
\end{equs}
Then we can estimate the terms above as
follows for $ t \in [t_{0}, \sigma_{\beta}^{\mygeq} (t_{0})] $:
\begin{equs}
\sum_{k, l \in \ZZ^{d}} \partial_{ \hat{u}^{k} \hat{u}^{l}} \Big(
\frac{A_{t}}{B_{t}} \Big) & : \ud \langle \hat{u}^{k}, \hat{u}^{l} \rangle_{t}
\lesssim \frac{\| u_{t} \|^{2}}{B_{t}} + \frac{\| u_{t} \|^{4}}{B_{t}^{2}}  + \frac{\| u_{t} \|^{6}}{B_{t}^{3}}  \lesssim_{\beta} 1 \;.
\end{equs}
Regarding the first term in \eqref{e:ItoFirst}, we observe that 
\begin{equ}
\frac{1}{\| \Pi_{L}^{\myl} u_{t} \|^{2}}\langle
\Pi_{L}^{\myg}  u_{t}, \mL  u_{t}\rangle - \frac{\| \Pi_{L}^{\myg}  u_{t} \|^{2}}{\|
\Pi_{L}^{\myl}  u_{t} \|^{4}} \langle \Pi_{L}^{\myl}  u_{t}, \mL u_{t}\rangle \leqslant
-(\zeta_{L+1} - \zeta_{L}) \frac{A_{t}}{B_{t}}\;.
\end{equ}
Here we have used the particular definition of the projection in
Definition~\ref{def:proj} to deal with distinct viscosity coefficients. Indeed,
note that for every $ \alpha \in \{ 1, \dots, m \} $ we have that
\begin{equation*}
\begin{aligned}
\langle \Pi^{\myg}_{L} u^{\alpha}, \mL^{\alpha} u^{\alpha} \rangle & = -
\nu^{\alpha}\langle \Pi^{\myg}_{L} u^{\alpha}, (- \Delta)^{\a }
u^{\alpha} \rangle \\
&  \leqslant - \nu^{\alpha} | (L+1)_{\alpha} |^{2 \a } \| \Pi^{\myg} u^{\alpha} \|^{2} 
= - (L+1)^{2 \a } \| \Pi^{\myg} u_{t} \|^{2} \;,
\end{aligned}
\end{equation*}
as desired, and similarly for the low frequency projection.
This concludes the proof of \eqref{eqn:mid-point}. As for the quadratic variation, the
martingale term is given by
\begin{equ}
\ud \mM_{t} =  2 \left[ \frac{1}{\| \Pi_{L}^{\myl} u_{t} \|^{2}}\langle
\Pi_{L}^{\myg}  u_{t},  u_{t} \cdot \ud W_{t}\rangle - \frac{\| \Pi_{L}^{\myg}  u_{t}
\|^{2}}{\| \Pi_{L}^{\myl} 
u_{t} \|^{4}} \langle \Pi_{L}^{\myl}  u_{t}, u_{t} \cdot \ud W_{t}\rangle
\right]\;.
\end{equ}
The first bracket (and verbatim for the second one) can be rewritten as
follows
\begin{equ}
\langle \Pi_{L}^{\myg}  u_{t},  u_{t} \ud W_{t}\rangle = \sum_{\alpha, \beta} \sum_{| k | >
(L+1)_{\alpha} } \hat{u}_{t}^{\alpha, k} \sum_{l} \hat{u}_{t}^{\beta, k-l} \ud
B^{\alpha, \beta}_{ l, t}\;,
\end{equ}
which has quadratic variation
\begin{equation*}
\begin{aligned}
\sum_{\alpha, \beta, \gamma, \eta} \sum_{l} \Big(  \sum_{| k | >
(L+1)_{\alpha} } \hat{u}_{t}^{\alpha, k}  \hat{u}_{t}^{\beta, k-l}
\Big) \cdot \Big(  \sum_{| k | >
(L+1)_{\gamma} } \hat{u}_{t}^{\gamma, k} \hat{u}_{t}^{\eta, k-l}
\Big) \Gamma^{\alpha, \gamma}_{\beta, \eta, l} \;,
\end{aligned}
\end{equation*}
which can be bounded by
\begin{equ}
\sum_{\alpha, \beta} \sum_{l} \overline{\Gamma}_{l} \Big| \sum_{| k | > (L+1)_{\alpha}} \hat{u}_{t}^{ \alpha, k}
\hat{u}_{t}^{\beta, k-l} \Big|^{2} \lesssim \| u_{t} \|^{4}\;.
\end{equ}
Hence with the same notation as above we obtain
\begin{equ}
  \frac{\ud}{\ud t}  \langle \mM \rangle_{t} \lesssim \frac{\| u_{t}
\|^{2}}{B_{t}^{2}} + \frac{\| u_{t} \|^{8}}{B_{t}^{4}}  \lesssim_{\beta} 1\;,
\end{equ}
which proves the required bound.
To obtain the estimate for $w^{\mygeq} $, we can follow verbatim the
previous calculations. The only difference lies in the treatment of the term
\begin{equ}
\frac{1}{\| \Pi_{L}^{\myl} u_{t} \|^{2}}\langle
\Pi_{L}^{\mygeq}   u_{t}, \mL  u_{t}\rangle - \frac{\| \Pi_{L}^{\mygeq}  u_{t} \|^{2}}{\|
\Pi_{L}^{\myl}  u_{t} \|^{4}} \langle \Pi_{L}^{\myl}  u_{t}, \mL u_{t}\rangle \leqslant
-(\zeta_{L} - \zeta_{L}) \frac{A_{t}}{B_{t}} = 0\;,
\end{equ}
which delivers the desired result.
\end{proof}

\subsection{Construction of the skeleton median}\label{sec:st-def}

We are now ready to define the stopping times $ \{ T_{i} \}_{i \in \NN} $ and
the skeleton median process $ (M_{t})_{t \geqslant 0} $. The first step in the
definition of the stopping times is to make sure that they do not kick in too
quickly. We therefore start by introducing a \emph{padding} time of length about $
\delta $, where $ \delta \in (0, 1) $ is a fixed parameter (the padding is the
 red region in the figure below).
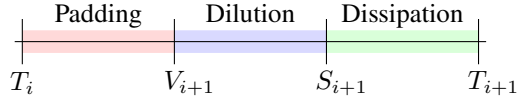
\begin{center}
\begin{tikzpicture}
\draw[line width=3mm,red!15] (0,0) -- (2,0);
\draw[line width=3mm,blue!15] (2,0) -- (4,0);
\draw[line width=3mm,green!15] (4,0) -- (6,0);
	\draw (-0.1,0.0) -- (6.1,0.0) ;
	\draw (0,0.3) -- (0, -0.3) ;
  \node at (0, -0.55) (a) {$ T_{i} $};
\node at (1.0,0.35) (a1) {\strut Padding};
\node at (3.0,0.35) (a2) {\strut Dilution};
\node at (5.0,0.35) (a3) {\strut Dissipation};
  \draw (2.0,0.3) -- (2.0, -0.3) ;
  \node at (2.2, -0.55) (b) {$ V_{i+1} $};
  \draw (4.0,0.3) -- (4.0, -0.3) ;
  \node at (4.2, -0.55) (c) {$ S_{i+1}$};
  \draw (6.0,0.3) -- (6.0, -0.3) ;
  \node at (6.2, -0.55) (c) {$ T_{i+1}$};
\end{tikzpicture} 
\end{center}
After the padding time, we wait for the noise to shift the
system into a diluted state, at time $ S_{i+1}$ \dash note
that the system may well be diluted to start with. Once the system is diluted, we wait for
dissipation to kick in and reduce the value of the median. 
Of course we may be in bad luck
and see an unexpected event somewhere along the line: if this is the case we
stop prematurely: for this reason $ V_{i+1} $ is for example a stopping time,
and not the deterministic $ T_{i} + \delta $. This discussion
motivates the definition of the following stopping times as well as of the
skeleton median process. Recall here the stopping
times defined in \eqref{eqn:tau-generic-n}, the marker defined in
\eqref{e:mft-def}, and let us also introduce the first dilution time
\begin{equ}[e:defsigmaDil]
\sigma^{\tD} (t_{0}) = \inf \{ t \geqslant t_{0}  \; \colon \;
\mf{m}_{t} =\tD \} \;.
\end{equ}

  \begin{definition}\label{def:Tn} Let $ \delta \in (0, 1) $ be fixed and
set, for any $ \pi_{0} \in S $ \[ T_{0} = 0 \;, \qquad  M_{0} = M(\pi_{0}) \;. \]  Then for any $ i \in \NN $,
assuming that $
T_{i} $ and $ (M_{t})_{t \leqslant T_{i}} $ are given, and writing $M_i =
M_{T_i}$ for short, define the following.
\begin{enumerate}
\item The next padding time:
\begin{equ}
  V_{i+1} = (T_i + \delta) \wedge \sigma^{\mygeq}_{5/4} (M_{i}, T_{i})\;.
\end{equ}
\item The next dilution time:
\begin{equation*}
S_{i+1} = \sigma^{\tD} (V_{i+1}) \wedge
\sigma^{\mygeq}_{3/2} (M_{i}, V_{i+1}) \;.
\end{equation*}
\item The next dissipation time:
 \begin{equs}
    T_{i+1}  = \tau^{\myl} (M_{i} -2, S_{i+1}) \wedge \sigma_{2}^{\mygeq}
(M_{i}-2,  S_{i+1}) \;.
  \end{equs}
\item The skeleton median process up to time $ T_{i+1} $:
\begin{equs}
M_{t} & = M_{i} \;, \qquad \qquad \qquad \quad \ \ \forall t \in [ T_{i} , T_{i+1}) \;, \\
  M_{T_{i+1}} & = M_{i+1} = \begin{cases} M_{i} - 1 \qquad  & \text{ if } \ \
    M(\pi_{T_{i+1}}) < M_{i} \;, \\ M (\pi_{T_{i+1}}) & \text{
  else. }\end{cases} 
\end{equs}
\end{enumerate}

Finally, for any stopping time $ t_{0}$ we write $
T(t_{0}) $ for the first time after $ t_{0} $ in which the skeleton median
jumps:
\begin{equation}\label{e:Tt0}
\begin{aligned}
T (t_{0}) = \min \{ T_{i}  \; \colon \; T_{i} > t_{0} \} \;.
\end{aligned}
\end{equation}
\end{definition}

\begin{remark}
This definition may appear slightly circular since $\sigma^{\tD}$ used in point~2 requires
knowledge of $M_t$ for $t > T_i$ which is only defined in point~4. 
To alleviate this, one could simply replace $ \mf{m}_{t} $ with $ \mf{m}(M_{t_{0}}-1,u_t)$ in 
\eqref{e:defsigmaDil} which leads to the same construction.
\end{remark}
As we have described before, the stopping times $ \sigma^{\mygeq}_{\beta} $
should kick in rarely (at least in the regime of interest to us, namely when the 
Lyapunov function is large) and are present to guarantee that the system stays under
control. We can therefore identify an event $ \mA_{i} $, in which all
the stopping times kick in, which are more likely to do so (at least if $ M
$ is sufficiently large). Its complement $ \mB_{i} = \mA^{c}_{i} $ covers instead the
case in which one of the \(\sigma_{\beta}^{\mygeq} \) kicks in, and happens
with small probability:
\begin{equation*}
\begin{aligned}
\mA_{i} & = \{ V_{i +1} = T_{i } + \delta  \} \cap \{ \mf{m}_{S_{i+1}} = \tD \} \cap
\{ T_{i +1} = \tau^{\myl}  (M_{i} -2, S_{i +1}) \} \;, \\ 
\mB_{i}& = \mA_{i}^{c}\;.
\end{aligned}
\end{equation*}
Now we would like to establish some basic properties of the skeleton median.
For example, the values $ 5/4, 3/2 $ and $ 2 $ are chosen increasingly so that
the stopping times kick in one strictly after the other. This is not always the
case, since $ \sigma_{2} (M_{i}-2, S_{i+1}) \leqslant \sigma_{3/2} (M_{i}, V_{i
+1})  $ is guaranteed only if $ \mf{m}_{S_{i+1}} = \tD $ (recall that $ \mf{m}_{t} = \mf{m}
(M_{t}-1, u_{t}) $ is defined in \eqref{e:mft-def}), and even this bound
requires a short computation, see below. Another fundamental
property is that on the event $ \mA_{i} $ we have $ M_{T_{i +1}} =
M_{T_{i}} -1$.

Further, the skeleton median process is defined to be constant on each interval
$ [T_{i}, T_{i +1}) $, as already anticipated. Yet, in opposition to the
heuristic definition we have previously given, at time $ T_{i +1} $ we \emph{do
not} update its value to match $ M(\pi_{T_{i+1}}) $, unless $
M(\pi_{T_{i+1}}) \geqslant  M_{T_{i}} $.
This choice is taken because of a rather technical issue. If $ M (u_{T_{i}}) \gg 1 $, then we expect the
true median to jump very quickly to lower
frequencies, meaning that $ M (\pi_{T_{i+1}}) \ll M (\pi_{T_{i}}) $. But if the
median drops quickly, it is cumbersome to control high-frequency regularity: as $ M $ decreases,
the value of  $ \| \pi_{t} \|_{M} $ increases in such a way that the first
effect could be canceled by the latter and may not see any decrease in our
Lyapunov functional. With our definition we artificially rule out such large drops
in $ M_{t} $. Our skeleton process still satisfies $ M_{T_{i}} \geqslant M
(\pi_{T_{i}}) $ for all \( i \in \NN \), but the inequality may be strict and
there is no upper bound on the gap $ M_{i} - M
(\pi_{T_{i}}) $. We collect these and other considerations
in the following lemma.

\begin{lemma}\label{lem:consistency} The skeleton median process $
(M_{t})_{t \geqslant 0} $ and the stopping times defined in
Definition~\ref{def:Tn} satisfy the following properties:
  \begin{enumerate}
\item For all $ t \geqslant 0 $, it holds that
\begin{equ}[e:bd-pi]
\| \Pi^{\mygeq}_{M_{t}} \pi_{t} \| \leqslant 2 \| \Pi^{\myl}_{M_{t}}
\pi_{t} \| \;, \qquad  1 \leqslant \sqrt{5} \| \Pi^{\myl}_{M_{t}} \pi_{t} \|
\;.
\end{equ}
     \item
There exists a (deterministic) constant $ E \in (0, 2) $ such that
\begin{equ}
\| w^{\mygeq} (M_{i}-2; S_{i+1}, \cdot) \| \leqslant E\;, \qquad \text{ if }
\quad \mf{m}_{S_{i+1}} = \tD \;.
\end{equ}
\item For all $ i \in \NN $ we can lower bound 
\( M_{i} \geqslant M (\pi_{T_{i}})\).
\item The jumps of the skeleton median are such that for all $ i \in
\NN$ one has $M_{i +1} \geqslant M_{i} - 1$, and furthermore
$M_{i+1} = M_{i} -1$ on $ \mA_{i}$.
\item For all $ i \in \NN $ we have $    T_{i+1} - T_{i} \leqslant 3$.
  \end{enumerate}
\end{lemma}

\begin{proof}
The first bound in \eqref{e:bd-pi} follows from  the definitions of $ \sigma^{\mygeq}_{\beta} $
and of $ \tau^{\myg} $ and the second bound follows from the
first since $\| \Pi^{\mygeq}_{M_{t}} \pi_{t} \|^2 + \| \Pi^{\myl}_{M_{t}}
\pi_{t} \|^2=1$. The second point holds since
\begin{equs}
    \frac{\| \Pi^{\mygeq}_{M_{i}-2} u_{S_{i+1}}  \|}{\|
    \Pi^{\myl}_{M_{i}-2} u_{S_{i+1}} \|} & \leqslant \sqrt{ \frac{\|
( \Pi^{\myc}_{M_{i}-1}+ \Pi^{\myc}_{M_{i}-2}) u_{S_{i+1}} \|^{2}}{ \| \Pi^{\myl}_{M_{i}-2}
u_{S_{i+1}} \|^{2}} + \frac{\| \Pi^{\mygeq}_{M_{i}} u_{S_{i +1}} \|^{2}}{\|
\Pi^{\myl}_{M_{i}-2} u_{S_{i+1}} \|^{2}} }  \\
& \leqslant \sqrt{ \left( \frac{1}{4} \right)^{2} + \frac{\| \Pi^{\myl}_{M_{i}} u_{S_{i +1}} \|^{2}}{\|
\Pi^{\myl}_{M_{i}-2} u_{S_{i+1}} \|^{2}} \cdot  \frac{\| \Pi^{\mygeq}_{M_{i}} u_{S_{i +1}} \|^{2}}{\|
\Pi^{\myl}_{M_{i}} u_{S_{i+1}} \|^{2}}} \\
& \leqslant \sqrt{ \left( \frac{1}{4} \right)^{2} + \left( 1 + \left(
\frac{1}{4} \right)^{2} \right) \left( \frac{3}{2} \right)^{2}} < 2 \;,
\end{equs}
where we used that since $ \mf{m}_{S_{i+1} } = \tD $ we have the bound $ \| (
\Pi^{\myc}_{M_{i} -2} + \Pi^{\myc}_{M_{i}-1}) u_{S_{i+1}}\| \leqslant
\frac{1}{4} \| \Pi^{\myl}_{M_{i}-2} u_{S_{i+1}} \| $, together with the fact
that $ S_{i+1} \leqslant \sigma_{3/2}^{\mygeq} (M_{i}, T_{i}) $. 

The third point follows immediately from Definition~\ref{def:Tn}. As for the fourth
point, we see that on $ \mA_{i} $ necessarily
either $ T_{i+1} = \tau^{\myl} (M_{i} -2, S_{i+1}) $. Therefore we find from
the definition of $ \tau^{\myl} $ that
\begin{equ}
\| \Pi^{\mygeq}_{M_{i}-1} u_{T_{i+1}} \| \leqslant \frac{1}{2} \|
\Pi^{\myl}_{M_{i}-2} u_{T_{i+1}} \| \leqslant \frac{1}{2} \|
\Pi^{\myl}_{M_{i}-1} u_{T_{i+1}}  \| \;,
\end{equ}
so that indeed $ M (u_{T_{i +1}}) \leqslant M_{i}-1 $.
The last point follows by observing that the stopping times $ \tau^{\mygeq} $ and $
\sigma^{\mygeq}_{\beta} $ kick in after a time interval of length at most one.
\end{proof}

Finally, to improve the readability of later calculations, we introduce the
following shorthand notations for the stopping times that have appeared so far
and an additional arbitrary stopping time $ t_{0} $:
\begin{equs}[e:short]
\tau^{\mygeq} (t_{0}) &= \sigma_{2}^{\mygeq}(M_{t_{0}}-2, t_{0}) \;,\quad \ \ 
& \tau^{\myl} (t_{0}) &= \tau^{\myl} (M_{t_{0}}-2, t_{0}) \;, \\
\sigma_{\beta}^{\mygeq} (t_{0}) & = \sigma^{\mygeq}_{\beta} (M_{t_{0}},
t_{0}) \;, \ \ & \sigma^{\mygeq}(t_{0}) & = \sigma^{\mygeq}_{3/2} (M_{t_{0}},
t_{0}) \;.
\end{equs}
\section{Proof of the main result}\label{sec:prf-main-result}

The proof of Theorem~\ref{thm:lyap-func}, follows from an analogue of the
same theorem for the stopped process.
To state this result, recall that the definition of $ M_{t} $ depends on the padding parameter
$ \delta \in (0, 1) $. Furthermore the Lyapunov functional $ \mf{F} $ defined
in \eqref{eqn:def-lyap-fun} depends on
the parameters $ \kappa_{0}> 0, k_{0} \in \NN $ and on the free variable $
\kappa > 0 $. For convenience we recall here its definition, and that of the
Lyapunov functional $ \mf{G} $:
\begin{equs}
  \mf{F} (\kappa, \bpi) & = \exp \left( \kappa_{0}  M +
    \kappa \| w(\bpi)\|_{M+ k_{0}}^{2} \right) \;, \qquad
  \bpi = (M, \pi) \in \NN \times S\;, \\
  \mf{G} (\pi) & = \exp \left( \kappa_{0} M (\pi) + \| \pi
  \|_{M(\pi) +k_{0}}^{2} \right)\;.
\end{equs}
Here we use the definition $w(\bpi) = \Pi^{\myg}_{M} \pi / \| \Pi^{\myl}_{M} \pi \|$ as in
\eqref{eqn:def-h}.
Finally, the Lyapunov property of $ \mf{F} $ will hold outside of a compact set
of the state space of $ \bpi_{t} $. We use the following convention for
any $ \mathbf{K}\in \RR_{+}$
\begin{equ}
  \| \bpi \| < \mathbf{K} \ \ \Leftrightarrow \ \
  M < \mathbf{K} \ \text{ and } \ \| w \|_{M + k_{0}}< 
  \mathbf{K} \;.
\end{equ}

\begin{theorem}\label{thm:lyap-func-discrete}
  Under the assumptions of Theorem~\ref{thm:lyap-func}, for any $ u_{0} \in
L^{2}_{\star} $ (or alternatively $ \pi_{0} \in S $) let $ M_{t} $ be the
skeleton median (with parameter $ \delta \in (0, 1)  $) from Definition~\ref{def:Tn} with associated stopping times
$ \{ T_{i} \}_{i \in \NN} $ and $ \bpi_{t} = (M_{t},
w_{t}) $, with $ w_{t} $ as in \eqref{eqn:def-h}.
  Then, for any $ b \geqslant 1  $ and $ \mf{c} \in (0,1) $ there exist $
  \delta \in (0, 1), k_{0} \in \NN^{+},  \mathbf{K} \in
  \RR_{+},  J > 0,  \kappa_{0} > 0 $ and increasing  maps $$ J, \mf{h} \colon
  [1/2, \infty) \to [1, \infty)\;, \qquad \text{ with } \qquad \mf{h}(1/2)=1 \;,$$ such that:
\begin{enumerate}
\item For all $
  \kappa \geqslant 1/2$ 
  \begin{equs}[eqn:contraction]
  \EE_{T_{i}} \left[ \mf{F}^{b}( \kappa , \bpi_{T_{i+1}}) \right] \leqslant
  \mf{c} \cdot \mf{h}(\kappa) \cdot \mf{F}^{b}( \kappa /2  , \bpi_{T_{i}}) & \one_{ \{
  \| \bpi_{T_{i}} \| \geqslant \K \} } \\
 +  J(\kappa)  &\one_{\{
  \| \bpi_{T_{i}}\|  < \K \}} \;.
\end{equs}
\item There exists a $ \mf{C} > 1 $ and $ \kappa_{1} > 0 $ such that uniformly over and $i \in
\NN $, if $ \kappa \geqslant \kappa_{1} $
\begin{equ}[e:unif]
  \EE_{T_{i}} \bigg[ \sup_{T_{i} \leqslant s< T_{i+1}} \mf{G}^{b}(\pi_{s}) \bigg]
  \leqslant \mf{C} \cdot \mf{F}^{b}(\kappa, \bpi_{T_{i}}) \;.
\end{equ}
\item If $ \kappa \leqslant 1/2 $, then 
$  \mf{G}(\pi_{0}) \geqslant \mf{F}(\kappa, \bpi_{0}) $ for all $
\pi_{0} \in \mS $ and $ \bpi_{0} = (M (\pi_{0}), \pi_{0}) $.
\end{enumerate}

\end{theorem}
In this result, the last two claims establish the relationship between the
skeleton Lyapunov functional and the full Lyapunov functional. We highlight in
particular that the lower bound in the last point can hold only at time $ t = 0
$, since at later times the skeleton median process $ t \mapsto M_{t} $ may be
much larger than the actual energy median $ M (\pi_{t}) $.
The crucial statement in the theorem is the first one, which establishes a
version of the Lyapunov property. 

The statement captures two effects: on one hand
taking expectations will reduce the value of the parameter $ \kappa $ in $
\mf{F} $: this is almost a super-Lyapunov property, although there is no gain
in the prefactor of the skeleton median. On the other hand,
when $ \kappa $ is equal to $ 1/2
$, then we recover the classical Lyapunov
property since $ \mf{h} (1/2) =1 $.
We will need these properties, because we will first of all apply the uniform
bound \eqref{e:unif}, which holds for some $ \kappa_{1} $ potentially larger
than $1/2$. Then we apply \eqref{eqn:contraction} several times, to decrease
the value of $ \kappa $ from $ \kappa_{1} $ to $1/2$. At this point,
\eqref{eqn:contraction} will guarantee the Lyapunov property.
Before we turn to the proof of this discretised result, we show that it is sufficient to
obtain Theorem~\ref{thm:lyap-func}.

\begin{proof}[of Theorem~\ref{thm:lyap-func}]
  Since the Markov process $ (\pi_{t})_{t \geqslant 0}$ is time homogeneous it
  suffices to show the desired property for $ t = 0 $. For any $ \mf{c} \in (0,
  1) $ consider $ \delta, k_{0},  \K, J$ and  $\kappa_{0} $ fixed as in
  Theorem~\ref{thm:lyap-func-discrete}. Then if we define $ A_{j} =
  [T_{j}, T_{j+1}) $ we have by Cauchy--Schwarz
  \begin{equ}
    \EE [\mf{G}(\pi_{t_{\star}})]  = \EE \Big[ \sum_{j \in \NN}
    \mf{G}(\pi_{t_{\star}}) \one_{A_{j}} (t_{\star})  \Big] 
     \leqslant \sum_{j \in \NN} \EE \left[ \mf{G}^{2}
      (\pi_{t_{\star}})\one_{A_{j}} (t_{\star}) 
  \right]^{\frac{1}{2}} \sqrt{\mathbf{p}_{ t_{\star}}(j) } \;,
  \label{eqn:cs-lyap}
  \end{equ}
  where we have defined
  \begin{equ}[eqn:def-p-star]
    \mathbf{p}_{t_{\star}}(j)  = \PP ( t_{\star} \in A_{j})\;,
    \qquad \forall j \in \NN  \;.
  \end{equ}
  Let us now fix $ \kappa_{1} = 8 \kappa_{0} + 1 $ and choose $ j_{0} \in \NN $
such that $ \kappa_{1} / 2^{j_{0}}
  \leqslant 1/2 $. Then we estimate uniformly over $ j \geqslant
  j_{0}$, using the uniform bound of Theorem~\ref{thm:lyap-func-discrete}
  with $ \kappa = \kappa_{1} $
  \begin{equs}
    \EE\left[ \mf{G}^{2} (\pi_{ t_{\star}}) \one_{A_{j}} 
    (t_{\star})  \right] & \leqslant \mf{C} \EE \left[
    \mf{F}^{2}(\kappa_{1}, \bpi_{T_{j}}) \right] \\
& \leqslant \mf{C} \EE\left[ \mf{c}^{j_{0}} \mf{h}^{j_{0}} (\kappa_{1}) \mf{F}^{2}(
      1/2,  \bpi_{T_{j-j_{0}}}) +J (\kappa_{1})
    \sum_{\l = 0}^{j_{0}-1} \mf{h}^{\l}(\kappa_{1})\mf{c}^{\l}\right] \;,
\end{equs}
by applying \eqref{eqn:contraction}  for a total of $ j_{0} $ times. Now we can
use the fact that $ \mf{h}(1/2) = 1 $ to further bound via
\eqref{eqn:contraction}:
\begin{equ}
     \EE\left[ \mf{G}^{2} (\pi_{ t_{\star}}) \one_{A_{j}} 
    (t_{\star})  \right]  \leqslant \mf{C} \mf{h}^{j_{0}} (\kappa_{1}) \mf{c}^{j}
\mf{F}^{2}(1/2, \bpi_{0}) + \frac{J(\kappa_{1}) \mf{C}\mf{h}^{j_{0}} (\kappa_{1}) }{1 -
\mf{c}}\;. \label{eqn:cs-j-bound}
  \end{equ}
Note that so far we did not use any property of $ t_{\star} $. Now we will
choose $ t_{\star} $ sufficiently large, so that the factor $
\mf{c}^{j} $ can compensate the constant $ \mf{C} \mf{h}^{j_{0}} ( \kappa_{1}) $.
  In particular, by Lemma~\ref{lem:stopping-density-bounds} for any $ j_{1} >
  j_{0} $ we can choose $
  t_{\star} > 0 $ such that $ \mathbf{p}_{t_{\star}} (j) = 0 $ for all $ j \leqslant
  j_{1} $.
  We therefore obtain the bound
  \begin{equs}
    \EE [ \mf{G}(\pi_{t_{\star}})] & \leqslant \sum_{j > j_{1}}
    \sqrt{p_{t_{\star}} (j)} \left( (\mf{C} \mf{c}^{j} \mf{h}^{j_{0}}(\kappa_{1}))^{\frac{1}{2}} 
    \mf{F}( 1/2, \bpi_{0}) + \sqrt{\frac{J(\kappa_{1}) \mf{C} \mf{h}^{j_{0}}(\kappa_{1})}{1 - \mf{c}}} \right) \\
    & \leqslant   \mf{c}^{\frac{j_{1}}{2}}  \frac{\sqrt{\mf{C} \mf{h}^{j_{0}}(\kappa_{1})}}{1 -
    \mf{\sqrt{c}}} \mf{F}(1/2, \bpi_{0})
    +C(t_{\star}) \sqrt{J(\kappa_{1}) \mf{C} \mf{h}^{j_{0}}(\kappa_{1}) (1 - \mf{c})^{-1}} \;,
  \end{equs}
  where an application of Lemma~\ref{lem:stopping-density-bounds} guarantees
  that for some $ C(t_{\star}) > 0 $
  \begin{equ}
    \sum_{j > j_{1} } \sqrt{\mathbf{p}_{t_{\star}}(j)} \leqslant
    C(t_{\star}) \;.
  \end{equ}
  Now for any $
  \overline{\mf{c}} \in (0, 1) $ we can choose $ j_{1} > j_{0} $ such
  that $$ \mf{c}^{\frac{j_{1}}{2}}  \frac{\sqrt{\mf{C} \mf{h}^{j_{0}}(\kappa_{1})}}{1 - \mf{\sqrt{c}}}
  \leqslant \overline{\mf{c}} \;. $$ Therefore, for any $ \overline{\mf{c}} \in
  (0, 1) $ we have found a $ t_{\star} > 0 $ and a $ J^{\prime}(t_{\star}) > 0 $ such that
\begin{equ}
  \EE [ \mf{G}(\pi_{ t_{\star}})]  \leqslant \overline{\mf{c}}  \mf{F}(1/2,
  \bpi_{0})  + J^{\prime}  \leqslant \overline{\mf{c}}  \mf{G}( \pi_{0})  + J^{\prime}\;,
  \end{equ}
  where we used that $ \mf{F}(1/2, \bpi_{0}) \leqslant \mf{G}(\pi_{0}) $ by the
  last property of Theorem~\ref{thm:lyap-func-discrete}. This completes the proof. 
\end{proof}

\begin{lemma}\label{lem:stopping-density-bounds}
  Under the assumptions of Theorem~\ref{thm:lyap-func}, for $ t_{\star} > 0 $ and $j \in \NN $, let $ \mathbf{p}_{
  t_{\star}}(j) $ be as in \eqref{eqn:def-p-star}.
  Then one has
\begin{enumerate}
\item For every compact set $ K \subseteq [0, \infty ) $ there exist constants
  $ c(K), C(K) > 0  $ such that
  \begin{equ}
  \mathbf{p}_{t_{\star}}(j) \leqslant C (K) e^{- c(K) j} \;, \qquad
   \forall j \in \NN \;, t_{\star} \in K \;.
  \end{equ}
\item We have $\mathbf{p}_{t_{\star}}(j) = 0$ for all $t_{\star} \geqslant 3(j+1)$.
\end{enumerate}
\end{lemma}
\begin{proof}
  The second claim follows immediately from the fact that $ T_{i+1} -
  T_{i} \leqslant 3 $, as observed in Lemma~\ref{lem:consistency},
so we focus on the first claim. By definition we have that
\begin{equation*}
\begin{aligned}
T_{i +1} - T_{i} \geqslant V_{i +1} - T_{i} = \delta \wedge
(\sigma^{\mygeq}_{5/4}(T_{i +1}) - T_{i} ) \;,
\end{aligned}
\end{equation*}
where we recall the notation from \eqref{e:short}. Moreover, by
construction, since $ M_{T_{i}} \geqslant M( \pi_{T_{i}}) $ by
Lemma~\ref{lem:consistency}, we have at time $ T_{i} $ that $ \| w^{\mygeq}
(M_{T_{i}}; T_{i},  \cdot) \| \leqslant 1 $, so that the event \(
\mC_{E} (T_{i}) \) takes place with $ E =1 $.
In particular, it follows that by the second bound of Lemma~\ref{lem:neg-mmts},
there exists an $ \ve > 0 $ such that
\begin{equ} \label{e:sup-prob}
\PP_{T_{i}} (T_{i+1} - T_{i} > \ve ) \geqslant 1/2 \;, \qquad \forall i \in
\NN \;.
\end{equ}
Neither the value of $ \ve $ nor the value $ 1/2 $ will play any particular
role in the following, but we do need that the bound is uniform over $ i
$. We write
$\mathbf{p}_{t_{\star}}(j) \leqslant \PP ( T_{j} \leqslant t_{\star})$
and obtain an upper bound on the latter. To do so, we divide
the first $ j $ stopping times into batches of size $ N_{\star} =
\lfloor t_{\star}/ \ve \rfloor +1 $. Suppose that $ j \geqslant
N_{\star} (\l +1) $ for some $ \l \in \NN $. Then it holds that
\begin{equation*}
\begin{aligned}
\mH_{\l} \eqdef \left\{ T_{i+1} - T_{i} > \ve \;, \;
\forall i \in [N_{\star} \l , N_{\star} (\l +1)] \right\} \subseteq
\{ T_{j} > t_{\star} \} \;,
\end{aligned}
\end{equation*}
since we would have 
\begin{equation*}
\begin{aligned}
T_{j} \geqslant \ve N_{\star}> t_{\star} \;.
\end{aligned}
\end{equation*}
Now, the event $ \mH_{\l} $ happens with positive probability that is uniform
over $ j $ but depends on the value of $ t_{\star} $:
\begin{equ}
\PP_{T_{N_{\star} \l}} (\mH_{\l}) \geqslant 2^{- N_{\star}} \;.
\end{equ}
In particular, if we define
\begin{equation*}
\begin{aligned}
\l_{j} = \max \{ \l \in \NN  \; \colon \; N_{\star} (\l +1) \leqslant j \} \;, 
\end{aligned}
\end{equation*}
then we obtain the following bound:
\begin{equation*}
\begin{aligned}
\PP(T_{j} \leqslant  t_{\star}) \leqslant \PP ( \cap_{\l \leqslant \l_{j}}
\mH^{c}_{\l}) \leqslant (1 - 2^{- N_{\star}})^{\l_{j}} \;.
\end{aligned}
\end{equation*}
Since there exists a constant $ c (t_{\star}) $ such that $ \l_{j} \geqslant c
(t_{\star}) j $ for all $ j \in \NN $ the claim is proven: it is clear that all
estimates hold locally uniformly over $ t_{\star} $.
\end{proof}

\subsection{Proof of the discretised theorem}

The proof of Theorem~\ref{thm:lyap-func-discrete} will build on distinguishing
between the ``good'' event $ \mA_{i} $ (because by Lemma~\ref{lem:consistency} we
have the certainty that $ M_{i+1} = M_{i} -1 $) and the ``bad'' event $ \mB_{i}
$ (because it might be that $ M_{i+1} \geqslant M_{i} $). Since $ i $ will be
fixed throughout the proof, we will from now on refrain from writing it as a
index for these events.

To further clarify the structure of the proof, let us highlight three results,
appearing in later sections, which play a
fundamental role in the present proof. These are Lemma~\ref{lem:bd-uparrow-n},
Lemma~\ref{lem:switch-diluted} and Proposition~\ref{prop:reg-high-freq}.
Lemma~\ref{lem:bd-uparrow-n} is a technical result which estimates the jump
of the norm $ \| w_{t} \|_{M_{t} + k_{0}}^{2} $ at time $ t = T_{i+1} $: note
that at this time both the skeleton median $ M_{t} $ and the process $
w_{t}$ jump.

Lemma~\ref{lem:switch-diluted} is used when proving that the 
``bad'' event happens with small probability. It is fundamental
in our analysis, since it allows to treat the case in which the projective
dynamic reaches a concentrated state, so that the event $ \mf{m}_{S_{i+1}} =
\tC $ has small probability. The lemma guarantees
that even on such occasion the skeleton median has
an arbitrarily high probability of diluting \dash a consequence of the
high-frequency instability assumed in Theorem~\ref{thm:lyap-func}.
 
 Finally, Proposition~\ref{prop:reg-high-freq} yields high-frequency
regularity estimates: these estimates split into three groups, depending on how
quickly the next stopping time kicks in. The one in
\eqref{eqn:reg-dwn-n} relates to stopping times with a deterministic lower
bound (conditional on some event): this will be used to treat the event $ V_{i
+1} = T_{i}+ \delta \leqslant
\sigma_{5/4}^{\mygeq} (M_{i}, T_{i})  $. The uniform estimate in 
\eqref{eqn:uniform-bd} relates instead to stopping times that can kick in very
quickly, so that dissipation may not have time to improve the regularity of the initial datum:
therefore the estimates involves in the upper bound the regularity of the process $
w_{t}$ at the starting time. The estimate
\eqref{e:rupnew} relates instead to stopping times that kick in when
we observe upwards jumps of the median. These stopping times take a macroscopic
time to kick in, so that dissipation has in the meantime smoothed out the
initial datum: therefore, the eventual estimates are uniform over the initial
datum.

Finally, let us introduce (and in part recall) the following shorthand
notation:
  \begin{equ}[e:shorthand]
  M_{i} = M_{T_{i}}\;, \qquad  \Delta M_{i}  =  M_{i+1} - M_{i}\;, \qquad
\mf{m}_{i} = \mf{m}_{S_{i+1}} \;.
  \end{equ}
 Note that even though $t \mapsto M_t$ remains 
constant over all of $[T_i, T_{i+1})$, this is not the case for $\mf{m}_t$
since $u_{t}$ still evolves.

\begin{proof}[of Theorem~\ref{thm:lyap-func-discrete}]
  We prove the result for $ b =1 $, since other values of $ b $ can be treated
  identically.
  Our analysis is divided into 5 steps: one for $ \mA $, two for the event $
\mB \cap \{ \| \bpi_{T_{i}} \|
\geqslant  \K \} $, one for the event $\{ \| \bpi_{T_{i}} \|
<  \K \} $, and a final step where we address the uniform in time bounds on
$ \mf{G} $ appearing in points~2 and~3 of
Theorem~\ref{thm:lyap-func-discrete}.

Recall that we have four parameters at our disposal: $ \kappa_{0}, k_{0}$
from the definition of $ \mf{F} $, $\K$ which defines the size of
the ``center set'' outside of which we hope to see contraction of the Lyapunov
functional, and $ \delta $ from the definition of the padding stopping time $
V_{i+1} $. 

Let now $c>1$ be the constant appearing in Lemma~\ref{lem:bd-uparrow-n} and 
$H$ and $\bar H$ the functions appearing in \eqref{e:defH} and \eqref{e:defbarH} below.
We then first choose $\kappa_0$ large enough and then $\eps$ small enough so that
  \begin{equ}[eqn:for-kappa-0]
     H(1/2) e^{- \kappa_{0}+c/2} \leqslant \mf{c} \;,
     \quad
     \eps \bar H(2c\kappa_0+c) \leqslant \mf{c}^2\;.
  \end{equ}
We then use the fact that
it is possible to choose $ \delta \in (0, 1) $ small enough so that, almost surely,
  \begin{equ}[e:d]
     \PP_{T_{i}} (\mB^{V} ) \leqslant \ve \;,\qquad
     \mB^{V} = \{V_{i+1} < T_{i} + \delta\}\;.
  \end{equ}
This follows because $\mB^V$ coincides
with the event $ \{ \sigma_{5/4}^{\mygeq} (T_{i}) < T_{i} + \delta \}
$, which has small probability via Lemma~\ref{lem:neg-mmts} since at time $
T_{i} $ we have $ \| w^{\mygeq} (M_{i}; T_{i}, \cdot)\| \leqslant 1 < 5/4 $ by the
third point of Lemma~\ref{lem:consistency}. Note also that 
$ \mB = \mB^{V} \cup \mB^{S} $ with
\begin{equ}
\mB^{S} = \{ \sigma^{\tD} (V_{i+1}) >\sigma^{\mygeq} (V_{i+1})\}
\cup \{\tau^{\myl} (S_{i+1})> \tau^{\mygeq} (S_{i+1})\} \;.
\end{equ}
  Similarly to above, a combination of Lemma~\ref{lem:ext-time} and
  Lemma~\ref{lem:switch-diluted} guarantees that for $\K$ large enough
  \begin{equ}
    \PP_{V_{i+1}} ( \mB^{S}) \one_{\{ M_{i} \geqslant \K \} } \leqslant
\ve\one_{\{ M_{i} \geqslant \K \} }\;.\label{eqn:prf-main-eps}
  \end{equ}
Indeed, the probability of the event $\{\tau^{\myl} (S_{i+1})> \tau^{\mygeq}
(S_{i+1})\} $  is arbitrarily small via
Lemma~\ref{lem:ext-time}
 and the bound in the second point of
Lemma~\ref{lem:consistency}, conditional on the event $ \{
\sigma^{\tD} (V_{i+1}) \leqslant \sigma^{\mygeq} (V_{i+1})\} $.
On the other hand, the probability of the event $ \{
\sigma^{\tD} (V_{i+1}) >\sigma^{\mygeq} (V_{i+1})\} $ is small
by Lemma~\ref{lem:switch-diluted}, both provided that $ \K $ is
sufficiently large.

Combining both \eqref{eqn:prf-main-eps} and \eqref{e:d} we conclude that, for
$\ve > 0$ satisfying \eqref{eqn:for-kappa-0}, 
there are an upper bound on  $ \delta $
and a lower bound on $\K$ (depending on $ \mf{c}, \kappa_{0} $ and $ \ve
$ which at this stage are already fixed) which guarantee that 
\begin{equ}
    \PP_{T_i} ( \mB) \one_{\{ M_{i} \geqslant \K \} } \leqslant
\ve\one_{\{ M_{i} \geqslant \K \} }\;.\label{eqn:prf-main-eps-final}
\end{equ}

  \textit{Step 1: $ \mA$.} In this step we obtain an estimate on
$ \EE_{T_{i}} [ \mf{F} (\kappa, \bpi_{T_{i+1}}) \one_{\mA}] $. By
  construction, the skeleton median decreases on the event $ \mA $: $ \Delta M_{i}
  = -1 $ by Lemma~\ref{lem:consistency}. Therefore by the second bound in
Lemma~\ref{lem:bd-uparrow-n}
  \begin{equ}
    \EE_{T_{i}} \left[ \mf{F} (\kappa, \bpi_{T_{i+1}}) \one_{\mA}
    \right] \leqslant e^{- \kappa_{0} + c \kappa }\EE_{T_{i}} \left[  e^{ c \kappa \|
      w_{T_{i+1}-} \|^{2}_{M_{i}+ k_{0}}}\one_{\mA}
    \right]  e^{\kappa_{0} M_{i}}\;.
  \end{equ}
Observe that the quantity under the expectation on the right-hand side depends only on the left limit
of $ w $ at time $ T_{i+1} $. Now, the aim of this step will be to show that there exists an increasing function 
$\kappa \mapsto H(\kappa) > 1 $ such that
\begin{equ}[e:aim-first]
    \EE_{T_{i}} \Big[  e^{ c \kappa \|
      w_{T_{i+1}-} \|^{2}_{M_{i}+ k_{0}}}\one_{\mA}
    \Big]\leqslant H (\kappa) e^{ 4 c e^{- 2 \delta
    k_{0}} \kappa \| w_{T_{i}} \|_{M_{i}+ k_{0}}^{2}} \;.
\end{equ}
  To obtain this bound we first condition on time $ V_{i+1} $ and use the
estimate \eqref{eqn:uniform-bd} from Proposition~\ref{prop:reg-high-freq}:
\begin{equs}
  \EE_{V_{i+1}} \Big[   e^{ c \kappa \| w_{T_{i+1}-} \|^{2}_{M_{i}+
      k_{0}}}\one_{\mA} \Big] &\leqslant \EE_{V_{i+1}} \Big[   e^{ c \kappa \| w_{T_{i+1}-} \|^{2}_{M_{i}+
      k_{0}}} \Big]\one_{\mA^V} \\ &\leqslant \hat C(c\kappa)e^{ 2 c \kappa \| w_{V_{i+1}-} \|^{2}_{M_{i}+
      k_{0}}} \one_{\mA^{V}} \;,
\end{equs}
where we have set $\mA^{V} = \{ V_{i+1} = T_{i} + \delta  \}$.
We then take the expectation at time $
T_{i} $ and use \eqref{eqn:reg-dwn-n} to deduce the desired bound \eqref{e:aim-first} with
\begin{equ}[e:defH]
H(\kappa) = \hat C(c\kappa) C(2c\kappa)\;.
\end{equ}
We then choose $ k_{0}$ sufficiently large as a function of $\delta$ such that 
  \begin{equ}[eqn:for-k0]
    4 c e^{- 2 \delta k_{0}} \leqslant 1/2 \;.
  \end{equ}
Let us observe that \eqref{eqn:for-k0} is the only point where we require the parameter $
k_{0} $.
With this choice of parameters and recalling that $\kappa_0$ was chosen in such a way that 
\eqref{eqn:for-kappa-0} holds, we conclude from \eqref{e:aim-first} that provided that $\mf{h}$ is
chosen in such a way that
\begin{equ}[e:h-1]
\mf{h}(\kappa) \geqslant e^{c\kappa-c/2}H(\kappa)/ H(1/2)\;,
\end{equ}
one has the bound
    \begin{equs}[eqn:final-2]
    \EE_{T_{i}} \left[ \mf{F} (\kappa, \bpi_{T_{i+1}}) \one_{\mA}
    \right] & \leqslant \mf{c} \cdot \mf{h} (\kappa) \cdot e^{\kappa_{0} M_{i} + \frac\kappa2 \| w_{T_{i}}
    \|_{M_{i}+ k_{0}}^{2}} \\
& \leqslant  \mf{c} \cdot \mf{h}(\kappa) \cdot \mf{F}
    (\kappa /2 , \bpi_{T_{i}}) \;,
  \end{equs}
  which is an estimate of the required order. In addition, the map $
\mf{h} (\kappa) $ defined above satisfies the requirements of the theorem.

\textit{Step 2: $ \mB \cap \{ M_{i} \geqslant \K \} \eqdef
\mB_{\K}$.} Once
again, let us start by estimating the
conditional expectation at time $V_{i+1}$. We observe that although $
\mB $ is a ``bad'' event, we have no guarantee that
$ \Delta M_{i} \geqslant 0$.
Hence we combine both estimates of Lemma~\ref{lem:bd-uparrow-n} (the first one with $ L
=1 $, in order to keep the estimate uniform over $ k_{0} $), to find that for
some constant $ c > 1 $
  \begin{equs}
    \EE_{V_{i+1}}  \Big[ \mf{F} &(\kappa, \bpi_{T_{i+1}})
\one_{\mB_{\K}}
    \Big] e^{- \kappa_{0} M_{i}}  \\
    & \leqslant \EE_{V_{i+1}} \Big[ \exp \Big( c ( \kappa_{0} + \kappa) \left(   \|
      w_{T_{i+1}-}\|^{2}_{M_{i}+ 1} + 1  \right) \Big) \one_{
\mB_{\K} }
  \Big] \;,
  \end{equs}
where we have in addition used that $ \| w \|_{M_{i} + k_{0}} \leqslant \|
w \|_{M_{i} + 1} $.
Now, we know from \eqref{e:d} and 
\eqref{eqn:prf-main-eps} that $ \PP_{T_{i}} ( \mB_{\K})
\leqslant \ve $.
  Therefore, via Cauchy--Schwarz and the mentioned bounds we obtain that
  \begin{equ}[eqn:niceBoundB]
    \EE_{T_{i}} \left[ \mf{F} (\kappa, \bpi_{T_{i+1}})
\one_{\mB_{\K }} \right]  
      \leqslant \sqrt{\ve} \EE_{T_{i}}\left[ e^{ 2 c (\kappa_{0 } + \kappa )  \|
       w_{T_{i+1}-} \|^{2}_{M_{i}+ 1} }  \one_{\mB}
   \right]^{\frac{1}{2}} e^{c(\kappa_{0} + \kappa )} e^{\kappa_{0} M_{i}}\;.
  \end{equ}

We now note that $\mB \subset \CE \cup \bar \CE$ where 
\begin{equs}
\CE = \{\sigma^{\mygeq}_{5/4}(T_{i}) < S_{i+1}\}
\;,\quad
\bar \CE = \{\mf{m}_{S_{i+1} } =\tD\} \cap \{
\tau^{\myl} (S_{i+1}) > \tau^{\mygeq} (S_{i+1}) \}\;.
\end{equs}
Given two random variables $a$ and $b$, write $(a)_\CE(b)$ for the random variable equal to $a$ on $\CE$ and $b$ otherwise.
It then follows immediately from Definition~\ref{def:Tn} that, if we set $\beta = (5/4)_\CE(2)$, $L = (M_i)_\CE(M_i - 2)$,
as well as $t_0 = (T_i)_\CE(S_{i+1})$, then the stopping time\footnote{Although $t_0$ isn't a stopping time, one verifies that $t_1$ is.} 
$t_1 = \sigma^{\mygeq}_{\beta}(L,t_0)$ satisfies $t_1 \leqslant T_{i+1}$ on $\mB$ so that, for any $\bar \kappa > 0$,
 \eqref{eqn:uniform-bd} yields
\begin{equ}[e:boundt1]
 \EE_{t_1} \left[ e^{\bar \kappa \|
       w_{T_{i+1}-} \|^{2}_{M_{i}+ 1}}  \one_{\mB} \right]  
       \leqslant \hat C (\bar \kappa)   e^{2\bar \kappa  \|
       w_{t_1} \|^{2}_{M_{i}+ 1}}
\one_{\{t_1 \leqslant T_{i+1}\}}  \;.
\end{equ}

We can then apply $\EE_{t_0}$ to both sides: observe that $
t_{0} $ is not a stopping time, but for any two stopping times $ \tau, \sigma $
we define the expectation $\EE_{\tau_{\CE} \sigma } \left[ X \right] = \EE_{\tau} [ X \one_{\CE}] +
\EE_{\sigma} [X \one_{\CE^{c}}] $, which still satisfies the 
``tower property'' $\EE_{T_i} = \EE_{T_i}\EE_{t_0}$. In particular, one has the identity
\begin{equ}
\EE_{T_i} X = \EE_{T_i} \EE_{t_1} X = \EE_{T_i} \EE_{t_0} \EE_{t_1} X\;,
\end{equ}
so that, by \eqref{e:boundt1}, one has
\begin{equ}
\EE_{T_i} \left[ e^{\bar \kappa \|
       w_{T_{i+1}-} \|^{2}_{M_{i}+ 1}}  \one_{\mB} \right]  
       \leqslant \hat C (\bar \kappa) \EE_{T_i} \EE_{t_0}\Big[ e^{2\bar \kappa  \|
       w_{t_1} \|^{2}_{M_{i}+ 1}}
\one_{\{t_1 \leqslant T_{i+1}\}}\Big]\;.
\end{equ}
 Hence, after applying $ \EE_{t_{0}} $ we can use 
the bound \eqref{e:rupnew} since, 
by the second step of Lemma~\ref{lem:consistency} (and $E$ as appearing there) and the fact that 
$M_i \ge M(u_{T_i})$, we have $ \| w^{\mygeq} (L; t_0, \cdot) \|
\leqslant (1)_\CE(E) < (5/4)_\CE(2)$. We conclude that there exists an increasing function $\bar H$ such
that
\begin{equ}[e:defbarH]
\EE_{T_i} \left[ e^{\bar \kappa\|
       w_{T_{i+1}-} \|^{2}_{M_{i}+ 1}}  \one_{\mB} \right] \leqslant e^{-\bar \kappa}\bar H(\bar \kappa) \;.
\end{equ}
Inserting this back into \eqref{eqn:niceBoundB}, we see that
\begin{equ}
\EE_{T_{i}} \left[ \mf{F} (\kappa, \bpi_{T_{i+1}})
\one_{\mB_{\K }} \right] 
\leqslant \sqrt{\eps \bar H(2c(\kappa_0+\kappa))} \mf{F}(0, \bpi_{T_{i}}) \;.
\end{equ}
Since we assumed that $\delta$ is chosen sufficiently small so that \eqref{eqn:for-kappa-0} holds, 
we conclude that, provided that we choose $\mf{h}$ large enough so that
\begin{equ}[e:h-2]
\mf{h}(\kappa)^2 \geqslant \bar H(2c(\kappa_0+\kappa))/\bar H(2c\kappa_0+c)\;, 
\end{equ}
we have the bound
\begin{equ}[e:f3]
\EE_{T_{i}} \left[ \mf{F} (\kappa, \bpi_{T_{i+1}})
\one_{\mB_{\K}}  \right] \leqslant \mf{c} \cdot \mf{h} (\kappa)
\cdot  \mF (0, \bpi_{T_{i}}) \leqslant  \mf{c} \cdot \mf{h} (\kappa) \cdot \mF
( \kappa /2 , \bpi_{T_{i}}) \;.
\end{equ}
  
  \textit{Step 3: $ \mB \cap \{ \|
  w_{T_{i}} \|_{M_{i}+k_{0}} \geqslant \K \} \eqdef
\widetilde{\mB}_{\K} $.} Here we follow
  a slightly different estimate than in the previous step, since we are not
allowed to use \eqref{eqn:prf-main-eps}, as $ M_{i}  $ is not necessarily large.
Therefore we are not able to obtain a small factor $ \ve $ because of an event that
happens with small probability. Instead, we will gain a small factor, provided
that $ \K $ is sufficiently large. 
We start once more by applying
  Lemma~\ref{lem:bd-uparrow-n} (using both bounds, the first one with $ L =1
$) and the fact that $ \| w \|_{M_{i} + k_{0}} \leqslant \| w \|_{M_{i} +1}$,
so that on the event $ \mB $ we have:
  \begin{equ}
    \mf{F} (\kappa, \bpi_{T_{i+1}}) e^{- \kappa_{0} M_{i}}\leqslant \exp
\left( c (\kappa_{0} + \kappa)
      ( \| w_{T_{i+1}-}\|^{2}_{M_{i}+1} + 1) \right)\;.
  \end{equ}
  Therefore, by \eqref{e:defbarH}, we obtain the bound
  \begin{equs}
    \EE_{T_{i}} \Big[ \mf{F} (\kappa, \bpi_{T_{i+1}}) \one_{
\widetilde{\mB}_{\K}} \Big] e^{- \kappa_{0} M_{i}} 
       & \leqslant  e^{c ( \kappa_{0} + \kappa)} \EE_{T_{i}} \left[
      e^{ c(\kappa_{0} + \kappa)  \| w_{T_{i+1}-} \|_{M_{i+1}}^{2}
    }\one_{ \widetilde{\mB}_{\K} }\right] \\
       & \leqslant \bar{H}(c(\kappa_{0} + \kappa)) \one_{\{ \|
  w_{T_{i}} \|_{M_{i}+k_{0}} \geqslant \K \}}\;.
  \end{equs}
Since $ \kappa \geqslant 1/2 $, we can conclude that
\begin{equs}
\EE_{T_{i}}  \left[ \mf{F} (\kappa, \bpi_{T_{i+1}}) \one_{
\widetilde{\mB}_{\K}} \right] 
 &\leqslant \bar{H}( c(\kappa_{0}+ \kappa))  e^{- \frac{1}{4} \K^{2}} \mf{F}
(\kappa/2, \bpi_{T_{i}}) \one_{\{ \| w_{T_{i}} \|_{M_{i}+k_{0}} \geqslant \K \}} \\
& \leqslant \mf{c} \cdot \mf{h} (\kappa) \cdot \mf{F} (\kappa,
\bpi_{T_{i}})  \;,\label{e:f4}
\end{equs}
  provided that
\begin{equ}[e:h-3]
\mf{h}(\kappa) \geqslant \bar{H}(c(\kappa_{0} + \kappa))/ \bar{H}(c (\kappa_{0}
+ 1/2)) \;,
\end{equ}
and by choosing a $ \K(\ve, \kappa_{0}, \mf{c})
> 0 $ such that in addition to \eqref{eqn:prf-main-eps} also the following
holds:
  \begin{equ}
    \bar{H}(c (\kappa_{0} +1/2))  e^{- \frac{1}{4} \K^{2}} \leqslant
    \mf{c} < 1 \;.
  \end{equ}
  As in the previous steps, we have therefore obtained a bound of the required order.

  \textit{Step 4: The case $ \| \bpi_{T_{i}} \| < \K $.} Also this
estimate follows along the lines of the previous steps: indeed the proof is even simpler since we
are not interested in obtaining the contraction constant $ \mf{c} $. We start
as usual by applying Lemma~\ref{lem:bd-uparrow-n}:
  \begin{equs}
    \EE_{T_{i}} & \left[ \mf{F}(\kappa, \bpi_{T_{i+1}})\right] e^{-
\kappa_{0} M_{i}} \one_{\{ \| \bpi_{T_{i}} \| <\K \}} \\
    & \leqslant e^{c (\kappa_{0} + \kappa)}\EE_{T_{i}} \left[ \exp \Big( c
(\kappa_{0} + \kappa)  \|
      w_{T_{i+1}-}\|^{2}_{M_{i}+ 1}  \Big) \right]\one_{\{ \| \bpi_{T_{i}} \|
    <\K \}} \;.
\end{equs}
Now, via the uniform estimate \eqref{eqn:uniform-bd} in
Proposition~\ref{prop:reg-high-freq} we obtain
\begin{equs}
\EE_{T_{i}} \left[ \exp \Big( c (\kappa_{0} + \kappa)  \|
      w_{T_{i+1}-}\|^{2}_{M_{i}+ 1}  \Big) \right] \leqslant \hat{C} ( c
(\kappa_{0} + \kappa)) e^{2 c (\kappa_{0} + \kappa) \| w_{T_{i}-}
\|^{2}_{M_{i}+1}} \;,
\end{equs}
which yields as desired that for some $ J (\kappa_{0}, \kappa, \K) \in
(0, \infty) $
\begin{equation} \label{e:f5}
\begin{aligned}
\EE_{T_{i}}  \left[ \mf{F}(\kappa, \bpi_{T_{i+1}})\right] \one_{\{ \|
\bpi_{T_{i}} \| <\K \}} \leqslant
J (\kappa) \one_{\{ \| \bpi_{T_{i}} \| <\K
\}} \;,
\end{aligned}
\end{equation}
as required. This concludes the proof of the contraction estimate
\eqref{eqn:contraction}, since we can combine \eqref{eqn:final-2},
\eqref{e:f3}, \eqref{e:f4} and \eqref{e:f5} to obtain
\begin{equation*}
\begin{aligned}
\EE_{T_{i}} \left[ \mF(\kappa, \bpi_{T_{i +1}}) \right] \leqslant
4 \mf{c} \cdot \mf{h} (\kappa)  \mF (\kappa /2, \bpi_{T_{i}}) \one_{ \{
  \| \bpi_{T_{i}} \| \geqslant \K \} } +  J(\kappa)  &\one_{\{
  \| \bpi_{T_{i}}\|  < \K \}}\;,
\end{aligned}
\end{equation*}
for any $ \mf{h} $ satisfying \eqref{e:h-1}, \eqref{e:h-2} and \eqref{e:h-3},
and such that $ \mf{h} (1/2) =1 $.
This is the desired bound, since $ \mf{c} \in (0, 1) $ can be chosen
arbitrarily small.

  \textit{Step 5: Uniform estimates.} It now remains to obtain
\eqref{e:unif} and the bound at time $ t=0 $. The uniform estimate
\eqref{e:unif} follows along similar
  lines as the estimates above, using the uniform bound \eqref{eqn:uniform-bd}
  from Proposition~\ref{prop:reg-high-freq}. 
  Our first objective is therefore to obtain some deterministic estimates which reduce
the problem to an exponential bound of the kind treated in
Proposition~\ref{prop:reg-high-freq}. For $ s \in [T_{i}, T_{i+1}) $ we have that
  \begin{equ}
    \mf{G} (\pi_{s}) = \mf{F} ( \kappa, \bpi_{T_{i}}) \mf{B}(\kappa, \bpi_{T_{i}},
    \pi_{s}) \;,
  \end{equ}
  where
  \begin{equ}
    \mf{B}(\kappa, \bpi_{T_{i}}, \pi_{s})  = \exp \left( \kappa_{0} (M (\pi_{s}) -
    M_{i}) + \| \pi_{s} \|^{2}_{M (\pi_{s}) + k_{0}} - \kappa \|
  w_{T_{i}} \|^{2}_{M_{i} + k_{0}}  \right) \;.
  \end{equ}
  Now since $ \pi_{s} \in S $ we can use Lemma~\ref{lem:est-reg-jump-n} to bound
  \begin{equs}
    \| \pi_{s} \|_{M (\pi_{s}) + k_{0}}^{2} & \leqslant \|
    \Pi^{\myg}_{M_{i}+ k_{0}} \pi_{s} \|_{M_{i} + k_{0}}^{2} + 4
(\nu_{\mathrm{min}}^{- \frac{1}{
2 \a }} +1)^{2} ( M(\pi_{s})
    - M_{i})_{-} \\
    & \leqslant \| \Pi^{\myl}_{M_{i}} \pi_{s} \|^{2} \| 
    w_{s} \|_{M_{i}+ k_{0}}^{2} + 4(\nu_{\mathrm{min}}^{- \frac{1}{
2 \a }} +1)^{2}( M(\pi_{s})
    - M_{i})_{-}  \\
    & \leqslant  \| 
    w_{s} \|_{M_{i} + k_{0}}^{2} + 4(\nu_{\mathrm{min}}^{- \frac{1}{
2 \a }} +1)^{2}( M(\pi_{s}) - M_{i})_{-} \;.
  \end{equs}
  In particular, we can use this bound to estimate, with $ c_{1} =4 (\nu_{\mathrm{min}}^{- \frac{1}{
2 \a }} +1)^{2} $:
  \begin{equs}
    \mf{B}(\kappa, \bpi_{T_{i}}, \pi_{s} ) \leqslant \exp \Big(\kappa_{0} (M (\pi_{s})  -
      M_{i})_{+} &- (\kappa_{0} - c_{1}) (M (\pi_{s}) - M_{i})_{-} \\
      & + \| w_{s} \|^{2}_{M_{i} + k_{0}} - \kappa \| w_{T_{i}}
      \|^{2}_{M_{i} + k_{0}}  \Big) \;.
  \end{equs}
  Further, via Lemma~\ref{lem:est-mid-jump-n} and once more
via Lemma~\ref{lem:est-reg-jump-n}:
  \begin{equs}
    (M(\pi_{s}) - M_{i})_{+} & \leqslant (M (\Pi_{M_{i}}^{\myg} \pi_{s}) -
    M_{i})_{+} \\
    & \leqslant  4 \nu_{\mathrm{max}}^{\frac{1}{2 \a }} \| \pi_{s} \|_{M_{i}}^{2}   
     \leqslant c_{2}  \left(  \| \pi_{s} \|_{M_{i} + k_{0} }^{2} + k_{0} \right)  \\
    & \leqslant c_{2}  \left( \| w_{s} \|_{M_{i} + k_{0} }^{2} +   k_{0} \right)  \;.
  \end{equs}
  Hence overall, since we can assume that $
\kappa_{0} \geqslant c_{1} $, we obtain that
  \begin{equs}
    \mf{B}(\kappa, \bpi_{T_{i}}, \pi_{s})\leqslant e^{ \kappa_{0}
 k_{0} c_{2}}  \exp \left( c_{2}\|
  w_{s} \|^{2}_{M_{i} + k_{0}} - \kappa \| w_{T_{i}}
\|^{2}_{M_{i} + k_{0}} \right)\;.
  \end{equs}
This estimate allows us to conclude, since we can apply \eqref{eqn:uniform-bd}
from Proposition~\ref{prop:reg-high-freq} three times to obtain that for 
  for $ \kappa \geqslant \kappa_{1} = 2^{3}c_{2}  $ and $
  A_{i} = [T_{i}, T_{i+1}) $ one has
\begin{equ}
  \EE_{T_{i}} \left[ \sup_{s \in A_{i}} \mf{C}(\kappa, \bpi_{T_{i}},
  \pi_{s}) \right] \leqslant C(\kappa_{0}, k_{0}, \delta) \;,
\end{equ}
which proves the desired bound.

  To conclude, we turn to the estimate at time $ t =0 $. Here we use that $
  M_{0} = M( \pi_{0}) $ to find that
  \begin{equ}
    \mf{F}(\kappa, \bpi_{0}) = \mf{G}(\pi_{0}) \mf{D}( \kappa, 
    \pi_{0}) \;,
  \end{equ}
  with 
  \begin{equ}
    \mf{D}(\kappa, \pi_{0}) = \exp \left( \kappa \| w_{0} \|^{2}_{M
      (\pi_{0}) + k_{0}} - \| \pi_{0} \|^{2}_{M(\pi_{0})+ k_{0}} \right) \;.
  \end{equ}
  By the definition of $ M (\pi_{0}) $ we know that $ 2^{-1} \leqslant \|
  \Pi^{\myl}_{M(\pi_{0})} \pi_{0} \|^{2} $, from which we deduce
  \begin{equ}
    \| \pi_{0} \|^{2}_{M(\pi_{0})+ k_{0}}   = \| \Pi^{\myl}_{M (\pi_{0}) } \pi_{0} \|^{2} \|
    w_{0} \|^{2}_{M ( \pi_{0}) + k_{0}}  \geqslant 2^{-1} \| w_{0} \|^{2}_{M
    (\pi_{0}) + k_{0}} \;,
  \end{equ}
  implying that $\mf{D}( \kappa, \pi_{0}) \leqslant 1$ if $\kappa \leqslant 1/2$ as
  required.
\end{proof}

\subsection{Regularity to median estimates}

In this subsection we state elementary lemmas that relate jumps of the
energy median to high frequency regularity and vice versa.

\begin{lemma}\label{lem:est-reg-jump-n}
  Consider two values $ L^{+}, L^{-} \in \NN $ and write $ \Delta L =
  L^{+} - L^{-} $. Then for any $ \varphi \in S $ and $ \gamma \geqslant 1/2 $
  \begin{equs}
    \| \varphi \|_{\gamma, L^{+}}^{2} \leqslant 
  \begin{cases}
    \| \varphi \|_{\gamma, L^{-}}^{2} & \text{ if } \Delta L \geqslant 0 \;, \\
    2^{2 \gamma }(\nu_{\mathrm{min}}^{- \frac{\gamma}{
\a }} +1 )(\Delta L)_{-} + 2^{2 \gamma -1} \| \varphi \|_{\gamma,
    L^{-}}^{2} & \text{ if } \Delta L < 0 \;,
  \end{cases}
  \end{equs}
  with $ (\Delta L)_{-} $ the negative part of $ \Delta L $.
 \end{lemma}
The restriction $ \gamma \geqslant 1/2 $ is not necessary: for $ \gamma < 1/2 $
the parameters of the estimate change only slightly, but such an estimate is not
required.

\begin{proof}
  Let us start with the case $ \Delta L \geqslant 0 $. From the definition of the norms
  we have for any $ \varphi \in S $ 
  \begin{equs}
    \| \varphi &\|_{\gamma, L^{+}}^{2} -  \| \varphi
    \|_{\gamma, L^{-}}^{2} \\
    & = \sum_{\alpha=1}^{m}  \left( \sum_{| k | > L^{+}_{\alpha}} (| k | - L^{+}_{\alpha} + 1
    )^{2 \gamma}|
    \hat{\varphi}^{\alpha}_{k}|^{2} - \sum_{| k | > L^{-}_{\alpha}} ( | k |
-L^{-}_{\alpha} + 1)^{2 \gamma} |
    \hat{\varphi}_{k}^{\alpha}  |^{2} \right) \leqslant 0 \;. 
  \end{equs}
  On the other hand if $ \Delta L < 0 $, since $ \| \varphi \| =1 $
  \begin{equs}
    \| \varphi &\|_{L^{+}}^{2} -   2^{2 \gamma -1} \| \varphi \|_{L^{-}}^{2}
    \\
    & = \sum_{\alpha =1}^{m} \left(  \sum_{|
    k | > L^{+}_{\alpha}} (| k | - L^{+}_{\alpha} +1 )^{2 \gamma}|
\hat{\varphi}_{k}^{\alpha} |^{2} -
     2^{2 \gamma -1} \sum_{| k | >
    L^{-}_{\alpha}} ( | k | -L^{-}_{\alpha} +1 )^{2 \gamma} |
\hat{\varphi}_{k}^{\alpha} |^{2}\right)  \\  
    &  \leqslant 2^{2 \gamma -1}\nu_{\mathrm{min}}^{- \frac{\gamma}{
\a }}  ( \Delta L )_{-}^{2 \gamma} + 2^{2 \gamma -1} \sum_{\alpha
=1}^{m} \sum_{L^{+}_{\alpha}  < | k | \leqslant L^{-}_{\alpha}} (| k
    | - L^{+}_{\alpha}+1)^{2 \gamma}  | \hat{\varphi}_{k}^{\alpha} |^{2} \\
    & \leqslant 2^{2 \gamma }  [\nu_{\mathrm{min}}^{- \frac{\gamma}{
\a }}( \Delta L )_{-} +1] \\
    & \leqslant  2^{2 \gamma }(\nu_{\mathrm{min}}^{- \frac{\gamma}{
\a }} +1 ) ( \Delta L )_{-} \;,
  \end{equs}
  where we used that $ (\Delta L)_{-} \geqslant 1 $ together with 
  \begin{equ}
    (a - L^{+}_{\alpha})^{2 \gamma} = (a - L^{-}_{\alpha} + (\Delta
L_{\alpha})_{-})^{2 \gamma} \leqslant
    2^{2 \gamma -1 } \left[ (a - L^{-}_{\alpha})^{2 \gamma} + (\Delta
L_{\alpha})_{-}^{2 \gamma} \right] \;,
  \end{equ}
  for all $ a \geqslant L^{-} $.
  This concludes the proof.
\end{proof}
An analogous estimate provides a bound on the size of the energy median through high
frequency regularity.

\begin{lemma}\label{lem:est-mid-jump-n}
  Consider any $ L \in \NN $. Then for every $
\varphi \in S $ satisfying $ \| \Pi^{\myg}_{L} \varphi \| > 0 $ the following
estimate holds:
  \begin{equ}
    (M ( \Pi^{\myg}_{L} \varphi) - L)^{\frac{1}{2} } \leqslant 2
\nu_{\mathrm{max}}^{\frac{1}{4 \a }} \| \varphi
    \|_{L} \;.
  \end{equ}
\end{lemma}
\begin{proof}
Let us write $ M = M ( \Pi^{\myg}_{L} \varphi) $ for short and note that $
M > L $, so that $ M-1 \geqslant L$, since we are working with integers.
Therefore, from the definition of the
  energy median
  \begin{equs}
    \| \varphi \|_{L} & \geqslant \bigg( \sum_{\alpha =1}^{m}\sum_{| k
| > M_{\alpha} -1} ( | k | - L_{\alpha}+1) |
  \hat{\varphi}_{k}^{\alpha} |^{2} \bigg)^{\frac{1}{2}}\\
  & \geqslant \bigg(  \sum_{\alpha =1}^{m} (M_{\alpha} - L_{\alpha}) \|
\Pi^{\myg}_{M-1} \varphi^{\alpha} \|^{2} \bigg)^{\frac{1}{2}} \geqslant
  \frac{1}{2 \nu_{\mathrm{max}}^{\frac{1}{4 \a }}} ( M - L)^{\frac{1}{2}}\;.
  \end{equs}
  Hence, the result is proven.
\end{proof}
The next result considers the jump of the process \( w_{t} =
\Pi_{M_{t}}^{\myg}  u_{t} / \| \Pi^{\myl}_{M_{t}} u_{t}  \| \) at the jump times $
T_{i} $: it is a slight adaptation of Lemma~\ref{lem:est-reg-jump-n}.

\begin{lemma}\label{lem:h-jump-n}
Under the assumptions of Theorem~\ref{thm:lyap-func} and with the stopping
times and skeleton median from Definition~\ref{def:Tn} the following holds.
  For every $ i \in \NN $ and $ \Delta M_{i} = M_{i+1} - M_{i} =
  M_{T_{i+1}} - M_{T_{i}} $ we have the following estimates for any $
  k_{0} \in \NN $ and $ \gamma \geqslant 1/2 $
  \begin{equ}
    \| w_{T_{i+1}} \|_{\gamma, M_{i+1} + k_{0}}^{2} \leqslant \begin{cases} 
      \| w_{T_{i+1}-} \|_{\gamma, M_{i} + k_{0}}^{2} & \text{if } \Delta M_{i} \geqslant 0
      \,, \\
      5 \cdot 2^{2 \gamma} \left[ \nu_{\mathrm{min}}^{- \gamma / \a} +1 
 + 2^{ - 1}  \| w_{T_{i+1}-} \|_{\gamma, M_{i}+ k_{0}}^{2} \right] &
      \text{if } \Delta M_{i} < 0 \,.
    \end{cases} 
  \end{equ}

\end{lemma}

\begin{proof}
  Let us start with the first estimate, in the case $ \Delta M_{i} \geqslant 0
  $. In this case we find that $ \| \pi_{T_{i+1}} \|_{\gamma, M_{i+1}+ k_{0}} \leqslant \|
  \pi_{T_{i+1}} \|_{\gamma, M_{i}+ k_{0}} $, as well as $ \| \Pi^{\myl}_{M_{i}}
\pi_{T_{i+1}} \| \leqslant \| \Pi^{\myl}_{M_{i+1}}
  \pi_{T_{i+1}} \| $. Hence, in
particular, we conclude that
  \begin{equ}
    \| w_{T_{i+1}} \|_{\gamma, M_{i+1}+ k_{0}} = \frac{\| \pi_{T_{i+1}}
\|_{\gamma, M_{i+1} + k_{0}}}{\| \Pi^{\myl}_{M_{i+1}} \pi_{T_{i+1}} \|} \leqslant \frac{\| \pi_{T_{i+1}}
    \|_{\gamma, M_{i}+ k_{0}}}{\| \Pi^{\myl}_{M_{i}} \pi_{T_{i+1}} \|}= \|
    w_{T_{i+1} -} \|_{\gamma, M_{i}+ k_{0}} \;.
  \end{equ}
  On the other hand, in the case \( \Delta M_{i} < 0 \), which by
  Definition~\ref{def:Tn} of the skeleton median process $ (M_{t})_{t
\geqslant 0} $ implies that the jump is exactly of size one, i.e.\ $ \Delta
M_{i} = -1 $, we find by
  Lemma~\ref{lem:est-reg-jump-n} and \eqref{e:bd-pi}
  \begin{equs}
    \| w_{T_{i+1}} \|_{\gamma, M_{i+1}+ k_{0}}^{2}   &= \frac{\| \pi_{T_{i+1}}
    \|_{\gamma, M_{i+1}+ k_{0}}^{2} }{ \| \Pi^{\myl}_{M_{i+1}} \pi_{T_{i+1}}
\|^{2}} \\
    & \leqslant 2^{2 \gamma} (\nu_{\mathrm{min}}^{- \frac{\gamma}{
\a }} +1 )\frac{| \Delta M_{i} |}{ \| \Pi^{\myl}_{M_{i+1}} \pi_{T_{i+1}}
    \|^{2}} + 2^{2 \gamma -1} \frac{\| \pi_{T_{i+1}} \|_{\gamma, M_{i}+
k_{0}}^{2} }{ \| \Pi_{M_{i+1}}^{\myl}  \pi_{T_{i+1}} \|^{2}} \\
    & \leqslant 2^{2 \gamma} (\nu_{\mathrm{min}}^{- \frac{\gamma}{
\a }} +1 ) \cdot 5 + 2^{2 \gamma -1} \frac{\| \pi_{T_{i+1}} \|_{\gamma, M_{i}+
    k_{0}}^{2} }{ \| \Pi_{M_{i+1}}^{\myl}  \pi_{T_{i+1}} \|^{2}}\;.
  \end{equs}
  We can also estimate via \eqref{e:bd-pi}
  \begin{equ}
    \frac{\| \Pi_{M_{i}}^{\myl}  \pi_{T_{i+1}} \|^{2}}{ \| \Pi_{M_{i+1}}^{\myl} 
    \pi_{T_{i+1}} \|^{2}} \leqslant 5 \;,
  \end{equ}
  so that overall
  \begin{equ}
    \| w_{T_{i+1}} \|_{\gamma, M_{i+1}+ k_{0}}^{2} \leqslant 2^{2 \gamma} (\nu_{\mathrm{min}}^{- \frac{\gamma}{
\a }} +1 )
    \cdot 5 + 2^{2 \gamma -1} \cdot 5  \frac{\| \pi_{T_{i+1}} \|_{\gamma,
M_{i}+ k_{0}}^{2} }{ \| \Pi^{\myl}_{M_{i}}
    \pi_{T_{i+1}} \|^{2}} \;,
  \end{equ}
  which completes the proof of the lemma.
\end{proof}

Now, a combination of the previous estimates delivers the following lemma.

  \begin{lemma}\label{lem:bd-uparrow-n}
    Under the assumptions of Theorem~\ref{thm:lyap-func} and with the stopping
times and skeleton median from Definition~\ref{def:Tn}, there exists a constant $
    c >1$ such that, uniformly over all $ L, i \in
\NN $ and $\kappa, \kappa_0 > 0$:
      \begin{enumerate}
      \item If $ \Delta M_{i} \geqslant 0 $, then:
    \begin{equs}
      \mf{F}(\kappa, &\bpi_{T_{i+1}}) e^{- \kappa_{0} M_{i}} 
 \leqslant \exp\left( c  \kappa_{0}
        \left(  \| w_{T_{i+1} -} \|_{M_{i}+ L}^{2} + 
        L \right) +  \kappa \| w_{T_{i+1}-} \|_{M_{i}+
    k_{0}}^{2} \right)   \;.
    \end{equs}
      \item  If $ \Delta M_{i} < 0 $, then:
    \begin{equ}
     \mf{F}(\kappa, \bpi_{T_{i+1}}) e^{- \kappa_{0} M_{i}} \leqslant  \exp\left( -
      \kappa_{0}  + c \kappa ( \| w_{T_{i+1} -}
    \|_{M_{i}+ k_{0}}^{2}+1) \right)   \;.
    \end{equ}
    \end{enumerate}
  \end{lemma}

  \begin{proof}
We start by showing that the following estimate holds uniformly over $ L, i \in
\NN$:
  \begin{equ}[e:extra1]
      \Delta M_{i} \leqslant c \| w_{T_{i+1}-} \|_{M_{i}+  L}^{2} + c
      L \;.
    \end{equ}
    This bound is only non-trivial when $\Delta M_{i} > 0$ so we assume that this is the case.
    Recall that $ w_{T_{i+1}-} =
    \Pi^{\myg}_{M_{i}} \pi_{T_{i+1}} / \| \Pi^{\myl}_{M_{i}} \pi_{T_{i+1}} \|$ and that,
    because of the presence of the projection $ \Pi^{\myg}_{M_{i}} $, we have
by definition that $
M_{i+1} \leqslant M( w_{T_{i+1}-} ) $. Hence
    applying first Lemma~\ref{lem:est-mid-jump-n} (with $L = M_i$ and $\varphi = \pi_{T_{i+1}}$) and then
    Lemma~\ref{lem:est-reg-jump-n} we obtain
    \begin{equs}
      \Delta M_{i} &  = \left( \sqrt{\Delta M_{i} } \right)^{2} \\
    & \leqslant c(\nu_{\mathrm{max}}) \|\pi_{T_{i+1}}\|_{M_{i}}^{2} \leqslant c(\nu_{\mathrm{max}}) \| w_{T_{i+1}-} \|_{M_{i}}^{2}  \\
      & \leqslant c(\nu_{\mathrm{max}})  \left( \| w_{T_{i+1}-} \|_{M_{i}+ L}^{2} + (\|
      w_{T_{i+1}-} \|_{M_{i}}^{2} - \| w_{T_{i+1}-} \|_{M_{i}+
    L}^{2}) \right) \\
    & \leqslant c(\nu_{\mathrm{max}}, \nu_{\mathrm{min}}) \left(  \| w_{T_{i+1}-} \|_{M_{i}+ L}^{2} + L \|
    w_{T_{i+1}-} \|^{2} \right)  \\
    & \leqslant  c(\nu_{\mathrm{max}}, \nu_{\mathrm{min}}) \left( \| w_{T_{i+1}-} \|_{M_{i}+ L}^{2} + 
     L\right) \;, \label{e:lastBound}
    \end{equs}
    where in the last line we have made use of the bound
\eqref{e:bd-pi} in Lemma~\ref{lem:consistency} since
\begin{equ}
\| w_{T_{i+1}-} \|= \lim_{t \uparrow T_{i +1}} \frac{\| \Pi_{M_{t}}^{\myg} 
\pi_{t} \|}{ \| \Pi_{M_{t}}^{\myl} \pi_{t} \|} \leqslant 2  \;,
\end{equ} 
thus proving \eqref{e:extra1}. Note that the
values of the constants $c$ in \eqref{e:lastBound} change from line to line.

At this point, the two desired estimates follow from Lemma~\ref{lem:h-jump-n}.
  If $ \Delta M_{i} \geqslant 0 $, then  we have $\| w_{T_{i+1}} 
  \|_{M_{i+1}+ k_{0}}^{2} \le \| w_{T_{i+1}-} \|_{M_{i}+ k_{0}}^{2}$, so that \eqref{e:extra1}
implies as desired
  \begin{equ}
    \mf{F}(\kappa, \bpi_{T_{i+1}}) e^{- \kappa_{0} M_{i}}
    \leqslant \exp \left( \kappa_{0} c  \|
      w_{T_{i+1} -}\|_{M_{i}+ L}^{2}  + \kappa_{0} c L + \kappa \| w_{T_{i+1}-} \|_{M_{i} + k_{0}}^{2}
  \right) \;.
  \end{equ}
    If instead $ \Delta M_{i} < 0 $ then $\Delta M_{i} = -1$ by definition and, by Lemma~\ref{lem:h-jump-n}, 
  we have 
  \begin{equ}
    \| w_{T_{i+1}} \|_{M_{i+1}+k_{0}}^{2} \leqslant c(\nu_{\mathrm{min}})(1 +   \| w_{T_{i+1}-}
    \|_{M_{i}+ k_{0}}^{2} )\;,
  \end{equ}
  which yields the desired bound.
\end{proof}

\subsection{Exit time bounds}\label{sec:exit-time}

In this subsection we collect estimates on certain exit times.
Given any stopping time $ t_{0} $ and any $ \mF_{t_{0}} $-adapted random
variables $ E \in (0, \infty) $ and $ L \in \NN $, let us define the following event:
\begin{equs}[e:sets-1n]
	\mD_{E} (L;t_{0}) & = \{ \| w^{\mygeq}   (L; t_{0}, \cdot) \|
	\leqslant E \}  \;.
\end{equs}
 Our first result is a negative exponential moment bound on
$ \sigma_{\beta}^{\mygeq} (L,t_{0}) - t_{0} $: recall the definition in
\eqref{eqn:tau-generic-n}.

\begin{lemma}\label{lem:neg-mmts}
Under the assumptions of Theorem~\ref{thm:lyap-func}, 
  fix any $ \kappa, E, \beta  > 0 $ such that $E \in
(0, \beta) $ and let $ t_{0} $ be any stopping
time. Then there exist a deterministic 
function $ (\zeta, \kappa, \beta, \delta) \mapsto C(\zeta, \kappa, \beta,
\delta) \in (0, \infty)$ such that uniformly over all parameters
\begin{equs}
  \EE_{t_{0}} & \left[ \exp \left(\kappa (\sigma_{\beta}^{\mygeq}(L, t_{0}) - t_{0})^{- \zeta}
    \right) \one_{\mD_{E}(L;t_{0})}
  \right]  < C(\zeta, \kappa, \beta, E ) \one_{\mD_{E}(L;t_{0})}\;.
\end{equs}
\end{lemma}
Let us highlight that
the estimate is not uniform over $ E \uparrow \beta $ and as a matter of fact fails at $
E \geqslant \beta $.

\begin{proof}

This result follows from Proposition~\ref{prop:drft-OU}, as a consequence of
Doob's submartingale inequality.
Indeed, consider $ R$ and $ \mN $ as in Proposition~\ref{prop:drft-OU}, so that we can
define for all $ t \in [t_{0}, \sigma_{\beta}^{\mygeq} (L, t_{0})] $
\begin{equ}
 x_{ t} = E+ R(\beta) \cdot (t - t_{0}) + \mN_{t}\;.
\end{equ}
For simplicity, we will consider $ \mN $ defined for all $ t \geqslant
t_{0} $ by setting $ \mN_{t} = \mN_{t \wedge \sigma_{\beta}^{\mygeq} (L,
t_{0})} $. By comparison, we then have that 
$ \sigma_{\beta}^{\mygeq}(L, t_{0}) \geqslant
\underline{\sigma}^{\mygeq}_{\beta}(L, t_{0})$, where the latter is defined by
\begin{equ}
\underline{\sigma}_{\beta}^{\mygeq}(L, t_{0}) = \inf \{ t \geqslant t_{0}  \ \colon \
x_{t} \geqslant \beta \} \;.
\end{equ}
Now fix any $ t_{\star}(\beta, \beta - E) > 0 $ satisfying $ 
R \, t_{\star} \leqslant  \frac{1}{2} ( \beta - E)$. 
Then for $ s \in [0, t_{\star}]  $ we find
\begin{equ}
\underline{ \sigma}_{\beta}^{\mygeq}(L, t_{0}) \leqslant t_{0} + s \ \ \Rightarrow \ \
\sup_{r \in [0, s]} \mN_{t_{0} + r} \geqslant  \frac{1}{2} ( \beta - E) \eqdef
\mu \;,
\end{equ}
so that overall for $ s \in [0, t_{\star}] $
\begin{equ}
  \PP_{t_{0}} ( \sigma_{\beta}^{\mygeq}(L, t_{0}) \leqslant  s )
  \one_{\mD_{E}(L; t_{0} )} 
  \leqslant \PP_{t_{0}} \left( \sup_{r \in [0, s]} \mN_{t_{0} + r} \geqslant \mu
\right)\one_{\mD_{E}(L;t_{0})}\;.
\end{equ}
We can now apply
Doob's submartingale inequality, since by
Proposition~\ref{prop:drft-OU} the quadratic variation of $ \mN $ satisfies $
\ud \langle \mN \rangle_{t} \leqslant R \ud t $ for all $ t \in  [t_{0} ,
\sigma_{\beta}^{\mygeq} (L; t_{0})] $, to obtain for any $ \lambda > 0 $
\begin{equs}
  \PP_{t_{0}}  \left( \sup_{r \in [0, s]} \mN_{t_{0} + r} \geqslant \mu
  \right) \one_{\mD_{E}(L;t_{0})} & = \PP_{t_{0}} \left( \sup_{r \in [0, s]}
    \exp( \lambda \mN_{t_{0} + r} ) \geqslant e^{\lambda \mu}
  \right) \one_{\mD_{E}(L;t_{0})}\\
& \leqslant \exp \left( -
\lambda \mu + \frac{\lambda^{2} s R}{2}  \right)
\one_{\mD_{E}(L;t_{0})}\;,
\end{equs}
so that choosing $ \lambda = \frac{\mu}{R s}  $ delivers eventually,
for $ s \in [0, t_{\star}] $, the bound
\begin{equ}
  \PP_{t_{0}} \left( \sup_{r \in [0, s]} \mN_{t_{0} + r} \geqslant \mu
  \right) \one_{\mD_{E}(L;t_{0})} \leqslant \exp \left( - \frac{ \mu^{2}}{2R}
  s^{-1}\right) \one_{\mD_{E}(L;t_{0})}\;,
\end{equ}
from which the required moment bound follows.
\end{proof}
Also the next result follows along classical lines from
Proposition~\ref{prop:drft-OU}. It guarantees a drift of the median to low
frequencies, in that the probability of the relative energy process about the
median becoming small is much higher than
that of increasing, provided the energy median is sufficiently large.
For this purpose, for any $ \K> 0, E \in (0, 2) $,  define $\mD_{E}^{\K}(t_{0}) =
\mD_{E}( M_{t_{0}} -2;t_{0} ) \cap \{ M_{t_{0}} > \K \}$. Recall also the
stopping times defined in \eqref{e:short}.

\begin{lemma}\label{lem:ext-time}
  Under the assumptions of Theorem~\ref{thm:lyap-func}, fix any $ E \in (0, 2) $ and let $ t_{0} $ be any stopping
  time. Then for every $ \ve \in (0, 1) $ there exists a $ \K (\ve, E) $
  such that 
\begin{equ}
  \PP_{t_{0}}  \big( \tau^{\myl}(t_{0}) < \tau^{\mygeq}(t_{0}) \big)\geqslant  1 - \ve  \;,
\end{equ}
on the event $ \mD^{\K}_{E}(t_{0}) $.
\end{lemma}

\begin{proof}
  Let $ \mM $ and $ R $ be as in Proposition~\ref{prop:drft-OU}, and define for
$ t \geqslant t_{0} $ (by setting $ \mM_{t} = \mM_{t \wedge \tau^{\mygeq}
(t_{0})} $ to define the martingale also for times larger than $
\tau^{\mygeq} $):
	\begin{equs}
		x_{t} &= E + R(t - t_{0}) + \mM_{t}\;.
	\end{equs}
  Then the stopping time $ \underline{\tau}^{\mygeq} (t_{0}) $ given by
\begin{equ}
\underline{\tau}^{\mygeq}(t_{0}) = \inf \left\{ t \geqslant t_{0}  \ \colon \
x_{t} \geqslant 2  \right\} \;,
\end{equ}
satisfies by Doob's submartingale inequality that for some constant $ C(\zeta,
E) $:
\begin{equ}
  \EE_{t_{0}} \left[ \exp \left( (\underline{\tau}^{\mygeq}(t_{0}) -
  t_{0})^{- \zeta} \right) \right] 1_{\mD^{\K}_{E}(t_{0})}  \leqslant C(\zeta, E)\;.
\end{equ}
Therefore, via the previous Lemma~\ref{lem:neg-mmts}, for any $ \ve \in
(0, 1) $ we can find a $ \delta
(\ve, E) \in (0, 1) $ such that on
the event $ \mD^{\K}_{E} (t_{0}) $ 
\begin{equ}
  \PP_{t_{0}} \left( \underline{\tau}^{\mygeq}(t_{0}) \leqslant t_{0} + \delta \right) \leqslant
  \ve \;.
\end{equ}
Now, on the event $ \{ \underline{\tau}^{\mygeq} (t_{0}) > t_{0} + \delta , \tau^{\mygeq}
(t_{0}) > t_{0} + \delta \} $ we find by
Proposition~\ref{prop:drft-OU} that if $ \tau^{\myl}(t_{0}) > t_{0} + \delta $, then
\begin{equ}
  \| w (M_{t_{0}}-2; t_{0} + \delta, \cdot)  \|^{2} \leqslant -
\nu_{\mathrm{\min} }\delta
  \Delta_{M_{t_{0}}-2} + x_{\delta} \leqslant - \nu_{\mathrm{min}} \delta \Delta_{M_{t_{0}}-2} + 2
  \leqslant \frac{1}{4} \;,
\end{equ}
where the last inequality holds provided $ \K $ is
chosen sufficiently large depending on $ \delta $, and thus on $ \ve $ and $
E $. This proves that $$ \PP_{t_{0}} (\underline{\tau}^{\mygeq} (t_{0}) > t_{0} + \delta , \tau^{\mygeq}
(t_{0}) > t_{0} + \delta, \tau^{\myl} (t_{0}) > t_{0} + \delta )= 0 $$ on $
\mD^{\K}_{E} (t_{0}) $, so that
\begin{equ}
  \PP_{t_{0}} ( \tau^{\myl}(t_{0}) < \tau^{\mygeq}(t_{0}))\one_{
  \mD^{\K}_{E} (t_{0}) } \geqslant (1 -
    \ve)\one_{ \mD^{\K}_{E} (t_{0}) } \;,
\end{equ}
as required.
\end{proof}

\section{High-frequency stochastic instability}\label{sec:insta}

The main aim of this section
is the proof of Proposition~\ref{prop:hf-insta}.
It furthermore collects some estimates relevant to the study of the energy median dynamics in the
concentrated setting. It is a fundamental
tool of our analysis, as we make use of the non-degeneracy of the
noise to obtain the instability of high-frequency states. 

\begin{proof}[of Proposition~\ref{prop:hf-insta}]

Let us start by considering the solution $ u $
to \eqref{eqn:main} with initial condition $ u_{0} $ such
that $ \| u_{0} \| =1 $, up to normalising it.
By \eqref{e:idata}, showing high-frequency stochastic instability requires us to 
consider initial data $ u_{0} $ satisfying
\begin{equs}
 \| (\Pi_{M-1}^{\myc} + \Pi_{M-2}^{\myc}) u_{0}  \| & \geqslant \frac{1}{4} \| \Pi^{\myl}_{M-2} u_{0} \| \;, \\
 \| \Pi^{\mygeq}_{M} u_{0} \| & \leqslant 2 \| \Pi^{\myl}_{M} u_{0} \| \;.
\end{equs}
It follows in particular that
\begin{equs}
\| (\Pi_{M-1}^{\myc} + \Pi_{M-2}^{\myc})  u_{0} \|^{2} & \geqslant \frac{1}{16} \| \Pi^{\myl}_{M-2} u_{0}
\|^{2} \;, \\
\| (\Pi_{M-1}^{\myc} + \Pi_{M-2}^{\myc}) u_{0} \|^{2} & \geqslant
\frac{1}{2^{4} \cdot (1 +
2^{4})} \| \Pi_{M}^{\mygeq}  u_{0} \|^{2} \;,
\end{equs}
which implies, since $ \| u_{0} \|=1 $, that there exists $
\beta_{0} \in \{ 1, \dots, m \} $ such that
\begin{equ}[e:beta0-1] 
\| (\Pi_{M-1}^{\myc} + \Pi_{M-2}^{\myc}) u_{0}^{\beta_{0}} \| \geqslant \eta_{0} \;,
\qquad \eta_{0} = \frac{1}{\sqrt{m} \sqrt{1 + 2^{5} + 2^{8}}} \;.
\end{equ}
It follows from Assumption~\ref{assu:noise} and in particular \eqref{e:lsnd} 
that there exists $ \alpha_{0} \in \{ 1, \dots, m \}$ with $
\nu^{\beta_{0}} \geqslant \nu^{\alpha_{0}} $ and a constant $ c > 0 $ such that
\begin{equation} \label{e:simply-n}
\begin{aligned}
\sum_{| k | \leqslant (M-3)_{\alpha_{0}}} \sum_{l \in \ZZ^{d}} \Gamma^{\alpha_{0}}_{\beta_{0}, l} 
|\hat{\varphi}^{\beta_{0 }, k+ l}|^{2} \geqslant
c \| (\Pi_{M-1}^{\myc} + \Pi_{M-2}^{\myc}) \varphi^{\beta_{0}} \|^{2} \;.  
\end{aligned}
\end{equation}
Indeed, since $ \nu^{\beta_{0}}
\geqslant \nu^{\alpha_{0}} $ we have that $ L_{\beta_{0}} \leqslant
L_{\alpha_{0}} $ for all $ L \in \NN $. Therefore, to deduce \eqref{e:simply-n}
it suffices to show that there exists a constant $ c > 0 $ such that
\begin{equation*}
\begin{aligned}
\sum_{| k | \leqslant (M-3)_{\beta_{0}}} \sum_{l \in \ZZ^{d}} \Gamma^{\alpha_{0}}_{\beta_{0}, l} 
|\hat{\varphi}^{\beta_{0 }, k+ l}|^{2} \geqslant
c \sum_{| k | \leqslant M_{\beta_{0}}} | \hat{\varphi}^{\beta_{0}, k}
|^{2} \;,
\end{aligned}
\end{equation*}
which is the case with $ c = \min_{l \in \mA}
\Gamma^{\alpha_{0}}_{\beta_{0}, l} > 0$, where $ \mA $ is the set in
Assumption~\ref{assu:noise}, since for every $ | k | \leqslant
M_{\beta_{0}} $ there exists an $ l \in \mA $ such that $ | k - l | \leqslant
(M-3)_{\beta_{0}}$: note that $ B ( M_{\beta_{0}} ) \subseteq  B (
(M-3)_{\beta_{0}} + 3 \nu_{\mathrm{min}}^{-1/2 \a }) $, so that indeed the claim follows from the assumption.

Given these preliminaries, we now write the solution $ u $ to \eqref{eqn:main}
in its mild formulation in
Fourier coordinates. Following the conventions of
Remark~\ref{rem:fourier}, this yields
\begin{equation*}
\begin{aligned}
\hat{u}_{t}^{k} = e^{- \nu \zeta_{k} t} \hat{u}_{0}^{k} + \int_{0}^{t}
e^{- \nu \zeta_{k} (t -s)} \sum_{l\in\ZZ^d}  \hat{u}_{s}^{k- l} \cdot \ud
B^{l}_{s} \;.
\end{aligned}
\end{equation*}
Now consider the time horizon $ t_{\star}^{M, (1)} $ and the threshold $ \eta
$:
\begin{equation*}
\begin{aligned}
t_{\star}^{M, (1)} = \frac{1}{2\Delta_{M-3}} \;, \qquad \eta <
\eta_{0} \;,
\end{aligned}
\end{equation*}
where $ \eta $ is arbitrary (as long as it satisfies the constraint with a
strict inequality). Then we introduce the following stopping times:
\begin{equ}
\tau_{\mathrm{low}}  = \inf \Big\{ t \in [0, t_{\star}^{M, (1)})  \; \colon \; e^{
\zeta_{M}t} \| (\Pi_{M-1}^{\myc} + \Pi_{M-2}^{\myc}) 
u_{t}^{\beta_{0}} \| \leqslant \eta \Big\} \;, 
\end{equ}
with the usual convention that $\inf \emptyset = +\infty$.
This definition may appear a bit odd, so let us explain it. Our aim will be to
prove that for short times, or at least until the system has diluted,
we have a lower bound on the energy in a shell of width two about level $ M-1 $ 
for the coordinate $ \beta_{0} $,
which is enough to shift the energy to level $ M-3 $. 

As we will see, we eventually show that the system has a (very small) chance
of diluting only after a time of order
\begin{equ}
t_{\star}^{M, (2)} = \frac{\log{(\lambda \Delta_{M-3})}}{2\Delta_{M-3}} \gg
t_{\star}^{M, (1)}\;.
\end{equ}
The parameter $ \lambda > 0 $ can be arbitrarily large and is required to
close the estimates below.
The eventual time scale $ t_{\star}^{M, (2)}  $ hides two different effects.
One is a shift of energy, and another is an increase due to the different rates
of dissipation. Indeed, we will provide a lower bound on the amount of energy
that is shifted of the following order (omitting noise terms, which are the
technically the most challenging):
\begin{equ}
e^{2 \zeta_{M-2} t} \| \Pi^{\myl}_{M-3} u_{t} \|^{2} \gtrsim e^{ 2\Delta_{M-2}
t} \int_{0}^{t}
e^{- 2\Delta_{M-2} s} \| (\Pi_{M-1}^{\myc} + \Pi_{M-2}^{\myc}) u_{s}^{\beta_{0}} \|^{2} \ud s \;.
\end{equ}
The main contribution of the integral comes from the time
interval $ [0, t_{\star}^{M, (1)}] $ and is of order $ 1/ \Delta_{M-2} $. This
then becomes of order $ \lambda$ by time $ t_{\star}^{M, (2)} $, when
multiplied with the factor $ e^{\Delta_{M} t} $, which accounts for the
different rates of dissipation. Hence $ t^{M,
(1)}_{\star} $ is the time-scale up to which we require a lower bound on the
energy in $ (\Pi_{M-1}^{\myc} + \Pi_{M-2}^{\myc})u^{\beta_{0}}_{s} $. At the same
time, we cannot expect a lower bound of this kind for
larger times scales, because modes in a shell of width two about level $ M-1
$ start dissipating at substantially different rates after that time, hence the
definition of $ \tau_{\mathrm{low}} $.
Next we consider the following stopping time to control central and high
frequencies
\begin{equ}
\tau_{\mathrm{up}}  = \inf \{ t \geqslant 0  \; \colon \; e^{ \zeta_{M-2}t } \|
\Pi^{\mygeq}_{M-2} u_{t} \| \geqslant 2 \} \;.
\end{equ}
Similarly, we introduce a stopping time for the low frequency component, which
kicks in roughly when the system has diluted:
\begin{equ}
\sigma = \inf \{ t \geqslant 0  \; \colon \; e^{\zeta_{M-2} t} \|
\Pi^{\myl}_{M-2} u_{t} \| \geqslant 8 \} \;.
\end{equ}
Finally, define
\begin{equ}
\tau = \sigma \wedge \tau_{\mathrm{up}} \wedge \tau_{\mathrm{low}} \;, \qquad v_{t} =
\begin{cases} u_{t} & \text{ if } t \leqslant
\tau \;, \\ e^{- \zeta_{M-2} (t - \tau)}
u_{\tau} & \text{
if } t > \tau \;. \end{cases} 
\end{equ}
With this definition we have that for all $ t \geqslant 0 $
\begin{equation} \label{e:ub-v}
\begin{aligned}
\| v_{t} \| \leqslant \sqrt{2^{6}+ 2^{2}} e^{- \zeta_{M-2}t} \;, \qquad \forall t \geqslant 0 \;,
\end{aligned}
\end{equation}
as well as
\begin{equation} \label{e:v-lbn}
\begin{aligned}
\| (\Pi_{M-1}^{\myc} + \Pi_{M-2}^{\myc}) v_{t}^{\beta_{0}} \| \geqslant
\eta^{\prime} e^{- \zeta_{M-2} t} \;, \qquad \forall t \in [0, t_{\star}^{M,
(1)}] \;,
\end{aligned}
\end{equation}
for some $ \eta^{\prime} < \eta $, were we used that $ \exp ( (\zeta_{M}-
\zeta_{M-2}) t_{\star}^{M, (1)} ) \leqslant c$, for some constant $ c > 0 $
independent of $ M $.
Consider then for any $ \alpha \in \{ 1, \dots , m \} $ and $ k \in
\ZZ^{d} $ the real-valued process $ \hat{z}^{\alpha, k}_{t} $ defined to be the
principal (positive) square root of $ | \hat{z}^{\alpha, k}_{t}|^{2} $, which
in turn solves the following equation:
\begin{equation}\label{e:z-2}
\begin{aligned}
\ud | \hat{z}^{\alpha, k}_{t}|^{2} = -2 \nu_{\alpha}  \zeta_{k} 
| \hat{z}^{\alpha, k}_{t}|^{2}  \ud t & +
\sum_{\beta =1}^{m}\sum_{l \in \ZZ^{d} } \Gamma^{\alpha}_{\beta, l}
|\hat{v}^{ \beta, k + l}_{t}|^{2}  \ud t \\
 + 2 \cdot \one_{[0, \tau]}(t) &\cdot \mf{Re} \bigg( \sum_{l \in \ZZ^d} 
\hat{v}^{\alpha, k}_{t}  \Big( \hat{v}^{- k - l}_{t} \cdot \ud
B^{l}_{t}\Big)^{\alpha}  \bigg) \;,
\end{aligned}
\end{equation}
for all $ k \in \ZZ^{d} $ and $ \alpha \in \{ 1, \dots, m \} $.
To see that $ | \hat{z}^{\alpha, k}_{t}|^{2} \geqslant 0 $ for all times
it suffices to observe that by It\^o's formula applied to the $ \alpha$-th component of $ u_{t} $ we
have that
\begin{equ}
| \hat{z}^{\alpha, k}_{t}|^{2} = | \hat{u}^{\alpha, k}_{t} |^{2} \;, \qquad \forall \alpha
\in \{ 1, \dots , m \}\;, \ k \in \ZZ^{d} \;, \ t
\leqslant \tau \;.
\end{equ}
Now, our aim is to prove by contradiction that $ \tau < t_{\star}^{M, (2)} $
with a probability bounded from below uniformly over the initial
data and with an explicit dependence on $ M $ (which we will then work to
improve by iterating the argument). To
do so, we start by obtaining a lower bound on the probability that $
\| \Pi_{M-2}^{\myl} z_{t_{\star}^{M, (2)}} \| $ is suitably large (in particular, it suffices to
show that $ \| \Pi_{M-3}^{\myl} 
z_{t_{\star}^{M, (2)}} \| $ is large).

Summing \eqref{e:z-2} over all $ k $ with $ | k | \leqslant
(M-3)_{\alpha}  $ and all $ \alpha \in \{ 1, \dots,m \} $ and using the bound
\eqref{e:simply-n} leads to
\begin{equation*}
\begin{aligned}
\ud \| \Pi^{\myl}_{M-3} z_{t} \|^{2} \geqslant  - & 2 \zeta_{M-3} \|
\Pi^{\myl}_{M-3} z_{t} \|^{2} \ud t + c \|(\Pi_{M-1}^{\myc} +
\Pi_{M-2}^{\myc})v_{t}^{\beta_{0}} \|^{2} \ud t \\
+ & \one_{[0, \tau]}(t)\sum_{\alpha =1}^{m} \sum_{| k | \leqslant (M - 3)_{\alpha}} 2 \mf{Re} 
\bigg( \sum_{l\in \ZZ^d} 
\hat{v}^{\alpha, k}_{t}  \left( \hat{v}^{- k - l}_{t} \cdot \ud
B^{l}_{t}\right)^{\alpha}  \bigg)\;.
\end{aligned}
\end{equation*}
Therefore, we obtain the following lower bound:
\begin{equation*}
\begin{aligned}
\| \Pi^{\myl}_{M-3} z_{t} \|^{2} & \geqslant c \int_{0}^{t}  e^{-2
\zeta_{M-3} (t-s)} \| (\Pi_{M-1}^{\myc} +
\Pi_{M-2}^{\myc})  v_{s}^{\beta_{0}} \|^{2} \ud s + e^{- 2 \zeta_{M-3} t}
\mM_{t}\\
&  = e^{-2 \zeta_{M-2}t}X_{t} + e^{- 2 \zeta_{M-2}t} Y_{t}\;,
\end{aligned}
\end{equation*}
where 
\begin{equation*}
\begin{aligned}
X_{t} & =   e^{2 \Delta_{M-3} t} c\int_{0}^{t}
e^{2 \zeta_{M-3} s}  \| (\Pi_{M-1}^{\myc} +
\Pi_{M-2}^{\myc})v_{s}^{\beta_{0}}  \|^{2} \ud s \;,\\
Y_{t} & = e^{2 \Delta_{M-3} t} \mM_{t} \;,
\end{aligned}
\end{equation*}
and where $ \mM_{t} $ is the martingale
\begin{equation*}
\begin{aligned}
\mM_{t} =   2 \mf{Re} \bigg( \int_{0}^{t \wedge \tau} e^{2 \zeta_{M-3}s}\sum_{\alpha =1}^{m} \sum_{| k | \leqslant (M - 3)_{\alpha}} \sum_{l\in\ZZ^d} 
\hat{v}^{\alpha, k}_{s}  \left( \hat{v}^{- k - l}_{s} \cdot \ud
B^{l}_{s}\right)^{\alpha}  \bigg)\;.
 \end{aligned}
\end{equation*}
At this point we would like to prove that by time $ t_{\star}^{M, (2)} $ the
drift $ X_{t} $ has become very large, while the martingale term is not relevant.
For this purpose, we require an upper bound on $ Y_{t} $ and a lower
bound on $ X_{t} $. For the martingale $ \mM_{t} $ we have the following
estimate on the quadratic variation:
\begin{equs}
\langle \mM \rangle_{t} & \leqslant  \sum_{\alpha, \beta=1}^{m} \sum_{l\in\ZZ^d}
\Gamma^{\alpha}_{\beta, l} \int_{0}^{t}  \Big\vert  e^{ 2
\zeta_{M-3} s}\sum_{| k | \leqslant (M-3)_{\alpha}} \hat{v}^{\alpha, k}_{s}
\hat{v}^{\beta, -k - l}_{s} \Big\vert^{2} \ud s \\
& \lesssim_{\Gamma}  \int_{0}^{t} e^{-4\Delta_{M-3} s} \ud s 
 \lesssim_{\Gamma} \frac{1 - e^{- 4\Delta_{M-3} t}}{\Delta_{M-3}} \;, \label{e:ub-mart}
\end{equs}
where we used both the decay assumption on $ \Gamma $ in \eqref{e:reg-assu} and
the upper bound on $ v $ from \eqref{e:ub-v}.
In particular, from the Burkholder--Davis--Gundy inequality, for any $ p
\geqslant 1 $ there exists a constant $ C(p, \Gamma) $ such that
\begin{equation*}
\begin{aligned}
\EE | Y_{t} |^{2p} \leqslant C(p, \Gamma) \left(\frac{e^{4
\Delta_{M-3} t}}{\Delta_{M-3}} \right)^{p} \;.
\end{aligned}
\end{equation*}

On the other hand for the drift term we have that for $ t \geqslant
t_{\star}^{M, (1)} $, and by \eqref{e:v-lbn}, the following lower bound holds:
\begin{equation} \label{e:lb-drift}
\begin{aligned}
X_{t} & \gtrsim_{\eta^{\prime} , \Gamma} e^{2\Delta_{M-3}t}
\int_{0}^{t_{\star}^{M, (1)}} e^{- \Delta_{M-3} s} \ud s
\gtrsim_{\eta^{\prime} , \Gamma} \frac{e^{2 \Delta_{M-3} t}}{\Delta_{M-3}} \;.
\end{aligned}
\end{equation}
Now define
\begin{equation*}
\begin{aligned}
 X = X_{t_{\star}^{M, (2)}} \;, \qquad Y = Y_{t_{\star}^{M, (2)}} \;,
\end{aligned}
\end{equation*}
so that by our lower bound on the drift \eqref{e:lb-drift} and our upper bound
on the martingale term \eqref{e:ub-mart}, there exist constants $
c > 0 $ and $ C_{p, \lambda} > 0 $ (for any $ p \geqslant 1 $, and where $
\lambda $ is the parameter in the definition of $ t_{\star}^{M, (2)} $) such that
\begin{equation} \label{e:bds-rv}
\begin{aligned}
 c \lambda \leqslant X \;, \qquad \EE [|Y|^{2 p}] \leqslant
C_{p, \lambda} (\Delta_{M-3})^{p} \;.
\end{aligned}
\end{equation}
At this point, our aim is to obtain a quantitative lower bound on the probability that $
X + Y $ is strictly positive. Note that at first it is not clear at all that
such a lower bound should hold, since albeit $ Y $ has mean zero, its fluctuations are much larger than the
lower bound on $ X $. Despite this fact, we obtain that \dash at least with some small
probability \dash the sum $ X+ Y$ stays positive.
Indeed, by \eqref{e:lb-drift} we have that
\begin{equation*}
\begin{aligned}
\PP (X + Y \leqslant 8) \leqslant  \PP ( Y \leqslant 8 - c \lambda ) \;.
\end{aligned}
\end{equation*}
The number $ 8 $ is chosen in connection to the definition of $ \sigma $, and
our aim is to prove that the probability above is bounded away from one.
Now, let us fix $ \lambda > 0 $ such that $ 8 - c \lambda = - 1 $, so that
our aim becomes to find an upper bound on the probability
$
\varphi = \PP (Y \leqslant - 1)
$.
Since $ Y $ has mean zero we have
\begin{equation*}
\begin{aligned}
0 = \EE[Y] \leqslant  - \varphi + \EE [Y \one_{\{Y > -1\}}] \;,
\end{aligned}
\end{equation*}
from which we deduce that
$\varphi \leqslant \EE [Y \one_{\{Y > - 1\}}]$.
At this point we use H\"older's inequality to bound 
\begin{equation*}
\begin{aligned}
\EE [Y \one_{\{Y > - 1\}}] \leqslant \EE [Y^{2p}]^{\frac{1}{2p}} (1-
\varphi)^{\frac{1}{q}} \;, \qquad \frac{1}{q} + \frac{1}{2 p} = 1 \;, \qquad q
\in (1, 2] \;.
\end{aligned}
\end{equation*}
Here the requirement $ q \in (1, 2] $ is needed to guarantee $ p \geqslant 1
$, so that \eqref{e:bds-rv} applies.
Indeed, from here, by using the second bound in \eqref{e:bds-rv}, we deduce
that for some constant $ c_{q} > 0 $ and any $ q \in (1, 2] $:
\begin{equation*}
\begin{aligned}
(\sqrt{\Delta_{M-3}})^{q}(1 - \varphi) \geqslant c_{q} \varphi^{q} \;.
\end{aligned}
\end{equation*}
Writing $ \varphi = 1 - \ve $ for $ \ve \in (0,1) $, we then have 
\begin{equation*}
\begin{aligned}
(\sqrt{\Delta_{M-3}} )^{q} \ve \geqslant c_{q} (1 - \ve)^{q} \geqslant 
{c}_{q}(1 - q \ve) \;,
\end{aligned}
\end{equation*}
which in turn implies the lower bound
\begin{equation*}
\begin{aligned}
\ve \geqslant \frac{c_{q}}{ c_{q} q + (\sqrt{\Delta_{M-3}})^{q}} \approx_{q}
\frac{1}{\Delta_{M-3}^{\frac{q}{2}}}  \;, \qquad \forall q \in (1, 2]\;.
\end{aligned}
\end{equation*}
Now we can choose $ q $ arbitrarily close to $1 $ and hence $ q/2 $ arbitrarily
close to $1/2$. For instance, by fixing appropriate $ q $ we find that
\begin{equation*}
\begin{aligned}
\PP (X + Y >  8) \geqslant \Delta_{M-3}^{-\frac{3}{4}} \;,
\end{aligned}
\end{equation*}
for any $ M \geqslant \K $, provided we choose $ \K $ sufficiently large. 

Observe
that on the event $ X + Y \geqslant 8 $ we must have $ \tau \leqslant 
t_{\star}^{M,(2)} $, since $ \sigma $ would kick in  by time $
t_{\star}^{M, (2)} $ at the latest, so that
\begin{equ}[e:boundtausmall]
\PP( \tau \leqslant t_{\star}^{M, (2)}) \geqslant \Delta_{M-3}^{-\frac{3}{4}}\;.
\end{equ}
We are left with two tasks. First, we
show that, conditional on this event we have, with high probability, that indeed $ \sigma
< \tau_{\mathrm{up}} \wedge \tau_{\mathrm{low}} $, meaning that we have overall
a (small) bound from below probability on the probability of dilution in the time interval $ [0,
t_{\star}^{M, (2)}] $. Second, we iterate our argument to go from a
small probability to a probability of order one by considering longer
time-scales.

We start with the first task, namely by proving that there exists a constant $
c > 0 $ for which
\begin{equation}\label{e:lowp-bd}
\begin{aligned}
\PP( \sigma < \tau_{\mathrm{low}} \wedge \tau_{\mathrm{up}} \wedge
t_{\star}^{M, (2)}) & \geqslant
c \Delta_{M-3}^{-\frac{3}{4}} \;.
\end{aligned}
\end{equation}
This requires some estimates on the high and central frequencies of the process
$ t \mapsto z_{t} $.

\textit{Estimate for $ \Pi^{\mygeq}_{M-1} z_{t} $.}
By summing \eqref{e:z-2} over $ | k | > (M-2)_{\alpha} $ and over $ \alpha \in
\{ 1, \dots, m \} $, we obtain for all $ t \geqslant 0 $
\begin{equation*}
\begin{aligned}
e^{2 \zeta_{M-2} t} \| \Pi^{\mygeq}_{M-2} z_{t} \|^{2} \leqslant \|
\Pi^{\mygeq}_{M-2}
z_{0} \|^{2} & + C(\Gamma)\int_{0}^{t} e^{2 \zeta_{M-2}s} \| v_{s} \|^{2} \ud s +
\mN_{t} \;,
\end{aligned}
\end{equation*}
again by the decay assumption \eqref{e:reg-assu} on the coefficients $
\Gamma^{\alpha}_{\beta, l} $ and where $ \mN $ is the martingale given by
\begin{equation*}
\begin{aligned}
\mN_{t} = 2 \mf{Re} \bigg( \sum_{\alpha =1}^{m} \sum_{l\in\ZZ^d}
\sum_{(M-2)_{\alpha} < | k | } \int_{0}^{t \wedge \tau}
e^{2 \zeta_{M-2} s} \hat{v}_{s}^{\alpha, k}
 \left( \hat{v}_{s}^{- k - l} \cdot \ud B^{l}_{s}\right)^{\alpha}  \bigg) \;.
\end{aligned}
\end{equation*}
Now, using once more the upper bound \eqref{e:ub-v} and the decay of the
coefficients $ \Gamma $, the continuous martingale
$ t \mapsto \mN_{t} $ has quadratic variation bounded by
\begin{equation*}
\begin{aligned}
\sum_{\alpha, \beta = 1}^{m} \sum_{l\in\ZZ^d} \Gamma^{\alpha}_{\beta, l}\int_{0}^{t}
\Big\vert \sum_{(M-2)_{\alpha} < | k |}e^{2
\zeta_{M-2} s} \hat{v}_{s}^{\alpha, k} \hat{v}_{s}^{\beta, - k - l} \Big\vert^{2}
\ud s \lesssim_{\Gamma} t
\;.
\end{aligned}
\end{equation*}
Therefore, by Doob's submartingale inequality, we obtain that for any $ p \in
[1, \infty] $ 
\begin{equation*}
\begin{aligned}
\PP \Big( \sup_{0 \leqslant s \leqslant t_{\star}^{M, (2)}} \mN_{s} \geqslant C
\Big) \lesssim \frac{\EE | \mN_{t_{\star}^{M,(2)}} |^{p}}{C^{p}} \lesssim
(t_{\star}^{M,(2)})^{\frac{p}{2}} \leqslant  \frac{1}{\Delta_{M-3}}  \;,
\end{aligned}
\end{equation*}
where the last inequality follows for example by choosing $ p =4 $, and
provided that $ \K $ is sufficiently
large. We therefore conclude that
\begin{equation} \label{e:up-bd-z}
\begin{aligned}
\PP \left( \tau_{\mathrm{up}} < \sigma \wedge \tau_{\mathrm{low}}
\wedge t_{\star}^{M, (2)} \right) \leqslant \frac{1}{\Delta_{M-3}} \;.
\end{aligned}
\end{equation}
This is a sufficient upper bound for our purposes, since the probability $
\Delta_{M-3}^{-1} $ is much smaller than $ \Delta_{M-3}^{- 3/4} $, which is the
lower bound on our tentative ``success'' event. The next step
is to establish a similar upper bound for the stopping time $
\tau_{\mathrm{low}} $.

\textit{Estimate for $ (\Pi_{M-1}^{\myc} + \Pi_{M-2}^{\myc})  z_{t}^{\beta_{0}} $.}
Similarly to above, we would now like to prove
a lower bound on $ (\Pi_{M-1}^{\myc} + \Pi_{M-2}^{\myc})  z_{t}^{\beta_{0}} $ up to time $ t_{\star}^{M,(1)} $,
with a very high probability. By summing
\eqref{e:z-2} for $ \alpha= \beta_{0} $ over $ (M-2)_{\beta_{0}} < | k | \leqslant M_{\beta_{0}} $ we find that
\begin{equation*}
\begin{aligned}
\sum_{(M-2)_{\beta_{0}} < | k | \leqslant M_{\beta_{0}}} e^{2 \zeta_{M} t} |
\hat{z}^{ \beta_{0}, k}_{t}
|^{2} \geqslant \| (\Pi_{M-1}^{\myc} + \Pi_{M-2}^{\myc})  z_{0}^{\beta_{0}} \|^{2} + \mO_{t} \;,
\end{aligned}
\end{equation*}
with 
\begin{equation*}
\begin{aligned}
\mO_{t} = 2 \mf{Re} \bigg( \sum_{l \in \ZZ^d} \sum_{(M-2)_{\beta_{0}}  < | k | \leqslant
M_{\beta_{0}}} \int_{0}^{t \wedge \tau }
e^{2 \zeta_{M} s} \hat{v}_{s}^{ \beta_{0}, k}
 \left( \hat{v}_{s}^{- k - l} \cdot \ud B^{l}_{s}\right)^{\beta_{0}} \bigg) \;.
\end{aligned}
\end{equation*}
Now, as before, we compute the quadratic variation of $ \mO $, up to time
$ t_{\star}^{M,(1)} $:
\begin{equation*}
\begin{aligned}
\langle & \mO \rangle_{t_{\star}^{M,(1)}} \\
&  \lesssim \sum_{\beta=1}^{m}
\sum_{l \in \ZZ^{d}}
\Gamma^{\beta_{0}}_{\beta, l}  \int_{0}^{t_{\star}^{M,(1)}} \Big\vert
\sum_{(M-2)_{\beta_{0}} < | k | \leqslant M_{\beta_{0}} }e^{2 \zeta_{M} s}
\hat{v}_{s}^{\beta_{0}, k} \hat{v}_{s}^{\beta, - k - l} \Big\vert^{2} \ud s \\
& \lesssim \sum_{\beta, l }
\Gamma^{\beta_{0}}_{\beta, l} \int_{0}^{t_{\star}^{M,(1)}}  e^{
4(\zeta_{M} - \zeta_{M-2}) s}
\bigg( \sum_{(M-2)_{\beta_{0}} < | k | \leqslant M_{\beta_{0}} }e^{2
\zeta_{M-2} s}
|\hat{v}_{s}^{\beta_{0}, k} \hat{v}_{s}^{\beta, - k - l}| \bigg)^{2} \ud s \\
& \lesssim_{\Gamma} t_{\star}^{M,(1)} \;,
\end{aligned}
\end{equation*}
where we used that for $ s \leqslant t_{\star}^{M,(1)} $ we have
$e^{ 4( \zeta_{M} - \zeta_{M-2})s} \leqslant c$,
for some constant $ c >0 $ independent of $ M $, the bound \eqref{e:ub-v}, and
the decay assumptions on $ \Gamma $.
Following the same steps as above, we
therefore find that, provided $ \K $ is sufficiently large:
\begin{equation} \label{e:l-bd-z}
\begin{aligned}
\PP (\tau_{\mathrm{low}} < \tau_{\mathrm{up}} \wedge \sigma \wedge
t_{\star}^{M, (2)}) \leqslant \frac{1}{\Delta_{M-3}} \;.
\end{aligned}
\end{equation}
Hence, combining \eqref{e:up-bd-z} and \eqref{e:l-bd-z} with \eqref{e:boundtausmall}, we obtain
\eqref{e:lowp-bd}, since $ \Delta_{M-3}^{-1} \ll \Delta_{M-3}^{-
\frac{3}{4}} $.

\textit{Iteration.}
The lower bound \eqref{e:lowp-bd} guarantees that we can dilute, but only with
a very small probability, by time $ t_{\star}^{M, (2)} $. The last step in the
proof is to perform a large number of attempts (more than $
\Delta_{M-3}^{3/4}$, in order to compensate the small probability of success) so that with
a high probability, we observe at least one success. We therefore define the final time horizon
\begin{equation*}
\begin{aligned}
t_{\star}^{M,(3)} = \frac{ \log{( \lambda \Delta_{M-3})}}{2
\Delta_{M-3}^{1 - r}} = \Delta_{M-3}^{r} \cdot
t_{\star}^{M, (2)} \;,
\end{aligned}
\end{equation*}
for an arbitrary parameter $ r \in (3/4, 1) $.
Before we can conclude let us observe that the previous calculations prove
also, up to changing the value of the proportionality constant, that there
exists a $ c > 0 $ such that for any
stopping time $ t_{0} $

\begin{equ}[e:t1bd]
\PP_{t_{0}} ( \sigma^{\tD}(M, t_{0}) \leqslant t_{0} + t_{\star}^{M, (2)})
\one_{ \mB_{\K, t_{0}} }> c \Delta_{M-3}^{- \frac{3}{4}}   \one_{ \mB_{\K,
t_{0}}}\;,
\end{equ}
where $ \mB_{\K, t_{0}} $ is the event
\begin{equation*}
\begin{aligned}
\mB_{\K, t_{0}} = \Big\{ \| (\Pi_{M-1}^{\myc} + \Pi_{M-2}^{\myc}) u_{t_{0}} \|
& \geqslant
\frac{1}{4} \| \Pi^{\myl}_{M -2} u_{t_{0}} \| \;,\\
 M & \geqslant \K\;,
\|w^{\mygeq} (M; t_{0}, \cdot) \| \leqslant 3 \Big\} \;.
\end{aligned}
\end{equation*}
The event $ \mB_{\K , t_{0}} $ is almost equivalent to the assumption on the
initial condition appearing in
Definition~\ref{def:high-freq-insta} for high-frequency stochastic instability
(up to a time shift).
The only difference is that at time $ t_{1} $ we only assume $ \| \Pi^{\mygeq}
u_{t_{1}} \| \leqslant 3 \| \Pi^{\myl}_{M} u_{t_{1}} \| $, rather than the same
inequality with the constant $ 2 $. Therefore \eqref{e:t1bd} follows
identically to \eqref{e:lowp-bd}, up to choosing a slightly smaller value for
$ \eta_{0} $ in \eqref{e:beta0-1}.

In addition, a slight adaptation of the second estimate of
Lemma~\ref{lem:neg-mmts} (the only difference being that we do not assume that
$ M $ is the skeleton median) and the same calculations that led to
\eqref{e:l-bd-z}, we find that the event $ \| w (M;
t_{0}, \cdot) \| > 3 $ has small probability, with respect to our benchmark
probability $ \Delta_{M-3}^{- 3/4} $, at least if $ t_{0} \leqslant t_{\star}^{M, (3)} $:
\begin{equs}
\PP \bigg( \sup_{0 \leqslant t \leqslant t_{\star}^{M, (3)}} &\| w^{\mygeq}
(M ;  t, \cdot ) \| > 3 \bigg) \leqslant
\frac{1}{\Delta_{M-3}}  \;, \label{e:sigma3bd}
\end{equs}
provided $ \K $ is sufficiently large: as above, the upper bound $ \Delta_{M-3}^{-1}  $
could be replaced by arbitrary larger inverse power, up to choosing
sufficiently large $ \K$. For later reference, let us denote
\begin{equation*}
\begin{aligned}
\widetilde{\sigma}  = \inf \{ t \geqslant 0   \; \colon \;  \| w^{\mygeq}
(M ;  t, \cdot ) \| > 3\} \;.
\end{aligned}
\end{equation*}

Now we are ready to iterate our bound to obtain the desired result.
Let us fix 
\begin{equation*}
\begin{aligned}
L_{M} = \lfloor \Delta_{M-3}^{r}\rfloor -1\;,
\end{aligned}
\end{equation*}
as well as the following sequence of events for $ \l \in \{ 1, \dots,
L_{M} \}  $:
\begin{equs}
\mB_{\l} &= \left\{ \| (\Pi_{M-1}^{\myc} + \Pi_{M-2}^{\myc}) u_{t} \| \geqslant
\frac{1}{4} \| \Pi^{\myl}_{M -2} u_{t} \| \;,  \quad \forall t \in [\l t_{\star}^{M,(2)},
(\l+1) t_{\star}^{M,(2)}] \right\} \\
& \qquad \cap \{ M \geqslant \K\;, \widetilde{\sigma} \geqslant \l
\cdot t_{\star}^{M, (2)} \} \;, \\
\mH_{\l} &= \bigcap_{j = 1}^{\l} \mB_{j}\;.
\end{equs}
Then we have
\begin{equation*}
\begin{aligned}
\PP (\sigma^{\tD}
\geqslant  L_{M} t_{\star}^{M,(2)}) & \leqslant \PP (
\mH_{L_{M}}) + \PP ( \widetilde{\sigma} \leqslant t_{\star}^{M,
(3)} ) \\
& \leqslant \PP ( \mH_{L_{M}}) + \Delta_{M-3}^{-1}  \;, 
\end{aligned}
\end{equation*}
by \eqref{e:sigma3bd}.
In particular, for our purposes it suffices to find an upper bound to $
\PP (\mH_{L_{M}} )$. Here we find that for any $ \l \in \NN^{+} $:
\begin{equs}
\PP (\mH_{\l}) & = \EE \left[ \one_{\mH_{\l-1}}\PP_{(\l-1)
t_{\star}^{M,(2)}}(\mB_{\l})
\right] \\
& =  \EE  \left[
\one_{\mH_{\l-1}}\PP_{(\l-1) t_{\star}^{M,(2)}}(\mB_{\K, \l
t_{\star}^{M, (2)}})  \right]+ \PP( \widetilde{\sigma} \leqslant t_{\star}^{M,
(3)} )   \\
& \leqslant \EE \left[
\one_{\mH_{\l-1}} \right](1 - c \Delta_{M-3}^{- \frac{3}{4}}) +
\Delta_{M-3}^{-1}  \\
& \leqslant \PP (\mH_{\l-1}) (1 - c \Delta_{M-3}^{- \frac{3}{4}}) +
\Delta_{M-3}^{-1}   \;,
\end{equs}
where we used both \eqref{e:sigma3bd} and \eqref{e:t1bd}. Now we can iterate
this bound to obtain overall for some $ c, C > 0 $:
\begin{equs}
\PP (\mH_{L_{M}}) & \leqslant (1 - c \Delta_{M-3}^{- \frac{3}{4}})^{L_{M}} +
\Delta_{M-3}^{-1} \sum_{\l =
0}^{L_{M}-1}(1 - c \Delta_{M-3}^{- \frac{3}{4}})^{\l} \\
&\leqslant (1 - c \Delta_{M-3}^{- \frac{3}{ 4}})^{L_{M}} + C 
\Delta_{M-3}^{-1}  \Delta_{M-3}^{r} \;.
\end{equs}
This last bound is now sufficient to conclude the proof, since for every $
\ve \in (0, 1) $ there exists a $ \K (\ve) $ such that
\begin{equation*}
\begin{aligned}
(1 - c \Delta_{M-3}^{- \frac{3}{ 4}})^{L_{M}} + 
C\Delta_{M-3}^{-1}  \Delta_{M-3}^{r} \leqslant \ve \;,
\qquad \forall M \geqslant \K(\ve) \;,
\end{aligned}
\end{equation*}
where we have used that in the definition of $ L_{M} $ the parameter $ r $
satisfies $ r \in (3/4, 1) $.
Then choose $ \K (\mf{t}, \ve)
$ sufficiently large such that $ \mf{t} \geqslant t_{\star}^{\K, (3)} $, and such
that all previous calculations hold true. We have proven that
\begin{equation*}
\begin{aligned}
\PP ( \sigma^{\tD} \leqslant  \mf{t}) \geqslant
 1- \ve \;, \qquad \forall M \geqslant \K(\mf{t}, \ve)\;.
\end{aligned}
\end{equation*}
This concludes the proof.
\end{proof}
The next result is a simple corollary of the definition of high frequency
instability, together with the moment bounds in Lemma~\ref{lem:neg-mmts}.
Recall the notation in \eqref{e:short}.

\begin{lemma}\label{lem:switch-diluted}
  Under the assumptions of Theorem~\ref{thm:lyap-func}, and in particular if $
W $ induces a high-frequency stochastic instability, then the following holds. For any $ i
\in \NN $, let $ V_{i+1} $ be the stopping time defined in
Definition~\ref{def:Tn}.  For any $ \ve \in (0, 1) $ there exists a $ \K(\ve) > 0$ such that 
\begin{equ}
  \PP_{V_{i+1}} \left( \sigma^{\tD}(V_{i+1}) <
    \sigma^{\mygeq}(V_{i+1}) \right)  \geqslant 1 - \ve\;,
\end{equ}
on the event $\{ M_{T_{i}} \geqslant \K \}$.
\end{lemma}

\begin{proof}
To further lighten the notation we assume that $ V_{i+1} = 0 $ and write $
\sigma^{\tD} $ and $ \sigma^{\mygeq} $ instead of $
\sigma^{\tD} (V_{i+1})$ and $ \sigma^{\mygeq}
(V_{i +1}) $ respectively.
Next, let
$ \delta \in (0, 1) $ be a parameter. We can lower bound
\begin{equs}
\PP ( \sigma^{\tD} < \sigma^{\mygeq} ) & \geqslant \PP (
\sigma^{\tD} < \delta < \sigma^{\mygeq} ) \\
& \geqslant 1 - \PP( \sigma^{\tD} \geqslant \delta)-
\PP( \sigma^{\mygeq}  \leqslant \delta )\;.
\end{equs}
Now, by Lemma~\ref{lem:neg-mmts}, choose $ \delta(\ve) \in (0, 1) $
sufficiently small, so that
\begin{equ}
  \PP(\sigma^{\mygeq}  \leqslant \delta) \leqslant \ve \;.
\end{equ}
This is possible because by definition of $ V_{i+1} $ we have that $ \|
w^{\mygeq} (M_{i}; V_{i+1}, \cdot) \| \leqslant 5/4 < 3/2 $ (see
Definition~\ref{def:Tn}).
As for the first probability, we use the Definition~\ref{def:high-freq-insta}
regarding high frequency stochastic instability. By choosing $ \K (\delta, \ve)
> 0 $ sufficiently large and $ M \geqslant \K (\delta, \ve) $, we find that $
\PP ( \sigma^{\tD} \geqslant \delta) \leqslant \ve$, so that the
result follows.
\end{proof}
We conclude this section with a lemma which provides a simple criterion to establish the 
non-degeneracy property in Assumption~\ref{assu:noise}. See also
Remark~\ref{rem:a-grt-1}.

\begin{lemma}\label{lem:geom}
  For every $ d, \beta \in \NN $ and $ \ve > 0 $ there exists an $ L_{0} (\beta, \ve) $ such that for every $ k \in \ZZ^{d} $ with $ | k | \geqslant L_{0} $ there exists an $
\l (k) \in \ZZ^{d} $ such that
\begin{equation*}
\begin{aligned}
| \l | \leqslant \beta \;, \qquad | k + \l | \leqslant
| k | - \frac{\beta}{\sqrt{d} } + \ve \;.
\end{aligned}
\end{equation*}
\end{lemma}

\begin{proof}
  We can assume without loss of generality that $ k = (k_{i})_{i=1}^{d} $ satisfies $ k_{i} \geqslant 0 $
  for all $ i \in \{ 1, \dots, d \} $. Next fix an $ i_{0} \in \{ 1, \dots , d
\} $ such that $  k_{i_{0}}  \geqslant \frac{| k |}{\sqrt{d}} $ and fix
\begin{equation*}
\begin{aligned}
\l = - \beta e_{i_{0}} = (0, \dots, 0, - \beta,0, \dots, 0 ) \in \ZZ^{d} \;,
\end{aligned}
\end{equation*}
so that the value $ -\beta $ appears in the $ i_{0} $-th position. Then we can
compute that
\begin{equation*}
\begin{aligned}
| k + \l |^{2} = | k |^{2} - 2 k_{i_{0}} \beta + \beta^{2} \;.
\end{aligned}
\end{equation*}
Now assume that $ \beta > 0 $. Then by concavity of the square root we obtain that
\begin{equation*}
\begin{aligned}
| k + \l | \leqslant | k | - \frac{2 k_{i_{0}} \beta - \beta^{2}}{2| k |}
\leqslant | k | - \frac{\beta}{\sqrt{d}} + \frac{\beta^{2}}{L_{0}} \;.
\end{aligned}
\end{equation*}
Therefore the result follows by choosing $ L_{0} $ sufficiently large.
\end{proof}

\section{High-frequency regularity}

The final step towards the construction of the Lyapunov functional is to
establish high frequency regularity estimates. We start with a bound on
exponential moments of $ \| w_{t} \|_{\gamma, M_{t}} $ for $ \gamma \leqslant 1/2 $.
Later we proceed to polynomial moments of the same quantity, but allowing higher
values of $ \gamma $.

\subsection{Exponential moments}
The purpose of the next result
is to obtain suitable regularity estimates, up to certain stopping times.
For the statement of the proposition recall here the definition of the stopping
times in \eqref{eqn:tau-generic-n}, as well as
of the event $ \mD_{E}(Q;t_{0})$ from
\eqref{e:sets-1n}, and of the first jump time $ T (t_{ 0}) $ of
the skeleton median in \eqref{e:Tt0}.

\begin{proposition}\label{prop:reg-high-freq}
  Under the assumptions of Theorem~\ref{thm:lyap-func}, let $ t_{0} $
be any  stopping time. Further,
fix parameters $\a \geqslant 1, \kappa > 0$, and $\gamma \in (0, 1/2]$ (with $
\a $ appearing in \eqref{eqn:main}).
Finally, set $ L = M_{t_{0}} + k_{0} $ for any $ k_{0} \in \NN^{+} $. Then the
following estimates hold uniformly over
  $ k_{0} \in \NN^{+}, \gamma \in (0, 1/2] $ and locally uniformly over $ \a
\geqslant 1$.
\begin{enumerate}
\item Uniformly over $
\delta \in (0, 1) $, there exists a deterministic increasing function $ \kappa
\mapsto C(\kappa) $ such that:
\begin{equ}[eqn:reg-dwn-n]
  \EE_{t_{0}} \left[ e^{ \kappa \| w( 
    t_{0}+ \delta- , \cdot) \|_{\gamma, L}^{2} }  \one_{
      \{ t_{0} + \delta \leqslant T (t_{0}) \}}\right] \leqslant
C(\kappa) e^{ 2\kappa e^{- 2 \delta k_{0}} \| w(t_{0}, \cdot)
\|^{2}_{\gamma, L}}  \;.
\end{equ}
\item For any constants $ 0 < E < \beta $ and any $ \mF_{t_{0}} $-adapted random variable $ Q \in \NN $,  assuming that $ \sigma_{\beta}^{\mygeq} (Q, t_{0}) \leqslant T
(t_{0}) $ on $ \mD_{E}(Q;t_{0}) $, there exists a deterministic function $
(\kappa,\beta, \delta ) \mapsto C(\kappa, \beta, \delta) $
increasing in $ \kappa $ such that:
\begin{equ}[e:rupnew]
  \EE_{t_{0}}  \left[ e^{
\kappa \| w( \sigma_{\beta}^{\mygeq} (Q,t_{0})-  , \cdot)
\|_{\gamma, L}^{2}}  \one_{\mD_{E}(Q;t_{0})} \right]  \leqslant
C \left(\kappa, \beta, E\right) \one_{\mD_{E}(Q;t_{0})} \;. 
\end{equ}
\item For any
stopping time $t_1$ with  $ t_{0} \leqslant
t_{1} \leqslant T(t_{0})  $, the bound
\begin{equ}[eqn:uniform-bd]
  \EE_{t_{0}} \left[ \sup_{t \in [t_{0}, t_{1})} e^{\kappa \| w(t, \cdot) \|_{\gamma, L}^{2}} \right] \leqslant
  \hat C(\kappa) e^{ 2 \kappa \| w (t_{0}, \cdot) \|_{\gamma, L}^{2} } \;,
\end{equ}
holds almost surely, for some increasing deterministic function $ \kappa \mapsto \hat C(\kappa)$.
\end{enumerate}
\end{proposition}
We observe that in the bounds we consider left limits $ w (t_{1}-, \cdot) $
because by assumption $ t_{1} \leqslant T (t_{0}) $. It may therefore be that
$ t_{1} $ is the first time after $ t_{0} $ at which the skeleton median jumps
and there is a discontinuity in $ w $: indeed, in many instances in which we use
the estimates above this will be the case and we will have $ t_{1} = T
(t_{0}) $. The discontinuity will then be treated separately. Moreover, let us
observe as in Lemma~\ref{lem:neg-mmts} that the estimate in
\eqref{e:rupnew} breaks down as $ E \uparrow \beta $.
The proofs of all these estimates follow along similar lines, the main
difference being the treatment of the initial condition.

\begin{proof}

  Let us start with some general considerations. First, in all cases we are
  considering the evolution of $ w $ on some time interval $ [t_{0},
  t_{1}) $, where $ t_{1} $ is another stopping time. 
To lighten the notation, we will consider throughout the proof $ t_{0} = 0 $
and write $ M = M_{t_{0}} $ and $ L = M_{t_{0}} + k_{0} $.
Then, by Lemma~\ref{lem:consistency} we have that $ \| w
(t, x) \| \leqslant 2 $ for all $ t \geqslant 0$, and in addition since $
t_{1} \leqslant T (t_{0}) $ we have for all $ t \in [0, t_{1}) $ that $
M_{t} = M $ and, via
It\^o's formula, $ w$ satisfies the following equation:
  \begin{equ}[eqn:high-freq]
  \ud  w_{t}  = \Big[ \mL w_{t} + \psi(M, u_{t})
  w_{t} + Q(u_{t}) \Big] \ud
t + \sigma(u_{t}, \ud W_{t})\;, \qquad \forall t \in [0, t_{1}) \;,
\end{equ}
where  $  \left( \mL \varphi \right)^{\alpha} = - \nu^{\alpha} (-
\Delta)^{\a } \varphi^{\alpha} $ and \(\psi (M,
u_{t}) \in \RR \) is given by
\begin{equ}[eqn:def-alpha]
\psi(M, u_{t}) = -  \frac{1}{\| \Pi^{\myl}_{M} u_{t}
\|^{2}} \langle \Pi^{\myl}_{M} u_{t}, \mL  u_{t} \rangle\;.
\end{equ}
In particular, by our definition of projection in Definition~\ref{def:proj}, it holds that \( 0 \leqslant \psi (M, u_{t}) \leqslant
\zeta_{M}\). Moreover the martingale term $ \sigma $ is given by
\begin{equ}[eqn:def-sigma]
  \sigma(u_{t}, \ud W_{t})  =   \Pi_{M}^{\myg}  \left(\frac{ u_{t} \cdot \ud W_{t} }{ \|
\Pi^{\myl}_{M}
u_{t}  \|} \right) -  \frac{ \Pi_{M}^{\myg}  u_{t}}{\| \Pi^{\myl}_{
M} u_{t}
\|^{3}} \langle \Pi^{\myl}_{M} u_{t}, u_{t} \cdot \ud W_{t} \rangle
\end{equ}
and the vector-valued quadratic variation term $ Q $ is given by
\begin{equs}[eqn:for-Q]
  Q^{\alpha}(u_{t})(x)   &= \frac{3}{2} \Pi^{\myg}_{M} u_{t}^{\alpha}
\sum_{\gamma, \eta=1}^{m} \sum_{| k | \leqslant M_{\gamma}, | l | \leqslant
M_{\eta}} \frac{ \hat{u}_{t}^{\gamma, k}
\hat{u}_{t}^{\eta, l}}{ \| \Pi^{\myl}_{M} u_{t} \|^{5}} C_{k, l}^{\gamma, \eta}
(u_{t})\\
&\quad -\frac{1}{2} \Pi^{\myg}_{M} u_{t}^{\alpha}
\sum_{\gamma, \eta=1}^{m} \sum_{| k | \leqslant M_{\gamma}, | l | \leqslant
M_{\eta}} \frac{1}{\|
\Pi^{\myl}_{M} u_{t} \|^{3}} \one_{\{ k+l =0 \} \one_{\{ \gamma = \eta \}}}C_{k,
l}^{\gamma, \eta} (u_{t}) \qquad \\
&\quad  - \frac{1}{2} \sum_{\beta=1}^{m}\sum_{| k | > (M+1)_{\alpha}, | l |
\leqslant M_{\beta} }
\frac{ \hat{u}_{t}^{\beta, l} e_{k}(x)}{ \| \Pi^{\myl}_{M} u_{t} \|^{3} }  C_{k,
l}^{\alpha, \beta}(u_{t})\;,
\end{equs}
where $ C_{k, l} (u) $ is the covariation defined in Remark~\ref{rem:fourier}.

We can rewrite the term in the last line in spatial
coordinates, so that it becomes simpler to estimate: 
\begin{equ}[e:q-spatial]
  -\sum_{\beta=1}^{m}\sum_{| k | > (M+1)_{\alpha}, | l | \leqslant M_{\beta}}  
  \frac{ \hat{u}_{t}^{\beta, l} e_{k}(x)}{ \| \Pi^{\myl}_{M} u_{t} \|^{3} }
C^{\alpha, \beta}_{k, l}(u_{t}) = -  \Pi^{\myg}_{M} \left[
 F^{\alpha} (u_{t}) \right] (x) \;,
\end{equ}
with 
\begin{equ}
  F^{\alpha} (u)(x) = \sum_{\beta, \theta, \eta=1}^{m}
\frac{u^{\beta}(x)}{\| \Pi^{\myl}_{M} u\|^{3}} \int_{\TT^{d}} \Pi^{\myl}_{M} u^{\eta} (y)
u^{\theta}(y) \Lambda^{\alpha, \eta}_{\beta, \theta} (x-y) \ud y\;.
\end{equ}
Here $ \Lambda $ is the spatial correlation function introduced in
\eqref{e:def-Lambda}, and in particular by our regularity assumption
\eqref{e:reg-assu} on the noise, we have
that $ \| \Lambda \|_{\infty} < \infty $.

Next we will write $ w_{t_{1}} $ as a mild solution to
\eqref{eqn:high-freq}. 
Namely, for $ 0 \leqslant  s < t < \infty$ and $ L = M + k_{0} $  we
introduce the time-inhomogeneous semigroup as follows (written in coordinates
for $ \alpha \in \{ 1, \dots, m \} $): 
\begin{equ}[e:def-semi]
S_{s,t}^{L} \, \varphi^{\alpha} = e^{(t-s) \mL^{\alpha}  + \int_{s}^{t}
\psi_{r}^{\alpha}  \ud r} \Pi_{L}^{\myg}  \varphi^{\alpha} \;.
\end{equ}
Observe that although we do no state
this explicitly in the definition, the semigroup depends on $ M $ and
\( u \) through $ \psi $. Hence, for $ t \in [0, t_{1}) $
\begin{equ}[eqn:decomposition-w]
  \Pi_{L}^{\myg}  w_{t} =   S_{0,
  t}^{L} w_{0} 
   + \int_{0}^{t} S_{s, t}^{L} \, Q(u_{s}) \ud s +
  \int_{0}^{t} S_{s,
  t}^{L} \, \sigma(u_{s}, \ud W_{s})\;.
\end{equ}

Note also that we project on frequencies higher than $ L $ because
eventually we are interested in the norm $ \| w_{t} \|_{\gamma, L} $.
For later reference let us write, for $ t \in [0, t_{1}) $
\begin{equ}[eqn:def-y-z]
  y_{t} = \int_{0}^{t} S_{s, t}^{L} \, Q(u_{s}) \ud s, \qquad
  z_{t} = \int_{0}^{t} S_{s, t}^{L} \, \sigma(u_{s}, \ud W_{s})\;.
\end{equ}
Regarding the last term $ z_{t} $, we observe that for any $ L \in \NN $, the
semigroup $ S_{s, t}^{L} $ is not adapted to the filtration $ \mF_{s} $ at time
$ s $, because of the
presence of $ \psi_{t} = \psi(M, u_{t})$. But the stochastic integral
can be still defined by rewriting $$ \exp{ \Big( \int_{s}^{t} \psi_{r} \ud r
\Big)} = \exp { \Big( \int_{0}^{t} \psi_{r} \ud r \Big)} \exp {  \Big(-\int_{0}^{s}
\psi_{r} \ud r \Big)},$$
so that the part that is not adapted can be taken out of the stochastic
integral. Yet, for this very reason, proving a useful bound on the
stochastic convolution is a bit cumbersome and we will use energy estimates for
$z_{t}$ instead. In any case, as we already observed, the crucial observation for our proof is that
\begin{equ}[eqn:reg-prf-alpha-bd]
0 \leqslant   \psi_{t} \leqslant \zeta_{M}\;,
\end{equ}
so that we are in the setting of Lemma~\ref{lem:reg-semi}.
This concludes our preliminary observations.
Next we will concentrate on finding
separate bounds for the three terms $ S_{0, t}^{L} w_{0}
$, $ y_{t} $ and $ z_{t} $ for $ t \in [0, t_{1}) $. To lighten the notation we will omit
writing explicitly the dependence of constants on the parameters appearing in
the statement of the proposition.

\textit{Bound on the initial condition.} We start with the first term,
concerning the initial condition, where via the two estimates in 
Lemma~\ref{lem:reg-semi}, since $ \a  > 1/2 $ and in view of the bound \eqref{eqn:reg-prf-alpha-bd}, for some
deterministic $ C(\gamma) > 0 $, both of the following two bounds hold:
\begin{equs}[eqn:reg-prf-1]
  \|  S_{0, t}^{L} w_{0} \|_{\gamma,
  L} \leqslant \begin{cases} C(\gamma) t^{-
  \frac{\gamma}{2 \a }} \;, \\ e^{- t k_{0}} \|
w_{0} \|_{\gamma, L}\;, \end{cases} \forall t \in
[0, t_{1})\;,
\end{equs}
where the first estimate follows from 
\begin{equation*}
\begin{aligned}
  \|  S_{0, t}^{L} w_{0} \|_{\gamma,
  L} \lesssim_{\gamma} t^{-
  \frac{\gamma}{2 \a }} \| w_{0} \| \lesssim_{\gamma}t^{-
  \frac{\gamma}{2 \a }} \;,
\end{aligned}
\end{equation*}
since $ \| w_{0} \| \leqslant 2 $.

\textit{Bound for $ y $.} For the convolution term $ y_{t} $ we start with a bound on
the quadratic variation $ Q(u_{t}) $ from \eqref{eqn:for-Q}. We find that for
all $ t \in [0, t_{1}) $ and $ \alpha \in \{ 1, \dots, m \} $, by the estimate in
Remark~\ref{rem:fourier} and \eqref{e:q-spatial}:
\begin{equs}
  \| Q^{\alpha}(u_{t}) \| &\lesssim  \| \Pi_{M}^{\myg}  u_{t}^{\alpha} \|   \left( \frac{\left( \sum_{ k }| \hat{u}_{t}^{k} | \| u_{t} \|_{\l^{2}_{k}} \right)^{2} }{\|
\Pi_{M}^{\myl}  u_{t} \|^{5}} + \sum_{ k } \frac{\| u_{t} \|_{\l^{2}_{k}}^{2}}{\|
\Pi_{M}^{\myl} u_{t} \|^{3}} \right)  + \left\| \Pi^{\myg}_{M} \left[  F^{\alpha}(u_{t}) \right] \right\|\\
&\lesssim  \frac{\| \Pi_{M}^{\myg}  u_{t} \| \| u_{t} \|^{2}}{\|
\Pi_{M}^{\myl}  u_{t} \|^{3}} \frac{\| \Pi_{M}^{\myl}  u_{t} \|^{2}}{ \|
\Pi_{M}^{\myl}  u_{t} \|^{2}} + \frac{\| \Pi_{M}^{\myg}  u_{t} \| \|
u_{t} \|^{2}}{ \| \Pi_{M}^{\myl}  u_{t} \|^{3}} + \frac{\| u_{t} \|^{3}}{\|
\Pi^{\myl}_{M} u_{t} \|^{3}} \| \Lambda \|_{\infty} \\
 &\lesssim  1 \;.
\end{equs}
Here we used that by assumption 
$ t_{1} \leqslant T(t_{0}) $ so that $ M = M_{t} $ for all $ t \in
[0,t_{1}) $ and therefore
by Lemma~\ref{lem:consistency} we have $\| u_{t} \| \leqslant \sqrt{5} \|
\Pi^{\myl}_{M} u_{t} \| $ for all $ t \in [0, t_{1}) $.
We therefore conclude that for some deterministic constant $ C \in (0, \infty)  $
\begin{equ}[eqn:reg-prf-2]
 \| Q (u_{t}) \| \leqslant C\;, \quad \forall t \in [0, t_{1}) \;.
\end{equ}
In view of this bound, we now use Lemma~\ref{lem:reg-semi} with $
\delta = \gamma$ to estimate the term $
t \mapsto y_{t} $ as follows. Here we recall that by assumption we have $ t_{1}
\leqslant 3 $, since $ T (t_{0}) - t_{0} \leqslant 3  $ by
Lemma~\ref{lem:consistency}, so that we can estimate
\begin{equ}
  \bigg\| \int_{0}^{t} S^{L}_{s, t} Q(u_{s}) \ud s
  \bigg\|_{\gamma, L }  \leqslant \bigg\| \int_{0}^{t} S^{L}_{s , t} Q(u_{s}) \ud s
  \bigg\|_{\gamma, L}  \lesssim
\int_{0}^{t } (t-s)^{- \frac{\gamma}{2 \a }}  \ud s \lesssim 1 \;,
\end{equ}
uniformly over $ t \in [0, t_{1}) $, where we applied \eqref{eqn:reg-prf-2} and used
that $ \gamma < 2 \a  $. Hence we find
a deterministic constant $ C(\gamma)>0 $ such that 
\begin{equ}[eqn:reg-prf-3]
  \sup_{0 \leqslant t < t_{1}} \bigg\| \int_{0}^{t} S^{M}_{s, t} Q(u_{s}) \ud s
  \bigg\|_{\gamma, L}\leqslant C(\gamma)\;.
\end{equ}

This concludes our bound on the convolution term $ y $. The stochastic
convolution term $ t \mapsto z_{t} $ requires special care, so we defer its
study to the separate Lemma~\ref{lem:bd-z}. In particular, the named lemma
shows that for every $ \kappa $ and $ \gamma $ as in our assumption there
exists a $ C(\kappa, \gamma) \in (0, \infty) $ such that
\begin{equation} \label{e:zbd}
\begin{aligned}
\EE \left[ \sup_{0 \leqslant t< t_{1}} e^{\kappa \| z_{t} \|_{\gamma,
L}^{2}} \right] \leqslant C(\kappa,
\gamma) \;.
\end{aligned}
\end{equation}
We have now all tools at our disposal to deduce the desired bounds.

\textit{Proof of \eqref{eqn:reg-dwn-n}.} The only difference that appears
throughout the proofs of the bounds lies in the treatment of the initial
condition. In all cases we have that
\begin{equation*}
\begin{aligned}
 e^{\kappa \| w_{t_{1} -} \|^{2}_{\gamma, L}} \leqslant 
\exp \left( 2\kappa  \| S^{L}_{0, t_{1}-} w_{0}  \|^{2}_{\gamma,L} + 4\kappa  \|
y_{t_{1}-}  \|^{2}_{\gamma, L} + 4 \kappa \| z_{t_{1}-} \|^{2}_{\gamma, L} \right)  \;,
\end{aligned}
\end{equation*}
as $ (a + b)^{2} \leqslant 2 a^{2} + 2 b^{2} $.
Therefore, by \eqref{eqn:reg-prf-3}, we can further bound for $ t_{1} =
t_{0} + \delta $ on the event $ \{ t_{0} + \delta < T (t_{0}) \} $:
\begin{equation*}
\begin{aligned}
\EE \left[ e^{\kappa \| w_{t_{1} -} \|^{2}_{\gamma, L}} 1_{ \{ t_{0} + \delta
\leqslant T(t_{0}) \}}\right] \leqslant
C(\kappa, \gamma) \EE \left[
\exp \left( 2\kappa  \| S^{L}_{0, t_{1}-} w_{0}  \|^{2}_{\gamma,L} + 4 \kappa
\| z_{t_{1}-} \|^{2}_{\gamma, L} \right)  \right] \;.
\end{aligned}
\end{equation*}
Now, to obtain \eqref{eqn:reg-dwn-n} we apply the second estimate in
\eqref{eqn:reg-prf-1} as well as \eqref{e:zbd} to bound
\begin{equation*}
\begin{aligned}
\EE \left[ e^{\kappa \| w_{t_{1} -} \|^{2}_{\gamma, L}} \right]
\lesssim e^{2 \kappa e^{- 2 \delta k_{0}} \| w_{0} \|_{\gamma, L}^{2}} \;,
\end{aligned}
\end{equation*}
as required.

\textit{Proof of \eqref{e:rupnew}.} 
In this estimate, the point is that the stopping time $
t_{1} = \sigma_{\beta}^{\mygeq}(Q; t_{0}) $ take a time of order one to kick in, which suffices to
regularise the initial condition. As above, we start from the estimate
\begin{equation*}
\begin{aligned}
\EE \left[ e^{\kappa \| w_{t_{1} -} \|^{2}_{\gamma, L}} \right] \leqslant
C(\kappa, \gamma) \EE \left[
\exp \left( 2\kappa  \| S^{L}_{0, t_{1}-} w_{0}  \|^{2}_{\gamma,L} +4 \kappa
\| z_{t_{1}-} \|^{2}_{\gamma, L} \right)  \right] \;.
\end{aligned}
\end{equation*}
This time, we apply the first estimate in \eqref{eqn:reg-prf-1} and
Cauchy--Schwarz to obtain
\begin{equation*}
\begin{aligned}
\EE \left[ e^{\kappa \| w_{t_{1} -} \|^{2}_{\gamma, L}} \right] \leqslant
C(\kappa, \gamma) \EE \left[ e^{4 \kappa C(\gamma) t_{1}^{-
\frac{\gamma}{\a }}} \right]^{\frac{1}{2}} \EE \left[ e^{8 \kappa \|
z_{t_{1}-} \|_{\gamma, L}^{2}} \right]^{\frac{1}{2}} \;.
\end{aligned}
\end{equation*}
Now, since  $ \gamma /
\a  < 1 $, and since $ E < \beta $  we can apply
Lemma~\ref{lem:neg-mmts} together with \eqref{e:zbd}. We then obtain that as desired
\begin{equation*}
\begin{aligned}
\EE \left[ e^{\kappa \| w (\sigma_{\beta}^{\mygeq} (Q, t_{0}) - , \cdot)  \|^{2}_{\gamma, L}} \right]
1_{\mD_{E}(Q; t_{0})} \leqslant C \one_{\mD_{E}(Q; t_{0})}\;,
\end{aligned}
\end{equation*}
with the proportionality constant depending on all the parameters of the
problem.

\textit{Proof of the uniform bound \eqref{eqn:uniform-bd}.} 
We find, as above, that
\begin{equation*}
\begin{aligned}
\EE \left[ \sup_{0 \leqslant t < t_{1}} e^{\kappa \| w_{t} \|^{2}_{\gamma, L}}
\right] & \leqslant \EE \left[ \sup_{0 \leqslant t < t_{1}} 
e^{ 2\kappa  \| S^{L}_{0, t} w_{0}  \|^{2}_{\gamma,L} + 4\kappa  \|
y_{t}  \|^{2}_{\gamma, L} + 4 \kappa \| z_{t} \|^{2}_{\gamma, L} }
\right] \\
& \leqslant C(\kappa, \gamma) e^{2 \kappa \| w_{0} \|^{2}_{\gamma, L}} \EE \left[ \sup_{0 \leqslant t < t_{1}} 
e^{ 4 \kappa \| z_{t} \|^{2}_{\gamma, L} } \right] \\
& \lesssim e^{2 \kappa \| w_{0} \|_{\gamma, L}^{2}}\;,
\end{aligned}
\end{equation*}
by \eqref{e:zbd} and the second bound on \eqref{eqn:reg-prf-1}, as desired.
This concludes the proof of the proposition.
\end{proof}
In the following result we obtain exponential moments of the stochastic
convolution term that was relevant in the preceding proof.

\begin{lemma}\label{lem:bd-z} 
In the same setting as that of Proposition~\ref{prop:reg-high-freq}, for any
stopping time $ t_{0} $ and any other stopping time $ t_{1} $ such that $ t_{0}
\leqslant t_{1} \leqslant T(t_{0}) $, let $ z_{t} $ be the solution to:
\begin{equ}
\ud  z_{t}  = \big[ \mL z_{t} + \psi_{t} z_{t} \big] \ud
t + \sigma(u_{t}, \ud W_{t}), \qquad z_{t_{0}} = 0\;, \qquad t \in
[t_{0}, t_{1}) \;,
\end{equ}
where $ \psi_{t} $ is defined in \eqref{eqn:def-alpha}, and $ \sigma $ in
\eqref{eqn:def-sigma}. Then for any $ \gamma \in (0,  1/2] $ and $ \a \geqslant
1 $ we can bound
\begin{equation*}
\begin{aligned}
\EE_{t_{0}} & \left[ \sup_{s \in [t_{0}, t_{1})} \exp\big( \ve^{-1}\| z_{s}
  \|_{\gamma, M_{t_{0}} + k_{0}}^{2}\big) \right] < C(\ve, \gamma, \a)  \;,
\end{aligned}
\end{equation*}
for any $ \ve > 0 $, and where the constant $ C (\ve) > 0 $ additionally
depends on all the parameters in the statement of
Proposition~\ref{prop:reg-high-freq}.
\end{lemma}

\begin{proof}
As in the previous proof we fix $ t_{0} = 0, M =
M_{t_{0}} $ and $
M_{t_{0}} + k_{0} = L $.
Our aim is to obtain an energy estimate for $ \| z_{t} \|_{\gamma, L}^{2} $. Therefore we compute
\begin{equ}[eqn:for-z]
  \ud \| z_{t} \|_{\gamma, L}^{2} = 2 \langle \Lambda^{2
  \gamma}_{L} z_{t}, \mL z_{t} + \psi_{t} z_{t}  \rangle \ud t +
\overline{Q}(u_{t}) \ud t  + \ud \overline{\mM}_{t}\;,
\end{equ}
where $ \Lambda^{2 \gamma}_{L} $ is the Fourier multiplier, which in
coordinates for $ \alpha \in \{ 1, \dots, m \} $ is defined by $$ \widehat{\Lambda^{2
\gamma}_{L} z^{\alpha}} (k) = (1+|k| - L_{\alpha})^{2 \gamma}
\hat{z}^{\alpha, k}
 \one_{\{ | k | > (L+1)_{\alpha}\}}\;, $$
$ \overline{Q} $ is a quadratic variation term which we compute below, 
and $ \overline{\mM} $ is a local
martingale defined by
\begin{equ}
  \ud \overline{\mM}_{t} = \langle \Lambda^{2 \gamma}_{L}
  z_{t}, \sigma(u_{t}, \ud W_{t}) \rangle \;,
\end{equ}
with $ \sigma $ as in \eqref{eqn:def-sigma}. In particular, in Fourier
coordinates
\begin{equation*}
\begin{aligned}
\mF ( \sigma^{\alpha}  (u_{t}, \ud W_{t})) (k) = 1_{\{ | k | > (M+1)_{\alpha} \}}\sum_{l \in \ZZ^{d}} 
\bigg( & \sum_{\beta=1}^{m} \frac{\hat{u}^{\beta, k-l}}{\|
\Pi^{\myl}_{M} u_{t} \|} \ud B^{\alpha, \beta}_{l , t} \\
 - \frac{ \hat{u}^{\alpha, k}}{\| \Pi^{\myl}_{M} u_{t} \|^{3}} &\sum_{h \in
\ZZ^{d}} \sum_{\gamma, \eta= 1}^{m} \hat{u}^{\gamma, h}_{t} \hat{u}^{\eta,
h-l} \ud B^{\gamma, \eta}_{l, s} \bigg) \;.
\end{aligned}
\end{equation*}
As for the quadratic variation term $ \overline{Q} $, we can compute it as
follows:
\begin{equs}[e:qbar]
\overline{Q} (u_{t}) = &\sum_{\alpha =1}^{m} \sum_{| k | > (L+1)_{\alpha}}
 (1 +  | k | - L_{\alpha})^{2 \gamma} \sum_{l \in \ZZ^{d}}\Bigg(  \sum_{\beta,
\beta^{\prime}} \Gamma^{\alpha, \alpha}_{\beta,
\beta^{\prime}, l} \frac{ \hat{u}^{\beta, k - l}_{t} \hat{u}^{\beta^{\prime},
-k + l}_{t}}{\|
\Pi^{\myl}_{M} u_{t} \|^{2}} \\
& - 2 \mf{Re} \Bigg(  \sum_{\beta, \gamma, \eta=1}^{m} \Gamma^{\alpha, \gamma}_{\beta, \eta, l}
\frac{\hat{u}^{\beta, k-l}_{t} \hat{u}^{\alpha, -k}_{t} }{\| \Pi^{\myl}_{M}
u_{t} \|^{4}} \sum_{h \in \ZZ^{d}} \hat{u}^{\gamma, h}_{t} \hat{u}^{\eta,
h-l}_{t} \Bigg) \\
& +\frac{ | \hat{u}^{\alpha, k}_{t} |^{2} }{\| \Pi^{\myl}_{M} u_{t}
\|^{6}}
\sum_{\gamma, \gamma^{\prime}, \eta, \eta^{\prime} =1}^{m}
\Gamma^{\gamma, \gamma^{\prime}}_{\eta, \eta^{\prime}, l} \Bigg(  \sum_{h \in
\ZZ^{d}} \hat{u}^{\gamma, h}_{t} \hat{u}^{\eta,
h-l}_{t}\Bigg) \Bigg(  \sum_{h \in
\ZZ^{d}} \hat{u}^{\gamma^{\prime}, h}_{t} \hat{u}^{\eta^{\prime},
h-l}\Bigg)\Bigg) \;.
\end{equs}
In particular, via Remark~\ref{rem:fourier}, we can bound $ \overline{Q} $ by
\begin{equation*}
\begin{aligned}
\overline{Q} (u_{t}) \lesssim_{\Gamma} & \sum_{\alpha =1}^{m} \sum_{| k | > (L+1)_{\alpha}}
 (1 +  | k | - L_{\alpha})^{2 \gamma} \Bigg(\frac{| \hat{u}^{\alpha,
k}_{t} |^{2}}{\| \Pi^{\myl}_{M} u_{t} \|^{2}}
\frac{\| u_{t} \|^{4}}{\| \Pi^{\myl}_{M} u_{t} \|^{4}}  \\
& + \sum_{l \in \ZZ^{d}}    \Bigg(
\overline{\Gamma}_{l} \frac{| \hat{u}^{k - l}_{t} |^{2} }{\| \Pi^{\myl}_{M} u_{t}
\|^{2}} + \overline{\Gamma}_{l} \frac{| \hat{u}^{\alpha, k}_{t} |  | \hat{u}^{k -
l}_{t} |}{\| \Pi^{\myl}_{M} u_{t} \|^{2}} \frac{\| u_{t} \|^{2}}{\|
\Pi^{\myl}_{M} u_{t} \|^{2}}\Bigg)  \Bigg) \;,
\end{aligned}
\end{equation*}
where the first line contains a bound on the last term appearing in
\eqref{e:qbar}. Then by using the estimate $ \| u_{t} \| \lesssim \|
\Pi^{\myl}_{M} u_{t} \| $, via Lemma~\ref{lem:consistency}, since $ M =
M_{t} $, we can further estimate
\begin{equation*}
\begin{aligned}
\overline{Q} (u_{t}) & \lesssim \sum_{\alpha =1}^{m} \sum_{| k | > (L+1)_{\alpha}}
 (1 +  | k | - L_{\alpha})^{2 \gamma} \Bigg(\frac{| \hat{u}^{\alpha,
k}_{t} |^{2}}{\| \Pi^{\myl}_{M} u_{t} \|^{2}} + \sum_{l \in \ZZ^{d}} 
\overline{\Gamma}_{l} \frac{| \hat{u}^{k - l}_{t} |^{2} }{\| \Pi^{\myl}_{M} u_{t}
\|^{2}} \Bigg) \\
& \lesssim \| w_{t} \|_{\gamma, L}^{2} +  \sum_{\alpha =1}^{m} \sum_{| k | > (L+1)_{\alpha}}
 (1 +  | k | - L_{\alpha})^{2 \gamma} \sum_{l \in \ZZ^{d}} 
\overline{\Gamma}_{l} \frac{| \hat{u}^{k - l}_{t} |^{2} }{\| \Pi^{\myl}_{M} u_{t}
\|^{2}}\;.
\end{aligned}
\end{equation*}
Here the last term is the most tedious one to control, since we have on the one
hand to pass the quantity $ (1+| k | - L_{\alpha})^{2 \gamma} $ inside the norm
$ \| \cdot \|_{\l^{2}_{k}} $ and on the other hand because inside the sum we do
not have the $ \alpha $-th component of the solution, which means that we must
replace $ L_{\alpha} $ with $ L_{\beta} $ for arbitrary $ \beta \in \{ 1,
\dots, m \} $. To be precise, we bound the sum as follows:
\begin{equs}
 \sum_{\alpha =1}^{m} & \sum_{| k | > (L+1)_{\alpha}}
 (1 +  | k | - L_{\alpha})^{2 \gamma} \sum_{l \in \ZZ^{d}} 
\overline{\Gamma}_{l} \frac{| \hat{u}^{k - l}_{t} |^{2} }{\| \Pi^{\myl}_{M} u_{t}
\|^{2}} \\
 \leqslant  & \nu_{\mathrm{max}}^{\gamma/ \a} \nu_{\mathrm{min}}^{- \gamma/\a}
\sum_{\alpha, \beta =1}^{m} \sum_{| k | > (L+1)_{\beta}}
 (1 +  | k | - L_{\beta})^{2 \gamma} \sum_{l \in \ZZ^{d}} 
\overline{\Gamma}_{l} \frac{| \hat{u}^{\beta, k - l}_{t} |^{2} }{\| \Pi^{\myl}_{M} u_{t}
\|^{2}} \\
& + \sum_{\alpha, \beta =1}^{m} \sum_{(L +1)_{\beta} \geqslant | k | > (L+1)_{\alpha}}
 (1 + \Delta L^{\beta, \alpha}_{+} )^{2 \gamma} \sum_{l \in \ZZ^{d}} 
\overline{\Gamma}_{l} \frac{| \hat{u}^{\beta, k - l}_{t} |^{2} }{\| \Pi^{\myl}_{M} u_{t}
\|^{2}} \;, \label{e:2ndt}
\end{equs}
where $ \Delta L^{\beta, \alpha}_{+} = \{ L_{\beta} - L_{\alpha} \} \vee 0$.

To treat the first term above, let us introduce for $ k \in \ZZ^{d} $ the weight $
\varrho^{L}_{k}$ defined as follows:
\begin{equ}[eqn:def-weights]
\varrho^{L}_{k} = \begin{cases} 1, \qquad & \text{ if } | k | \leqslant L\;, \\
(1+| k | - L)^{2 \gamma}, \qquad & \text{ if } | k | > L\;.
\end{cases} 
\end{equ}
We omit writing the dependence of $ \vr $ on $ \gamma, $ since this parameter is fixed
throughout the proof. Similarly, define $ \varrho_{k} = (1 + | k
|)^{2 \gamma} $. Then by Lemma~\ref{lem:bound-weights} there exists a $
c(\gamma) > 0 $ such that
\begin{equ}
\varrho^{L}_{k} \leqslant c(\gamma) \varrho_{k+l}^{L} \varrho_{l}, \qquad \forall k, l
\in \ZZ^{d}\;.
\end{equ}
In particular, by the decay assumptions on $ \Gamma $
together with the fact that $ \gamma <  1 $, we find
\begin{equs}
\sum_{| k | > (L+1)_{\beta}}
 (1 +  | k | - L_{\beta})^{2 \gamma} \sum_{l \in \ZZ^{d}} 
\overline{\Gamma}_{l} \frac{| \hat{u}^{\beta, k - l}_{t} |^{2} }{\| \Pi^{\myl}_{M} u_{t}
\|^{2}}  & \lesssim  
 \sum_{l \in \ZZ^{d}} \overline{\Gamma}_{l} \varrho_{l} \sum_{k \in \ZZ^{d}}
\varrho^{L_{\beta}}_{l+k} | \hat{u}_{t}^{\beta, k+l} |^{2} \\
&\lesssim_{\Gamma}  \| \Pi_{L}^{\myleq}  u_{t} \|^{2} + \| u_{t} \|_{\gamma,
L}^{2}\;.
\end{equs}
As for the second term in \eqref{e:2ndt}, we have
\begin{equation*}
\begin{aligned}
 \sum_{(L +1)_{\beta} \geqslant | k | > (L+1)_{\alpha}}
 (1 + \Delta L^{\beta, \alpha}_{+} )^{2 \gamma} \sum_{l \in \ZZ^{d}} 
\overline{\Gamma}_{l} \frac{| \hat{u}^{\beta, k - l}_{t} |^{2} }{\| \Pi^{\myl}_{M} u_{t}
\|^{2}} \lesssim_{\Gamma, \nu} L^{2 \gamma} \frac{\| u_{t} \|^{2}}{ \|
\Pi^{\myl}_{M} u_{t} \|^{2}} \;.
\end{aligned}
\end{equation*}
Therefore, overall and once more via Lemma~\ref{lem:consistency}, we deduce that for all $
t \in [0, t_{1}) $
\begin{equ}[eqn:reg-prf-4]
  \overline{Q} (u_{t}) \lesssim L^{2 \gamma} \frac{\|  u_{t}  \|^{2}}{ \|
\Pi^{\myl} u_{t}\|^{2}} + \frac{\| u_{t}  \|^{2}_{\gamma, L}}{\|
\Pi^{\myl} u_{t} \|^{2}} 
 \lesssim L^{2 \gamma}  + \| w_{t} \|_{\gamma, L}^{2}\;.
\end{equ}
Now we in turn estimate for $ t \in [0, t_{1}) $
\begin{equ}
  \| w_{t} \|_{\gamma, L} \leqslant \| S_{0, t}^{L}
  w_{0} \|_{\gamma, L} + \|y_{t} \|_{\gamma, L} + \| z_{t}
  \|_{\gamma, L}\;.
\end{equ}
Hence, substituting \eqref{eqn:reg-prf-4} into \eqref{eqn:for-z} and using that $
\psi_{t} \leqslant \zeta_{M}$ we obtain for some $ C(\gamma )> 0 $ and $
\mf{c}_{1}(\nu), \mf{c_{2} (\nu)} > 0 $
and uniformly over $ t \in [0, t_{1})$
\begin{equs}
  \ud \| z_{t} \|_{\gamma, L}^{2} \leqslant & - \left(  \mf{c}_{1}(\nu)  \| z_{t} \|_{\gamma +
   \a, L}^{2} + \mf{c}_{2}(\nu) M^{2 \a -1} k_{0} \| z_{t} \|_{\gamma, L} \right) \ud
t\\
&   + C \left( L^{2 \gamma} + \| z_{t} \|_{\gamma,L}^{2}  + \| S_{0,
      t}^{L} w_{0} \|_{\gamma, L}^{2}  + \| y_{t}
\|_{\gamma, L}^{2} \right) \ud t + \ud \overline{\mM}_{t}\;,
\end{equs}
where we have employed the two different estimates below for the quadratic form $ \langle
\Lambda^{2 \gamma}_{L} z_{t}, ( \mL + \psi_{t}) z_{t}  \rangle  $ appearing in
\eqref{eqn:for-z}. On the one
hand, we have
\begin{equs}
 \langle \Lambda^{2 \gamma}_{L} z_{t}, ( \mL + \psi_{t}) z_{t}  \rangle &  = 2 \sum_{\alpha=1}^{m} \sum_{| k | > (L +1)_{\alpha}} (1+ | k | - L_{\alpha})^{2
\gamma} \left( - \nu^{\alpha} | k |^{2 \a} + \psi_{t} \right)
|\hat{z}_{t}^{\alpha, k}|^{2} \\
& \leqslant - \sum_{\alpha=1}^{m} \sum_{| k | > (L +1)_{\alpha}  }(1+ | k | -
L_{\alpha})^{2 \gamma}\left( \nu^{\alpha} | k |^{2 \a} - \zeta_{M} \right)
|\hat{z}_{t}^{\alpha, k}|^{2}  \\
& \leqslant -  \nu_{\mathrm{min}} \sum_{\alpha=1}^{m} \sum_{| k | > (L +1)_{\alpha} }(1+| k | -
L_{\alpha})^{2 \gamma}(1 + | k |- L_{\alpha})^{2 \a} |\hat{z}_{t}^{\alpha, k}|^{2}
\\& \leqslant -  \nu_{\mathrm{min}} \| z_{t} \|_{H^{\gamma + \a}_{L}}^{2}\;.
\end{equs}
Here we have used that $ \hat{z}^{\alpha, k} = 0 $ for $ | k | \leqslant
L_{\alpha}+1 $, and that since $ k_{0} \in \NN^{+} $ we have $ L \geqslant M +1 $ 
together with the inequality
$ x^{2 \a} - y^{2\a} \geqslant (x - y)^{2\a} $ for all $ x \geqslant y
\geqslant 0 $. 

On the other hand, we can estimate the quadratic form by using convexity of the
map $ x \mapsto x^{2 \a}  $, since $ \a \geqslant 1/2 $:
\begin{equs}
 \langle \Lambda^{2 \gamma}_{L} z_{t}, ( \mL + \psi_{t}) z_{t}  \rangle &  =  \sum_{\alpha=1}^{m} \sum_{| k | > (L +1)_{\alpha}} (1+ | k | - L_{\alpha})^{2
\gamma} \left( - \nu^{\alpha} | k |^{2 \a} + \psi_{t} \right)
|\hat{z}_{t}^{\alpha, k}|^{2} \\
& \leqslant - \sum_{\alpha=1}^{m} \sum_{| k | > (L +1)_{\alpha}  }(1+ | k | -
L_{\alpha})^{2 \gamma}\left( \nu^{\alpha} | k |^{2 \a} - \zeta_{M} \right)
|\hat{z}_{t}^{\alpha, k}|^{2}  \\
& \leqslant - \mf{c}_{2}(\nu) \sum_{\alpha=1}^{m} \sum_{| k | > (L +1)_{\alpha} }(1+| k | -
L_{\alpha})^{2 \gamma} M^{2\a -1} k_{0}|\hat{z}_{t}^{\alpha, k}|^{2}
\\& \leqslant - \mf{c}_{2} (\nu) M^{2\a -1} k_{0} \| z_{t} \|_{\gamma, L}^{2}\;.
\end{equs}

At this point we use interpolation with $ \mu = \frac{\gamma}{\gamma+\a} $ and
then Young's inequality for products for any $ \ve \in (0, 1) $ to estimate for
some $ c > 0 $
\begin{equs}
  \| z_{t} \|_{\gamma, L} & \leqslant c \| z_{t} \|^{1- \mu } \|
z_{t} \|_{\gamma +\a, L}^{\mu} \\
& \leqslant c (1-\mu) \ve^{-\frac{1}{1-\mu}} \| z_{t} \| + c \mu \ve^{\frac{1}{\mu}} \| z_{t}
\|_{\gamma+\a, L} \\
& \leqslant  c (1-\mu) \ve^{-\frac{1}{1-\mu}} (\| w_{t} \| + \|
S^{L}_{0, t} w_{0}\| + \| y_{t} \|) + c \mu \ve^{\frac{1}{\mu}} \| z_{t}
\|_{\gamma+\a, L}\\
& \leqslant  c (1-\mu) \ve^{-\frac{1}{1-\mu}}C+ c \mu \ve^{\frac{1}{\mu}} \| z_{t}
\|_{\gamma+\a, L}\;, \label{eqn:interpolation}
\end{equs}
where the bound on $ \| S_{0 , t} w_{0} \| $ and $ \| y_{t} \| $ is provided in
the relative steps of the proof of Proposition~\ref{prop:reg-high-freq} above.
Thus, finally, via \eqref{eqn:reg-prf-3} and the above estimate, by choosing $
\ve \in (0, 1) $ such that $ c \mu
\ve^{\frac{1}{\mu}} \leqslant \mf{c}_{1}(\nu) /2 $, we obtain for some $ C > 0 $, whose value
changes from line to line:
\begin{equs}
  \ud \| z_{t} \|_{ \gamma, L}^{2}  \leqslant &- \left(  
\frac{\mf{c}_{1}(\nu)}{2} \| z_{t} \|_{\gamma +
  \a, L}^{2} + \mf{c}_{2}(\nu) M^{2 \a -1}k_{0} \| z_{t} \|_{\gamma,
L}^{2} \right) \ud t \\
& + C ( L^{2 \gamma} + \| S^{L}_{0, t} w_{0} \|_{\gamma,
L}^{2}  ) \ud t + \ud \overline{\mM}_{t} \\
 \leqslant &-\left(   \frac{\mf{c}_{1}(\nu)}{2} \| z_{t} \|_{\gamma +
  \a, L}^{2}  + \mf{c}_{2}(\nu) M^{2 \a -1}k_{0} \| z_{t} \|_{\gamma,
L}^{2} \right) \ud t \\
&  + C ( L^{2 \gamma} + t^{- \frac{\gamma}{\a}}  ) \ud t + \ud \overline{\mM}_{t}\;. \label{eqn:reg-prf-5}
\end{equs}
The final ingredient in our study is a bound on the quadratic variation of the
martingale $ \overline{\mM} $. Note that in Fourier coordinates
\begin{equs}
\ud \overline{\mM}_{t} = & \sum_{l \in \ZZ^{d}}  \left(
\sum_{\alpha=1}^{m}\sum_{| k | >
 ( L+1)_{\alpha}} \frac{\varrho^{L_{\alpha}}_{k} \hat{z}_{t}^{\alpha, k}
 \left( \hat{u}_{t}^{ k-l} \cdot \ud B^{l}_{t} \right)^{\alpha}  }{\|
\Pi^{\myl}_{M} u_{t} \|} \right) 
\\
& - \sum_{l \in \ZZ^{d}}   \left( \frac{\langle \Lambda^{2 \gamma}_{L} z_{t},
u_{t} \rangle}{\| \Pi^{\myl}_{ M} u_{t}
\|^{3}} \sum_{\alpha=1}^{m} \sum_{| k | \leqslant M_{\alpha}}
\hat{u}_{t}^{\alpha, k}  \left( \hat{u}_{t}^{k-l} \cdot \ud B^{l}_{t}
\right)^{\alpha} \right)   \;.
\end{equs}
We then bound the quadratic variation of $ \overline{\mM} $ as follows, using
as usual that $ t_{1} \leqslant T(t_{0}) $ together with
Lemma~\ref{lem:consistency}
\begin{equs}
 \ud \langle \overline{\mM} \rangle_{t} \lesssim &\sum_{l \in \ZZ^{d}}
\overline{\Gamma}_{l}
 \left\vert \sum_{\alpha= 1}^{m} \sum_{| k | > (L+1)_{\alpha}} \frac{\varrho^{L_{\alpha}}_{k} 
|\hat{z}^{\alpha, k} | }{\| \Pi^{\myl}_{M}
u_{t}\|} | \hat{u}_{t}^{k - l} | \right\vert^{2}   \\
& + \sum_{l \in \ZZ^{d}} \overline{\Gamma}_{l} \left\vert\frac{|\langle
\Lambda^{2 \gamma}_{L} z_{t}, u_{t} \rangle|}{\| \Pi^{\myl}_{ M} u_{t}
\|^{3}} \sum_{\alpha=1}^{m} \sum_{| k | \leqslant M_{\alpha}}
|\hat{u}_{t}^{\alpha, k}|  |\hat{u}_{t}^{k-l} | \right\vert^{2}\\
\lesssim & \sum_{l \in
\ZZ^{d}} \overline{\Gamma}_{l} \left( \| z_{t} \|_{2
\gamma, L}^{2} \frac{ \| u_{t} \|^{2}}{\| \Pi_{M}^{\myl} 
u_{t} \|^{2}}  + \| z_{t} \|_{2 \gamma,L}^{2} \frac{\|
u_{t}\|^{6}}{\| \Pi^{\myl}_{M} u_{t} \|^{6}}  \right) \ud t\\
\lesssim & \| z_{t} \|_{2 \gamma, L}^{2} \ud t \;. \label{eqn:bd-z-qv-1}
\end{equs}
Now we make use of the fact that $ \gamma < 1 \leqslant \a$, so that we can use
interpolation to estimate the norm $ \| z_{t} \|_{2\gamma, L}^{2} $ by
$ \| z_{t} \|_{\gamma + \a, L}^{2} $, with the gain of a
small factor. In particular, by Young's inequality and interpolation as in
\eqref{eqn:interpolation}, we obtain
that for any $ \ve \in (0, 1) $ there exists a $ C(\ve) \in (0, \infty)  $ such
that for all $ t \in [0, t_{1}) $
\begin{equs}[eqn:bd-qv-total]
 \ud \langle \overline{\mM} \rangle_{t}
  \leqslant \left\{ \ve \| z_{t} \|^{2}_{\gamma+\a, L} + C(\ve) \right\} \ud t \;.
\end{equs}
Now from \eqref{eqn:reg-prf-5} and since by assumption $ L^{2 \gamma} \leqslant
M + k_{0} $ as $ \gamma \leqslant 1/2 $, we find that
\begin{equation*}
\begin{aligned}
\| z_{t} \|_{\gamma, L}^{2} \leqslant \int_{0}^{t} e^{- \mf{c}_{2} (\nu)
M^{2 \a -1} k_{0} (t -s)} C(M + k_{0} + s^{- \frac{\gamma}{\a}}) \ud s +
\mZ_{t} \;,
\end{aligned}
\end{equation*}
where the process $ \mZ_{t} $ satisfies $ \mZ_{0} = 0 $ and
\begin{equation*}
\begin{aligned}
\ud \mZ_{t} & = - \left(   \frac{\mf{c}_{1}(\nu)}{2} \| z_{t} \|_{\gamma +
  \a, L}^{2}  + \mf{c}_{2}(\nu) M^{2 \a -1}k_{0} \mZ_{t}\right) \ud t + \ud
\overline{\mM}_{t} \\
& \leqslant  -  \frac{\mf{c}_{1}(\nu)}{2} \| z_{t} \|_{\gamma +
  \a, L}^{2} \ud t + \ud \overline{\mM}_{t} \;.
\end{aligned}
\end{equation*}
In particular, since $ 2\a -1 \geqslant 1 $ for $ \a \geqslant 1 $, and since
by assumption $ t_{1} \leqslant 3 $ (which follows from
Lemma~\ref{lem:consistency}, because $ T (t_{0}) - t_{0} \leqslant 3 $ for any
stopping time $ t_{0} $), there exists a (deterministic) constant $ C > 0
$ such that
\begin{equation*}
\begin{aligned}
 \int_{0}^{t} e^{- \mf{c}_{2} (\nu)
M^{2 \a -1} k_{0} (t -s)} (M + k_{0} + s^{- \frac{\gamma}{\a}}) \ud s \leqslant
C \;, \qquad \forall M , k_{0} \in \NN^{+}\;. 
\end{aligned}
\end{equation*}
Therefore we obtain that uniformly over $ M $ and $ k_{0} $ we have, for any $
\mu > 0  $
\begin{equation*}
\begin{aligned}
\EE \left[ \sup_{t \in [0, t_{1})} e^{ \mu \| z_{t} \|_{\gamma, L}^{2}} \right]
\lesssim_{\mu}\EE \left[ \sup_{t \in [0, t_{1})} e^{ \mu \mZ_{t} } \right] \;.
\end{aligned}
\end{equation*}
Now for $ \mZ_{t} $ we find that if we set $ F_{\mu} (x) = e^{\mu x} $,
then
\begin{equation*}
\begin{aligned}
\ud F_{\mu} ( \mZ_{t}) \leqslant & F_{\mu} (\mZ_{t}) \left( - \mu \frac{\mf{c}_{1}(\nu)}{2} \| z_{t} \|_{\gamma +
  \a, L}^{2} + \frac{1}{2} \mu^{2} (\ve  \| z_{t} \|_{\gamma +
  \a, L}^{2}+ C(\ve)) \right) \ud t \\
& + \partial_{x} F_{\mu} (\mZ_{t}) \ud \overline{\mM}_{t} \\
\leqslant &-  F_{\mu} (\mZ_{t}) \mu \frac{\mf{c}_{1}(\nu)}{4} \| z_{t} \|_{\gamma +
  \a, L}^{2}  + \partial_{x} F_{\mu} (\mZ_{t}) \ud \overline{\mM}_{t}\;,
\end{aligned}
\end{equation*}
where we have used \eqref{eqn:bd-qv-total} to obtain the first inequality, and
to obtain the second inequality we have assumed that $ \mu $ satisfies
\begin{equation} \label{e:mu-bound}
\begin{aligned}
\mu \leqslant \frac{\mf{c}_{1}(\nu) \ve^{-1}}{2} \;.
\end{aligned}
\end{equation}
In particular, we conclude that for $ \mu $ satisfying \eqref{e:mu-bound}
\begin{equation*}
\begin{aligned}
\sup_{0 \leqslant t \leqslant 1} \EE \left[ e^{\mu \mZ_{t \wedge t_{1}}}+
\int_{0}^{3} F_{\mu} (\mZ_{s \wedge t_{1}}) \| z_{s \wedge t_{1}} \|_{\gamma + \a, L}^{2}  \ud s\right] < \infty
\;.
\end{aligned}
\end{equation*}
To pass the supremum inside the expectation we can further bound, by using the
already mentioned bound $ t_{1} \leqslant 3 $, which follows from our
assumptions:
\begin{equation*}
\begin{aligned}
 \EE \left[ \sup_{0 \leqslant t \leqslant t_{1}} F_{\mu} (\mZ_{t})  \right]
 & \leqslant\EE \left[ \sup_{0 \leqslant t \leqslant 3} F_{\mu} (\mZ_{t \wedge
t_{1}})  \right] \\
& \lesssim \EE \left[ \int_{0}^{3} | \partial_{x} F_{\mu}
(\mZ_{s \wedge t_{1}}) |^{2}\ud \langle \overline{\mM}_{\cdot \wedge
t_{1} } \rangle_{s} \right] \\
& \lesssim \EE \left[ \int_{0}^{3} F_{2 \mu} (\mZ_{s \wedge t_{1}}) \| z_{s
\wedge t_{1}} \|_{\gamma + \a, L}^{2}\ud s  \right] < \infty \;,
\end{aligned}
\end{equation*}
once more by applying \eqref{eqn:bd-qv-total} and by the previous step,
provided that $ \mu \leqslant \mf{c}_{1} (\nu) \ve^{-1} /4 $. Since $ \ve \in
(0, 1) $ can be chosen small at will, this proves the result.

\end{proof}

\begin{lemma}\label{lem:bound-weights} For any $ \gamma >0 $ and any
$ L \in \NN $  consider $ ( \varrho^{L}_{k} )_{k \in \ZZ^{d}}  $ as in \eqref{eqn:def-weights}, and $ \varrho_{k}
= (1 + |k|)^{2\gamma}$ for $ k \in \ZZ^{d} $. Then there exists a $ c(\gamma) >0 $ such
that uniformly over $ L $
\begin{equ}
\varrho^{L}_{k} \leqslant c(\gamma) \varrho^{L}_{k +l} \varrho_{l}, \qquad \forall k, l
\in \ZZ^{d}\;.
\end{equ}
\end{lemma}
\begin{proof}
We distinguish three elementary cases. If $ | k+l | \geqslant | k | $, then the result
is trivial. Otherwise we can either have $ L < | k+l | < | k | $ or $ | k+l
| \leqslant L $ and $ | k+l | < | k | $. In the first case we find by the triangle inequality and for
some $ c(\gamma) > 0 $
\begin{equ}
\varrho^{L}_{k} \leqslant (| k+l |+ | l | - L )^{2 \gamma} \leqslant
c(\gamma) (| k+l | - L)^{2 \gamma} (| l |+ 1)^{2 \gamma} = c(\gamma)
\varrho_{k+l}^{L} \varrho_{l}\;,
\end{equ}
where in the last equality we used the assumption $ | k+l | > L $. Instead,
in the second case, $ | k +l | \leqslant L $ we have
\begin{equ}
| l | \geqslant | k | - L \quad \Rightarrow \quad
\varrho_{k}^{L} \leqslant \varrho_{l}\;,
\end{equ}
so that the claim follows.
\end{proof}
\begin{lemma}\label{lem:reg-semi}
  Fix any $ L^{(1)}, L^{(2)} \in \NN$ such that $ L^{(1)} \geqslant L^{(2)} $ and $
  \psi  \ \colon \ [0, \infty) \to \RR^{m} $ such that $
  | \psi_{s}^{\alpha}  | \leqslant \zeta_{L^{(2)}} $ for all $ s \geqslant 0 $
and $ \alpha \in \{ 1, \dots, m \} $, and
define the time-inhomogeneous semigroup for $ \varphi \in L^{2}(\TT^{d};
\RR^{m}) $ (written in components for $ \alpha \in \{ 1, \dots, m \} $):
  \[ S_{s, t}^{L^{(1)}} \, \varphi^{\alpha} = e^{(t-s) \mL^{\alpha} +
\int_{s}^{t} \psi_{r}^{\alpha} \ud r} \, \Pi_{L^{(1)}}^{\myg}
\varphi^{\alpha}\;, \quad 0 \leqslant s < t < \infty \;,\]
where $ \mL^{\alpha} = -\nu^{\alpha}  (- \Delta)^{\a } $, for $
\a  \geqslant 1/2$ and $(\nu^{\alpha})_{\alpha =1}^{m} > 0 $.
Then for any $ \gamma \in \RR $ and $ \delta \in [0, \infty) $ there exists a
constant $ C(\gamma, \delta) $ such that for all $ 0 \leqslant s < t < \infty $
\begin{equs}
  & \| S_{s, t}^{L^{(1)}} \, \varphi \|_{\gamma + \delta, L^{(1)}} \leqslant C
(\nu_{\mathrm{min}}, \gamma,
  \delta)(t -s)^{-\frac{\delta}{2 \a }} \| \Pi^{\myg}_{L^{(1)}} \varphi
\|_{\gamma,L^{(1)}}\;, \\
  & \| S_{s, t}^{L^{(1)}} \, \varphi \|_{\gamma, L^{(1)}} \leqslant e^{-
  |t -s|(L^{(1)} - L^{(2)})} \| \Pi^{\myg}_{L^{(1)}} \varphi \|_{\gamma,L^{(1)}}\;. 
\end{equs}
\end{lemma}
\begin{remark}\label{rem:semigroup-lemma}
 The first bound of the lemma is not optimal, since it should actually improve
 as $ L^{(1)} - L^{(2)} $ increases (in analogy to the second one). We state this
 weaker version only because we do not require the improved bound.
\end{remark}
\begin{proof}
  Let us start with the first bound. We find
\begin{equ}
  |\mF S_{s, t}^{L^{(1)}} \varphi^{\alpha} |(k) =  e^{-(t -s) \nu^{\alpha}  | k |^{2
\a } - \int_{s}^{t} \psi_{r}^{\alpha} \ud r} | \hat{\varphi}(k) | \one_{\{ | k | >
  (L^{(1)} + 1)_{\alpha} \}}\;.
\end{equ}
Since $ \psi_{r}^{\alpha} \leqslant \zeta_{L^{(2)}} $ we have
\begin{equ}
(t -s) \nu^{\alpha} | k |^{2 \a } + \int_{s}^{t} \psi_{r}^{\alpha} \ud r \geqslant (t-s)
(\nu^{\alpha}| k |^{2 \a } - | L^{(2)} |^{2 \a })\;,
\end{equ}
so that from the inequality $ x^{2 \a } - y^{2 \a } \geqslant (x
- y)^{2 \a } $, for $
x \geqslant y \geqslant 0 $, we conclude that
\begin{equs}
  |\mF & S_{s, t}^{L^{(1)}} \varphi^{\alpha} |(k) \\
&  \leqslant e^{- (t-s)
\nu^{\alpha} (| k | -
  L^{(2)}_{\alpha})^{2 \a }}
| \hat{\varphi}(k)|  \\
& \leqslant \big( (t - s)^{\frac{1}{2 \a }}( |
k | - L^{(2)}_{\alpha} ) \big)^{-  \delta} \big(
(t-s)^{\frac{1}{2 \a }} (| k | - L^{(2)}_{\alpha}) \big)^{ \delta} e^{-
(t-s) \nu^{\alpha} (| k | - L^{(2)}_{\alpha})^{2 \a }} | \hat{\varphi}(k)|\;.
\end{equs}
Then, using the bound $ \big( t^{\frac{1}{2 \a }} x \big)^{ \delta} e^{-
\nu^{\alpha} t x^{2 \a }}
\leqslant C(\nu_{\mathrm{min}}) $ for some $ C(\nu_{\mathrm{min}}) > 0 $ we obtain that uniformly over $ t,
L^{(1)}, L^{(2)} $ and $ | k | >(L^{(1)} + 1)_{\alpha}  $
\begin{equs}
  \| S_{s, t}^{L^{(1)}} \varphi \|_{\gamma + \delta,L^{(1)}}^{2} & =
\sum_{\alpha =1}^{m} \sum_{| k | >
  (L^{(1)}+1)_{\alpha} } (| k | - L^{(1)}_{\alpha})^{2(\gamma + \delta)} | \mF
S_{s, t}^{L^{(1)}} \varphi^{\alpha} |^{2} (k) \\
& \lesssim_{\nu_{\mathrm{min}}} (t -s)^{- \frac{ \delta}{ \a  }} \| \varphi
\|_{\gamma, L^{(1)}}^{2} \;,
\end{equs}
where we used that $ L^{(1)} \geqslant L^{(2)} $. This proves the first bound. As
for the second one, we simply have for all $ \alpha \in \{ 1, \dots, m \} $
\begin{equs}
  \| S_{s, t}^{L^{(1)}} \varphi^{\alpha} \|_{\gamma,L^{(1)}}^{2} & \leqslant e^{- 2 (t -
  s) \nu^{\alpha} (L^{(1)}_{\alpha} - L^{(2)}_{\alpha} )^{2 \a }}
\sum_{| k | > (L^{(1)} + 1)_{\alpha}} (| k | - L^{(1)}_{\alpha})^{2 \gamma}| \hat{\varphi} (k) |^{2} \\
  & \leqslant e^{- 2 (t - s)(L^{(1)} - L^{(2)})^{2 \a }} \| \varphi \|_{\gamma,
  L^{(1)}} \;,
\end{equs}
which immediately proves the estimate since $ \a  \geqslant 1/2 $. 
\end{proof}

\subsection{Higher order polynomial regularity estimates}

The aim of this section is to provide a proof of the higher order regularity estimate
appearing in Corollary~\ref{cor:finite-lyap}.

In particular, we will need a polynomial bound on the $
H^{\a} $ norm, as opposed to the estimate on the $ H^{1/2} $ provided by the Lyapunov
functional $ \mf{G} $ and in the preceding Proposition~\ref{prop:reg-high-freq}. 
But the latter two provide estimates on exponential moments of $ \| \pi
\|_{\gamma} $, and as we will see this result can be
improved if we restrict to proving polynomial moments of $ \pi_{t} $ in higher
regularity spaces.

\begin{proposition}\label{prop:higher-regularity}
  Under the assumptions of Theorem~\ref{thm:lyap-func}, for any $ \a \geqslant 1, \gamma \in (0, \a + 1/2) $ and $ n \in \NN $ there exist
  $ s_{\star}, C(\gamma,n) \in (0, \infty) $ such
  that for all $ \pi_{0} \in S $
  \begin{equ}
    \sup_{t \geqslant s_{\star}} \EE \| \pi_{t} \|_{H^{\gamma}}^{n} \leqslant
    C(\gamma, n) \mf{G}(\pi_{0}) \;,
  \end{equ}
where $ \mf{G} $ is the Lyapunov functional in \eqref{eqn:def-G-lyap}.
\end{proposition}

\begin{proof}

  Let us fix $ t_{\star} $ as in Theorem~\ref{thm:lyap-func}. Then for any $
s_{\star} \geqslant 0 $ we can rewrite
  \begin{equ}
    \sup_{t \geqslant s_{\star}}\EE \| \pi_{t} \|_{H^{\gamma}}^{n} =
    \sup_{m \in \NN} \sup_{t \in [m t_{\star}, (m+1) t_{\star} ]} \EE \| \pi_{
    s_{\star} + t } \|_{H^{\gamma}}^{n} \;.
  \end{equ}
  Now, if we  prove that uniformly over all $ m \in \NN $ the following bound
holds
  \begin{equ}[eqn:aim-hr]
    \sup_{t \in [m t_{\star}, (m+1) t_{\star} ]} \EE_{m t_{\star}} \| \pi_{
    s_{\star} + t } \|_{H^{\gamma}}^{n} \leqslant \overline{C}(\gamma, n)
    \mf{G}( \pi_{m t_{\star}} )\;,
  \end{equ}
  then by Theorem~\ref{thm:lyap-func} we conclude that as desired
  \begin{equ}
    \sup_{t \geqslant s_{\star}}\EE \| \pi_{t} \|_{H^{\gamma}}^{n} \leqslant
    \overline{C} (\gamma, n) \sup_{m \in \NN} \EE \mf{G}( \pi_{m t_{\star}})
    \leqslant C(\gamma, n) < \infty \;.
  \end{equ}
  Hence we are left with verifying \eqref{eqn:aim-hr}, and since $
  \pi_{t}$ is a time-homogeneous Markov process we can reduce the problem to
  verifying the claim for $ m= 0 $:
  \begin{equ}[eqn:aim-hr-reduced]
    \sup_{t \in [0, t_{\star} ]} \EE \| \pi_{
    s_{\star} + t } \|_{H^{\gamma}}^{n} \leqslant \overline{C}(\gamma, n)
    \mf{G}( \pi_{0} ) \;.
  \end{equ}
To obtain this bound, let us recall the definition of the stopping times $
  \{ T_{i} \}_{i \in \NN} $ and the skeleton energy median process $
  (M_{t})_{t \geqslant 0} $ as in Section~\ref{sec:st-def}. Then we
  estimate for $ A_{i} = [T_{i}, T_{i+1}) $ by the Cauchy--Schwarz inequality
  \begin{equs}
    \EE \| \pi_{s_{\star} + t} \|^{n}_{H^{\gamma}} & = \sum_{i = 0}^{\infty} \EE \left[
      \one_{A_{i}} (s_{\star}+t) \| \pi_{s_{\star}+t} \|_{H^{\gamma}}^{n}
    \right] \\
    & \leqslant \sum_{i =1}^{\infty} \sqrt{\mathbf{p}_{s_{\star}+t}(i)} \cdot \EE \left[
    \sup_{T_{i} \leqslant s < T_{i+1}} \| \pi_{s} \|_{H^{\gamma}}^{2 n}
  \right]^{\frac{1}{2}} \;,
  \end{equs}
  with $ \mathbf{p}_{s_{\star}+t}(i) = \PP ( s_{\star}+t \in [T_{i}, T_{i+1})) $
  for $ i \in \NN $, and where we assumed that the sum can start from $ i =1
  $ since by Lemma~\ref{lem:consistency}, $ T_{1} \leqslant 3 $ and we can
  choose $ s_{\star} > 3 $.
  Now for $ t \in  [ T_{i}, T_{i+1}) $ we find that 
  \begin{equs}
    \| \pi_{t} \|_{H^{\gamma}}^{2} & \leqslant M_{T_{i}}^{2 \gamma} +
    \sum_{| k | > M_{T_{i}}} (1 + | k |)^{2 \gamma} | \hat{\pi}_{t} (k)
    |^{2} \\
    & \lesssim M_{T_{i}}^{2 \gamma} + \| \pi_{t} \|_{\gamma, M_{T_{i}}}^{2} \\
    & \lesssim M_{T_{i}}^{2 \gamma} + \| w_{t} \|_{\gamma,
    M_{T_{i}}}^{2} \;, \label{e:ref}
  \end{equs}
  where in the last line $ w_{t} $ is as in \eqref{eqn:def-h} and we made use
  of the estimate \eqref{e:bd-pi}.
  Hence, if we would prove that
  \begin{equ}[eqn:aim-fk-discrete]
    \EE \left[ M_{T_{i}}^{n \gamma} + \sup_{t \in A_{i}} \| w_{t}
    \|_{\gamma, M_{T_{i}}}^{n}  \right] \leqslant \overline{C} (\gamma, n)
    \mf{G}(\pi_{0})\;,
  \end{equ}
  then we could conclude that
  \begin{equs}
    \sup_{t \in [0, t_{\star}]} \EE \| \pi_{
    s_{\star} + t } \|_{H^{\gamma}}^{n} & \leqslant \overline{C}(\gamma, n)
    \mf{G}( \pi_{0} ) \left( \sup_{t \in [0, t_{\star}]} \sum_{i =
    1}^{\infty} \sqrt{\mathbf{p}_{s_{\star}+t}(i)} \right) \\
    & \leqslant \overline{C}(\gamma, n) C(t_{\star}, s_{\star})\mf{G}( \pi_{0} ) \;.
  \end{equs}
  by Lemma~\ref{lem:stopping-density-bounds}, which implies
  \eqref{eqn:aim-hr-reduced}, up to choosing a larger $ \overline{C} $.
  
  To conclude, we
  observe that \eqref{eqn:aim-fk-discrete} follows from
  Lemma~\ref{lem:high-reg-disc} below and Theorem~\ref{thm:lyap-func-discrete}.
  Indeed the latter theorem implies immediately that uniformly over $ i \in \NN $
  \begin{equ}
    \EE[M_{T_{i}}^{\gamma n}] \lesssim_{\gamma, n} \EE
    [\mf{F} (1/2, \bpi_{T_{i}})] \lesssim \mf{F}(1/2, \bpi_{0}) =
    \mf{G}(\pi_{0}) \;.
  \end{equ}
  Furthermore, applying first the second and then the first estimate of
  Lemma~\ref{lem:high-reg-disc} guarantees that for any $ i \in \NN^{+} $
  \begin{equs}
    \EE \left[\sup_{t \in A_{i}} \| w_{t}
    \|_{\gamma, M_{T_{i}}}^{n} \right] & = \EE \left[ \EE_{T_{i}} \left[ \sup_{t
        \in A_{i}} \| w_{t} \|_{\gamma, M_{T_{i}}}^{n}
    \right] \right] \\
    & \leqslant \mf{C} \EE \left[\mf{F}(1/2, \bpi_{T_{i}}) +\| w_{T_{i}} \|_{\gamma, M_{T_{i}}}^{n}
    \right] \\
    & \leqslant \mf{C}^{2} \EE \left[ \mf{F}(1/2, \bpi_{T_{i-1}}) \right] \\
    & \leqslant \overline{C}(\gamma, n) \mf{F} (1/2, \bpi_{0}) =
    \overline{C}(\gamma, n) \mf{G} (\pi_{0})\;,
  \end{equs}
  where in the last lines we used again
  Theorem~\ref{thm:lyap-func-discrete} (up to choosing a larger $
  \mf{C}$).
\end{proof}
In the next lemma we provide the key technical estimate of this section. 
\begin{lemma}\label{lem:high-reg-disc}
In the setting of Proposition~\ref{prop:higher-regularity},
  for any $ \a \geqslant 1, \gamma \in (0, \a+ 1/2), n \in \NN $ there exists a $ \mf{C} (\gamma, n)> 1 $
  such that uniformly over $ i \in \NN$ and with $ A_{i} = [T_{i},
  T_{i+1}) $
  \begin{equs}
    \EE_{T_{i}} \left[ \| w_{T_{i+1}} \|_{\gamma, M_{T_{i}}}^{n} \right] & \leqslant
    \mf{C} \mf{F}(1/2, \bpi_{T_{i}}) \;, \\
    \EE_{T_{i}} \left[ \sup_{t \in A_{i}} \| w_{t} \|_{\gamma,
    M_{T_{i}}}^{n} \right] & \leqslant \mf{C}(\mf{F}(1/2, \bpi_{T_{i}}) + \| w_{T_{i}} \|_{\gamma,
    M_{T_{i}}}^{n} ) \;.
  \end{equs}
\end{lemma}

\begin{proof}
  The result is true for $ \gamma \leqslant 1/2 $ by
Theorem~\ref{thm:lyap-func-discrete}, so we consider only the case $ \gamma >
1/2 $.
The proof of both estimates follows along similar lines. As for the first
estimate, we observe first that by Lemma~\ref{lem:h-jump-n}, since $ \gamma > 1/2 $,
for some $ C(\nu, \gamma ) > 0 $
  \begin{equ}
    \| w_{T_{i+1}} \|_{\gamma, M_{T_{i}}}^{2} \leqslant C(\nu, \gamma) (1 + \| w_{T_{i+1} -} \|_{\gamma,
    M_{T_{i-1}}}^{2}) \;.
  \end{equ}
  Now, for $ t \in [T_{i}, T_{i+1}) $ we follow the same decomposition that we
have used in \eqref{eqn:decomposition-w}, in  the
  proof of Proposition~\ref{prop:reg-high-freq}, namely for all $ t \in
[T_{i}, T_{i+1}) $:
  \begin{equ}
    w_{t} = S_{T_{i}, t}^{M_{T_{i}}} w_{T_{i}} + y_{t} +
    z_{t} \;,
  \end{equ}
  where $ S^{M_{T_{i}}}_{T_{i}, \cdot} $ is the semigroup defined in
\eqref{e:def-semi} and the terms $ y_{t} $ and $ z_{t} $ are defined in \eqref{eqn:def-y-z}, so that
  \begin{equ}
    \| w_{T_{i+1}-} \|_{\gamma,M_{T_{i}}} \leqslant \| S_{T_{i},
    T_{i+1}}^{M_{T_{i}}} w_{T_{i}}\|_{\gamma, M_{T_{i}}} +\| y_{T_{i+1}}
    \|_{\gamma, M_{T_{i}}}+\| z_{T_{i+1}} \|_{\gamma, M_{T_{i}}} \;.
  \end{equ}
In particular, both the claimed estimates now follow if we prove the following
bounds for some deterministic constant $ C > 0 $:
\minilab{e:hf-aim}
\begin{equs}
  \EE_{T_{i}} \left[ \| S_{T_{i}, T_{i+1}}^{M_{T_{i}}} w_{T_{i}}\|_{\gamma,
  M_{T_{i}}}^{n}  \right] & \leqslant  C \;, \label{e:hf1}\\
  \EE_{T_{i}} \left[ \sup_{t \in A_{i}}  \| S_{T_{i},t}^{M_{T_{i}}}
  w_{T_{i}}\|_{\gamma, M_{T_{i}}}^{n} \right] & \leqslant  C \|
  w_{T_{i}} \|_{\gamma, M_{T_{i}}}^{n} \;, \label{e:hf2} \\
  \EE_{T_{i}} \left[ \sup_{t \in A_{i}} \| y_{t} \|^{n}_{\gamma, M_{T_{i}}}
  \right] & \leqslant  C \;, \label{e:hf3} \\
  \EE_{T_{i}} \left[ \sup_{t \in A_{i}} \| z_{t} \|^{n}_{\gamma, M_{T_{i}}}
  \right] & \leqslant  C \mf{F} (1/2, \bpi_{T_{i}}) \;. \label{e:hf4}
\end{equs}
The proofs of these bounds follow roughly the same arguments as in the proof of
Proposition~\ref{prop:reg-high-freq}, up to the stochastic convolution term $
z_{t} $, which requires a separate treatment. The rest of the proof is devoted
to obtaining \eqref{e:hf-aim}, treating term by term.

  \textit{Step 1. Bound on the initial condition and on $ y_{t} $.} As for
\eqref{e:hf1}, we estimate by Lemma~\ref{lem:reg-semi}
  \begin{equ}[eqn:bd-ic]
    \| S_{T_{i-1}, T_{i}}^{M_{T_{i-1}}} w_{T_{i-1}}\|_{\gamma, M_{T_{i-1}}}
    \lesssim (T_{i} - T_{i-1})^{-\frac{\gamma}{2\a}} \| w_{T_{i-1}} \| \lesssim
    (T_{i} - T_{i-1})^{-\frac{\gamma}{2\a}} \;,
  \end{equ}
  so that the estimate follows from Lemma~\ref{lem:neg-mmts}. Similarly also
  \eqref{e:hf2}.

  Instead, for the bound \eqref{e:hf3} on the term \( y_{t} \), we follow verbatim the estimate
  \eqref{eqn:reg-prf-3}, which holds for any $ \gamma \in (0, 2\a) $, to obtain
  for some deterministic $ C(\gamma) \in (0, \infty) $
  \begin{equ}
    \sup_{T_{i-1} \leqslant t \leqslant T_{i}}\| y_{t} \|_{\gamma,
M_{T_{i-1}}} \leqslant C (\gamma) \;.
  \end{equ}

  \textit{Step 3. Bound on $ z_{t} $.} Here we must follow a different
  estimate than in the proof of Proposition~\ref{prop:reg-high-freq} and in
particular Lemma~\ref{lem:bd-z}, although we will use, through a bootstrap
argument, the results in the quoted proposition and lemma.
  Recall that we have, with $ \psi $ as in
  \eqref{eqn:def-alpha} and $ \sigma $ as in \eqref{eqn:def-sigma}:

\begin{equ}
\ud  z_{t}  = \big[ \mL z_{t} + \psi_{t} z_{t} \big] \ud
t + \sigma(u_{t}, \ud W_{t})\;, \qquad z_{T_{i-1}} = 0\;, \qquad \forall t \in
[T_{i-1}, T_{i}) \;.
\end{equ}
To further simplify the notation let us assume as usual that $ T_{i-1} = 0 $ and write $
T_{1} = T $ and $ M_{T_{i-1}} = M $. We represent $ z_{t} $ in its mild form
\begin{equ}
  z_{t} = \int_{0}^{t} e^{(t - s) \mL} \psi_{s} z_{s} \ud s +
  \int_{0}^{t}e^{(t-s) \mL} \sigma(u_{s}, \ud W_{s}) \;.
\end{equ}
Then, since $ \psi_{t} \leqslant \zeta_{M} $ we find that for some $
c(\a)> 0 $
\begin{equs}
  \sup_{0 \leqslant t \leqslant T} &\| z_{t} \|_{H^{\gamma}} \leqslant  
  \sup_{0 \leqslant t \leqslant T} \| z_{t} \|_{H^{\gamma}} \sup_{0 \leqslant t
  \leqslant T}  \int_{0}^{t} e^{-(t-s) (M+1)^{2 \a}} \zeta_{M} \ud s \\
  & \qquad \qquad \ \  + \sup_{0 \leqslant t \leqslant T} \bigg\|
  \int_{0}^{t}e^{(t -s) \mL} \sigma( u_{s} , \ud W_{s}) \bigg\|_{H^{\gamma}} \\
  \leqslant & \frac{M^{2 \a}}{(M+1)^{2 \a}} \sup_{0 \leqslant t \leqslant T} \|
    z_{t} \|_{H^{\gamma}} + \sup_{0 \leqslant t \leqslant T} \bigg\|
    \int_{0}^{t}e^{(t -s) \mL} \sigma( u_{s} , \ud W_{s}) \bigg\|_{H^{\gamma}}
    \\
  \leqslant & \left( 1 -  \frac{c(\a)}{M} \right)  \sup_{0 \leqslant t \leqslant T} \|
    z_{t} \|_{H^{\gamma}}+ \sup_{0 \leqslant t \leqslant T} \bigg\|
    \int_{0}^{t}e^{(t -s) \mL} \sigma( u_{s} , \ud W_{s}) \bigg\|_{H^{\gamma}}
    \;, \quad \qquad \label{e:z-end}
\end{equs}
where we used that
\begin{equs}
  M^{2 \a} & = (M +1)^{2 \a} - ( (M +1)^{2\a} - M^{2 \a}) \\
  & \leqslant (M +1)^{2 \a} - \inf_{\xi \in [M , M+1]} 2 \a \xi^{2\a -1} \\
  & \leqslant (M +1)^{2 \a} - 2 \a M^{2\a -1} \\
& \leqslant  (M +1)^{2 \a} - 2^{-2\a+2} \a (M+1)^{2\a -1} \;,
\end{equs}
since $ M+1 \leqslant 2M $.
Hence from \eqref{e:z-end} we overall estimate
\begin{equs}
  \sup_{0 \leqslant t \leqslant T} \| z_{t} \|_{H^{\gamma}}^{n} &
\lesssim_{\a}   M^{n} \sup_{0 \leqslant t \leqslant T} \bigg\|
    \int_{0}^{t}e^{(t -s) \mL} \sigma( u_{s} , \ud W_{s}) \bigg\|_{H^{\gamma}}^{n} \\
    & \lesssim_{\a} M^{2 n} +  \sup_{0 \leqslant t \leqslant T}\bigg\| \int_{0}^{t}e^{(t -s) \mL} \sigma( u_{s} ,
    \ud W_{s}) \bigg\|_{H^{\gamma}}^{2n} \;. \label{eqn:imprvd-reglrty-1}
\end{equs}
Now our efforts will concentrate on estimating the latter stochastic integral.
We observe that contrary to our original term $ z_{t} $, the semigroup $
e^{- t \mL} $ is deterministic, so that the convolution becomes simpler to
bound with classical tools.

To this purpose, let us introduce parameters $ \overline{\gamma}, \alpha, \beta > 0 $ such that
\begin{equ}[eqn:assu-parameters-hr]
  \overline{\gamma} + \beta = \gamma \;, \quad \beta < 2 \cdot \alpha \cdot \a\;, \quad
  \overline{\gamma} \in (0, 1/2) \;, \quad \alpha \in (0, 1/2) \;.
\end{equ}
This choice is possible since \( \gamma \in (0, \a + 1/2) \).
Our aim will be to obtain an estimate on the stochastic convolution in
\eqref{eqn:imprvd-reglrty-1} that depends on $ \| z_{t}
\|_{H^{\overline{\gamma}}} $
(so that we have improved the regularity from $ \overline{\gamma} $ to $
\overline{\gamma}+ \beta $) and then use Proposition~\ref{prop:reg-high-freq}
to control the $ H^{ \overline{\gamma}} $ norm, since $ \overline{\gamma} \in
(0, 1/2) $. To estimate uniformly in time the stochastic integral we use the
so-called ``factorisation method'', in which one rewrites the convolution as
follows for any $ \alpha \in (0, 1) $ and an appropriate normalisation constant
$ c_{\alpha}> 0 $:
\begin{equ}
  \int_{0}^{t}e^{(t -s) \mL} \sigma( u_{s} , \ud W_{s})   = c_{\alpha}
  F_{t} \;,
\end{equ}
where $ F_{t} $ is given by
\begin{equ}
  F_{t}  =  \int_{0}^{t} (t -s)^{\alpha-1} e^{(t -s) \mL}  \left( \int_{0}^{s}(s-r)^{-\alpha}e^{(s -r)
  \mL} \sigma( u_{r} , \ud W_{r}) \right)\ud s \;.
\end{equ}
Then via Lemma~\ref{lem:reg-semi} we bound for $ t \geqslant 0 $ and $
G_{s} =\int_{0}^{s}(s-r)^{-\alpha}e^{(s -r) \mL} \sigma( u_{r \wedge T} , \ud
W_{r}) $ (note that since we have stopped $ u $ at time $ T $, the process $ G_{s} $ is defined for all $ s
  \geqslant 0 $):
\begin{equ}
  \| F_{t}  \|_{H^{ \overline{\gamma} + \beta}} \lesssim  \int_{0}^{t} (t -
  s)^{\alpha-1 - \frac{\beta}{2\a}} \left\|  G_{s} \right\|_{H^{
  \overline{\gamma}}}   \ud s \;.
\end{equ}
Since $ \beta/2\a < \alpha $ by \eqref{eqn:assu-parameters-hr}, there exists a $ p (\alpha, \beta) \in (1,
\infty) $ such that for $ p^{\prime} $ the conjugate exponent satisfying $ 1/p +
1/p^{\prime} =1 $, we have
\begin{equ}
  \int_{0}^{t} (t - s)^{\alpha-1 - \frac{\beta}{2\a}} \left\|  G_{s}
  \right\|_{H^{\gamma}}   \ud s \leqslant t^{q}\| G \|_{L^{p^{\prime}}([0,t];
  H^{\gamma})} \;,
\end{equ}
where $ p $ is chosen sufficiently close to $ 1 $ such that
\begin{equ}
  q = \frac{1}{p} \left[p \left(\alpha - 1 - \frac{\beta}{2\a} \right) +1 \right] = \alpha - 1 -
  \frac{\beta}{2\a} + \frac{1}{p}  > 0\;.
\end{equ}
Therefore, we can bound
\begin{equ}
  \sup_{0 \leqslant t \leqslant T} \| F_{t} \|_{H^{ \overline{\gamma} + \beta}} \lesssim
  T^{q} \| G \|_{L^{p^{\prime} }([0, T]; H^{ \overline{\gamma}}) }\;.
\end{equ}
Then, since $ T \leqslant 3 $ (see Lemma~\ref{lem:consistency}) and assuming
without loss of generality that $ n \geqslant p^{\prime} $ we
use Jensen's inequality to obtain
\begin{equ}
  \EE  \left[ \sup_{0 \leqslant t \leqslant T} \| F_{t} \|^{n}_{H^{ \overline{\gamma}+
  \beta }} \right] \lesssim \sup_{0 \leqslant t \leqslant 3} \EE \left[ \| G_{t
    } \|_{H^{ \overline{\gamma}}}^{n} \right] \;.
\end{equ}
Now for the last term we use the vector-valued BDG inequality (see for
example \cite[Theorem 1.1]{MarinelliRockner16BDG}) to obtain 
\begin{equ}[eqn:bdg]
  \EE \left[ \| G_{t}^{t} \|_{H^{ \overline{\gamma}}}^{n} \right] \lesssim \EE
  \left[  \langle G^{t} \rangle_{t}^{\frac{n}{2}} \right] \;.
\end{equ}
Here $ G^{t}_{s} $ is the martingale
\begin{equation*}
\begin{aligned}
 G^{t}_{s} = \int_{0}^{s}(t-r)^{-\alpha}e^{(s -r) \mL} \sigma( u_{r \wedge T} , \ud
W_{r}) \;, \qquad \forall s \in [0, t] \;,
\end{aligned}
\end{equation*}
and $ \langle G \rangle_{t} $ indicates the scalar quadratic variation computed with
respect to the Hilbert space $ H^{ \overline{\gamma}} $, namely the unique
continuous increasing process $ s \mapsto A_{s}^{t} $, for $ s \in [0, t] $ such that $ \| G^{t}_{s}
\|_{H^{ \overline{\gamma}}}^{2} - A_{s}^{t} /2 $ is a martingale on $
[0, t] $. In our setting, the quadratic variation is given by
\begin{equation*}
\begin{aligned}
 \ud \langle G^{t} \rangle_{s} = 2 (t -s)^{- 2 \alpha} \overline{Q}
(u_{s \wedge T}) \ud t\;, 
\end{aligned}
\end{equation*}
where $ \overline{Q} $ is defined in \eqref{e:qbar}, only in the present case with $ L = 0 $ and $ \gamma
$ replaced by $ \overline{\gamma} $.
In particular, following the calculation that leads to \eqref{eqn:reg-prf-4}
(the only difference being that since $ L = 0 $ the terms $ \Delta^{\alpha,
\beta} L_{} $ are not present), we obtain that
\begin{equation*}
\begin{aligned}
 \ud \langle G^{t} \rangle_{s} \lesssim (t-s)^{- 2 \alpha} (1 + \|
w_{s} \|^{2}_{H^{ \overline{\gamma}}}) \;.
\end{aligned}
\end{equation*}
Therefore,  by \eqref{eqn:bdg} and since $ \alpha \in (0, 1/2) $ by assumption we can conclude that
\begin{equs}
   \EE \left[ \| G_{t}^{t} \|^{n}_{H^{ \overline{\gamma}}} \right] 
  & \lesssim  \left( M^{n \overline{\gamma}} +  \EE  \left[ \sup_{0 \leqslant s
\leqslant T}\| w_{s} \|^{n}_{ \overline{\gamma}, M}\right] 
\right)  \cdot \left(
\int_{0}^{t} (t -s)^{- 2 \alpha}  \ud s\right)^{\frac{n}{2}} \\
& \lesssim_{\alpha}  M^{n \overline{\gamma}} +  \EE
 \left[ \sup_{0 \leqslant s \leqslant T}\| w_{s} \|^{n}_{\overline{\gamma}, M}\right] \;.
\end{equs}
Now we use the uniform bound \eqref{eqn:uniform-bd} from Proposition~\ref{prop:reg-high-freq} to
conclude that
\begin{equ}
  \sup_{0 \leqslant t \leqslant 3}\EE \left[ \| G_{t}^{t} \|^{n}_{H^{
\overline{\gamma}}} \right] \lesssim_{n} \mf{F} (1/2, \bpi_{T_{i}}) \;,
\end{equ}
from which all the desired estimates follow.
\end{proof}

\section{Uniqueness of the invariant measure}\label{sec:uniqueness}

Our aim is to apply Harris' theorem to the
Markov process $ ( [\pi_{t}])_{t \geqslant 0} $. Throughout this section we work
under Assumption~\ref{assu:non-deg-uniq}.
We denote with $ (\mP_{t})_{t \geqslant 0} $ the transition probabilities of $
[\pi_{t}] $ on $ \mathbf{P} $:
\begin{equ}
  \mP_{t}([\pi], A) = \PP([\pi_{t}] \in A \,|\, [\pi_{0}] = [\pi] ) \;, \qquad \forall A
  \subseteq \mathbf{P},\; [\pi] \in \mathbf{P}\;.
\end{equ}
Harris' theorem applies if the Markov process possesses a Lyapunov functional,
which is the case by Theorem~\ref{thm:lyap-func}, and if level sets of the
Lyapunov functional satisfy a ``smallness'' property, which we now recall. 
\begin{definition}\label{def:d-small}
  For any $ t > 0 $ we say that a set $ A \subseteq \mathbf{P} $ is small for $
  \mP_{t} $ if there exists $ \delta \in (0, 1) $ such that
  \begin{equ}
    \| \mP_{t}([\pi], \cdot)- \mP_{t}([\nu], \cdot) \|_{\mathrm{TV},
    \mathbf{P}} \leqslant 1- \delta\;, \qquad \forall \, [\pi], [\nu] \in A \;.
  \end{equ}
\end{definition}
Recall, that in \eqref{e:tv} we have defined the total variation distance  between two positive
measures by $\| \mu - \nu \|_{\mathrm{TV}, \mX} =  \frac{1}{2} \sup_{A \in \mF
} | \mu(A) - \nu (A) |$, so there with an additional $ 1/2 $ normalisation
factor in contrast to the ``usual'' definition.
For convenience, we also state Harris' theorem below, adjusted to our setting. See
for example \cite{HairerMattingly11Harris}.
\begin{theorem}[Harris]\label{thm:Harris}
  Consider the Lyapunov functional $ \mf{G} $ 
  as in Theorem~\ref{thm:lyap-func}. If there exists $
  t_{\star} > 0 $ such that for every $ R > 0 $ the set
  \begin{equ}
    V_{R} = \{ [\pi] \in \mathbf{P}  \; \colon \; \mf{G}([\pi]) \leqslant R \}
  \end{equ}
  is small for $ \mP_{ t_{\star}} $, then there exists a unique invariant measure $
  \mu_{\infty} $ for the Markov process $ [\pi_{t}] $, and it satisfies
  for some $ C, \gamma > 0 $
  \begin{equ}
    \| \mP_{t} ([\pi], \cdot) - \mu_{\infty} \|_{\mathrm{TV},
    \mathbf{P}} \leqslant C\mf{G}([\pi]) e^{- \gamma t}\;, \qquad \forall \,
    [ \pi] \in \mathbf{P} \;, t \geqslant  0 \;.
  \end{equ}
\end{theorem}
Hence, the rest of this section is devoted to establishing the smallness property of level
sets of the Lyapunov functional, from which the spectral gap in Theorem~\ref{thm:uniq} promptly
follows. This is a standard consequence of the strong Feller property (Lemma~\ref{lem:feller})
and controllability (Lemma~\ref{lem:control}).

\begin{proposition}\label{prop:d-small}
  In the setting of Theorem~\ref{thm:uniq}, and in particular under
Assumption~\ref{assu:non-deg-uniq}, there exists a $ t_{\star} > 0 $ such that for every $ R > 0 $,
  $V_{R}$ is small for $ \mP_{t_{\star}} $. 
\end{proposition}

\begin{proof}

  The first step is to establish the smallness property
  locally around $ [\pi] \equiv [1] $. Here we write $ 1 $ for the unit function 
\begin{equation*}
\begin{aligned}
\TT^{d} \ni x \mapsto (1, \dots, 1) \in \RR^{m} \;.
 \end{aligned}
\end{equation*}
 For this reason we define
  \begin{equ}
    B_{\ve} = \{u \in  H^{\gamma_{0}}  \; \colon \; \| u - 1
    \|_{H^{\gamma_{0}}} < \ve \}\;,
  \end{equ}
  the ball of radius $ \ve \in (0, 1) $ about $ 1 $ in the $
  H^{\gamma_{0}} $ topology, for $ \gamma_{0} > d/2 $ as in
  Assumption~\ref{assu:non-deg-uniq}, so that in particular by Sobolev
embedding $ H^{\gamma_{0}} \subseteq C (\TT^{d}) $: this will be essential for the
  invertibility of the multiplication operator further on. Since we are interested in the
  projective dynamic we also define  
  \begin{equ}[eqn:b-ve-proj]
    B_{\ve}^{\mathrm{proj}} = \{ [\pi]   \; \colon \; \pi \in
    B_{\ve} \} \subseteq \mathbf{P} \;.
  \end{equ}
By Lemma~\ref{lem:feller} below, for any $ \delta \in
  (0, 1) $ there exists $ \ve \in (0, 1) $ such that
  \begin{equ}
    \| \mP_{1} ([\pi], \cdot) - \mP_{1}([\nu], \cdot)  \|_{\mathrm{TV},
    \mathbf{P}} < 1 - \delta \;,
    \qquad \forall \, [\pi], [\nu] \in B_{\ve}^{\mathrm{proj}}\;.
  \end{equ} 
Next, by Lemma~\ref{lem:control} we find an $ s_{\star} >0 $ such
that for any $ R > 0 $ there exists
$ \delta^{\prime}$ for which
  \begin{equ}
    \mP_{s_{\star}}([\pi] , B_{\ve}^{\mathrm{proj}}) \geqslant \delta^{\prime}
\;, \qquad \forall \, [\pi] \in V_{R}\;.
  \end{equ}
It follows immediately that for any $[\pi], [\nu] \in V_{R}$,
\begin{equ}
 \| \mP_{ s_\star +1} ( [\pi] , \cdot) -  \mP_{ s_{\star} +1}( [\nu] , \cdot)
    \|_{\mathrm{TV}, \mathbf{P}}
\le \delta' (1-\delta) + 1-\delta' = 1-\delta\delta'\;,
\end{equ}
thus concluding the proof.
\end{proof}
The next lemma establishes the small set property locally around the point
$ \pi \equiv 1 $: the result follows by proving the strong Feller property for
the solution to a new SPDE, which coincides with the linear SPDE
\eqref{eqn:main} for $ u_{0} $ close to $ u\equiv 1 $.

\begin{lemma}\label{lem:feller}
  In the setting of Theorem~\ref{thm:uniq}, and in particular under
Assumption~\ref{assu:non-deg-uniq}, for any $ \delta \in (0, 1) $ and $
  t > 0 $ there exists $ \ve  \in (0, 1)  $ such that
  \begin{equ}[e:Doeblin]
    \| \mP_{t} ([\pi], \cdot) - \mP_{t}([\nu], \cdot)  \|_{\mathrm{TV},
    \mathbf{P}} < 1 - \delta\;,
    \quad \forall \, [\pi], [\nu] \in B_{\ve}^{\mathrm{proj}} \;,
  \end{equ}
  with $ B_{\ve}^{\mathrm{proj}} $ (depending on $ \gamma_{0} $ as in
Assumption~\ref{assu:non-deg-uniq}) defined in \eqref{eqn:b-ve-proj}. 
\end{lemma}
In the upcoming proof, for a map $ f \colon X \to Y $ between two Banach spaces $ X, Y $, we define its
Fr\'echet derivative $D f(x) \in \mathbf{L}(X,Y)$ as the bounded linear map $X\to Y$ such that
  \begin{equ}
\lim_{\| h \|_{X} \to 0}
\frac{\|f (x + h) - f (x) - D f (x)h\|_{Y} }{\| h \|_{X}} = 0  \;, \quad \forall
    x \in X \;,
  \end{equ}
  provided that it exists.
  We say that a functional $ f $ as above is of class $ C^{m} $, for $ m \in
  \NN $, if it is $ m $ times Fr\'echet differentiable (the derivatives being not
  necessarily  uniformly bounded), and that it is smooth if it lies in $
  C^{m}  $ for all $ m \in \NN $.

\begin{proof}
  Let us start by reducing the problem to the study of the strong Feller
  property of \eqref{eqn:main}. Since $ [\pi_{t}] $ is a functional of $
  \pi_{t} $, which in turn is a functional of $
  u_{t} $, if we denote with $ \mQ_{t}(u, \cdot) $ the law of the solution $
  u_{t} $ to \eqref{eqn:main} with initial condition $ u_{0} = u \in
  L^{2} $, then for $ \pi, \nu \in S $ 
  \begin{equ}
    \| \mP_{t}([\pi], \cdot) - \mP_{t}([\nu], \cdot) \|_{\mathrm{TV},
    \mathbf{P}} \leqslant \|
    \mQ_{t}(\pi, \cdot) - \mQ_{t}(\nu, \cdot) \|_{\mathrm{TV}, L^{2}} \;,
  \end{equ}
so that it suffices to check \eqref{e:Doeblin} for $ \mQ_{t} $, uniformly over $ B_{\ve} = \{ u \in  H^{\gamma_{0}}  \; \colon \; \| u - 1
    \|_{H^{\gamma_{0}}} < \ve \} $, again with $
\gamma_{0} $ as in Assumption~\ref{assu:non-deg-uniq}.
  We localise our argument around $ B_{\ve} $ by
  constructing a dynamic that coincides with that of $ u_{t} $ only on $
  B_{\ve} $. Since we are working under Assumption~\ref{assu:non-deg-uniq}, we can
  represent $ u_{t} $ as the solution to
  \begin{equ}
    \ud u_{t}^{\alpha} = - \nu^{\alpha} (- \Delta)^{\a} u_{t}^{\alpha}  \ud t
+( G(u_{t}) \cdot \ud \xi)^{\alpha}\;, \qquad \forall \alpha \in \{ 1, \dots ,
m \} \;,
  \end{equ}
  where $ (G(u) \cdot \ud \xi)^{\alpha} = \sum_{\beta =1}^{m} G^{\alpha, \beta}
(u) \ud \xi^{\alpha, \beta}  $, and where  $ \xi $ is a space-time white noise,
that is a homogeneous Gaussian field with, formally, the covariance structure
\begin{equation*}
\begin{aligned}
\EE [ \xi^{\alpha, \beta} (x) \xi^{\alpha^{\prime} , \beta^{\prime}}
(y)] = \one_{\alpha = \alpha^{\prime}} \one_{\beta = \beta^{\prime}} \delta
(x-y) \;,
\end{aligned}
\end{equation*}
and $G^{\alpha, \beta}(u) \colon L^{2}(\TT^{d}; \RR ) \to L^{2}(\TT^{d}; \RR ) $ is given in Fourier coordinates by
  \begin{equ}
    \mF [G^{\alpha, \beta} (u ) \varphi] (k) = \sum_{l} \hat{u}^{\beta}( k - l)
 \Theta^{\alpha, \beta}_{l} \hat{\varphi}(l) \;,
  \end{equ}
  where we have defined $ \Theta^{\alpha, \beta}_{l} = (\Gamma^{\alpha,
\beta}_{l})^{\frac{1}{2}} $. In other words, $ G^{\alpha, \beta} (u) =
M_{u^{\beta} }
\circ K_{ \Theta}^{\alpha, \beta} $, where $ M_{u} \colon L^{2} \to L^{2} $ is
the multiplication operator in physical coordinates $ M_{u^{\beta}} \varphi =
u^{\beta} \varphi
$ and $ K_{\Theta}^{\alpha, \beta} \colon L^{2}
\to L^{2} $ is the multiplication operator in Fourier coordinates $ \mF
[K_{\Theta}^{\alpha, \beta} \varphi ](k) = \Theta^{\alpha, \beta}_{k}
\hat{\varphi}(k) $. We note that for $ u \in H^{\gamma_{0}}(\TT^{d}; \RR^{m}) $ fixed
we have $ G^{\alpha, \beta}( u) \in \mathbf{L} (L^{2}(\TT^{d}; \RR))$, where $
\mathbf{L}(X) $ is the space of bounded linear operators from a Banach space $ X $ into itself, since we have
\begin{equation} \label{e:bd-op}
\begin{aligned}
\| G^{\alpha, \beta}  (u) \varphi \| \leqslant \| u \|_{\infty} \|
K_{\Theta}^{\alpha, \beta} \varphi \| \lesssim \| u \|_{H^{\gamma_{0}}} \|
K_{\Theta}^{\alpha, \beta} \varphi \| \lesssim_{\Gamma} \| u
\|_{H^{\gamma_{0}}} \| \varphi \| \;,
\end{aligned}
\end{equation}
by Sobolev embedding, since $ H^{\gamma_{0}}(\TT^{d}; \RR^{m}) \subseteq C(\TT^{d};
\RR^{m}) $, and by Assumption~\ref{assu:non-deg-uniq} on the noise
coefficients (here we have merely used that the coefficients are bounded).

\textit{Step 1: Localisation.}  We now construct an operator-valued map 
$$ \overline{G}^{\alpha, \beta} \colon
H^{\gamma_{0}} (\TT^{d}; \RR^{m}) \to \mathbf{L}(L^{2}(\TT^{d}; \RR)) \;, $$
such that for any $ u \in H^{\gamma_{0}} (\TT^{d}; \RR^{m})$,
the linear operator $ \overline{G}^{\alpha, \beta}(u) $ is globally invertible
and satisfies the following properties for some $ C \in (0, \infty) $:
  \begin{equs}[eqn:prop-B-bar]
 {}   & \overline{G}^{\alpha, \beta}(u)  =  G^{\alpha, \beta}(u)\;, \qquad  & & \text{if } u \in
      B_{\ve_{1}} \;, \\ 
& \| [ \overline{G}^{\alpha, \alpha }(u)]^{-1} 
\|_{\mathbf{L}(H^{\gamma_{0}}(\TT^{d}; \RR); L^{2}(\TT^{d}; \RR))}
    \leqslant C \;,  \qquad & & \forall  u \in H^{\gamma_{0}}\;.
  \end{equs}
In addition, the parameter $
\ve_{1} \in (0, 1) $ must be chosen small enough \dash as it turns out, a
sufficient choice is given by
\begin{equation} \label{e:eps1}
\begin{aligned}
\ve_{1} = C_{\mathrm{s}}^{-1} /4 \;,
\end{aligned}
\end{equation}
where $ C_{\mathrm{s}} $ is the constant in the continuous embedding $
H^{\gamma_{0}} (\TT^{d}; \RR^{m}) \subseteq C(\TT^{d}; \RR^{m}) $, so that $
\| \varphi \|_{\infty} \leqslant C_{\mathrm{s}} \| \varphi
\|_{H^{\gamma_{0}}} $.
  We can construct a $
\overline{G}^{\alpha , \beta} $ with the desired properties as
  follows. We choose a smooth functional $ \varrho
  \colon H^{\gamma_{0}} \to [0, 1] $, such that
  \begin{equ}
    \varrho(u) = \begin{cases} 1 & \ \ \text{ if } u \in B_{\ve_{1}} \;, \\
    0 & \ \ \text{ if } u \in B_{2 \ve_{1}}^{c} \;.\end{cases} 
  \end{equ}
Such a choice of $ \varrho $ is always possible, for example because $ H^{\gamma_{0}} $ is a Hilbert
space and one can define $ \varrho (u) = \widetilde{\varrho} ( \| u -1
\|^{2}_{H^{\gamma_{0}}}) $ for a suitable smooth and compactly supported
function $ \widetilde{\varrho} \colon \RR \to [0,1] $.
  Then we define $ \overline{G}^{\alpha, \beta}  $ by
  \begin{equ}
    \overline{G}^{\alpha , \beta} (u) =  \varrho(u) \cdot (M_{u^{\beta}} \circ
K_{\Theta}^{\alpha, \beta})  + (1 -
    \varrho(u)) \cdot K_{\Theta}^{\alpha, \beta} \;,
  \end{equ}
where the dot indicates multiplication with a scalar.
  If we introduce the map
  \begin{equ}
    \texttt{g} \colon H^{\gamma_{0}}(\TT^{d}; \RR^{m})  \to
H^{\gamma_{0}}(\TT^{d}; \RR^{m})  \;, \qquad \texttt{g}(u) =
    \varrho(u) u + (1 - \varrho (u)) 1 \;, 
  \end{equ}
  then we can rewrite $ \overline{G}^{\alpha, \beta} $ as $
\overline{G}^{\alpha, \beta}(u) = M_{\texttt{g}(u)^{\beta}} \circ
K_{\Theta}^{\alpha, \beta} $.
  In particular, the Fr\'echet derivative with respect to $ u
$ of $ \overline{G}^{\alpha, \beta} $ can be computed as follows, for $ h \in H^{\gamma_{0}} $:
  \begin{equ}[e:dG]
    D \overline{G}^{\alpha , \beta} (u)  = M_{(D \texttt{g} (u) h)^{\beta}} \circ
K_{\Theta}^{\alpha, \beta} \in \mathbf{L}(L^{2} (\TT^{d}; \RR))\;.
  \end{equ}
Here the derivative of $ \texttt{g} $ is given by
  \begin{equ}
    D \texttt{g} (u) h =  \big(D \varrho(u) h \big) u +
    \varrho(u) h - \big(D \varrho (u) h\big) 1 \in H^{\gamma_{0}} \;,
  \end{equ}
  so that for some $ C > 0$
  \begin{equ}[eqn:bdd-derivative]
    \| \texttt{g} (u) \|_{H^{\gamma_{0}}}  \leqslant 1  \;, \qquad \| D
    \texttt{g}(u) h \|_{H^{\gamma_{0}}}   \leqslant C \| h
    \|_{H^{\gamma_{0}}} \;.
  \end{equ}
The fact that $ D \overline{G}^{\alpha, \beta} (u) \in
\mathbf{L} (L^{2} (\TT^{d}; \RR)) $ follows from \eqref{eqn:bdd-derivative} and
the same calculation as in \eqref{e:bd-op}. 

Hence, with this definition, $ \overline{G}^{\alpha, \beta} $ satisfies the first
property in \eqref{eqn:prop-B-bar}. Let us check that is satisfies also the
second one.  For the inverse  we have
  \begin{equ}
    \overline{G}^{\alpha, \alpha} (u)^{-1} = M_{ [\texttt{g}(u)^{\alpha}]^{-1}} \circ
K_{\Theta^{-1}}^{\alpha, \alpha } \;,
  \end{equ}
where we have defined for any $ \alpha, \beta \in \{ 1, \dots, m \} $
\begin{equation*}
\begin{aligned}
    [\texttt{g}(u)^{\beta}]^{-1} (x) = 1/[\texttt{g}(u)^{\beta}(x)] \;, \qquad \mF
[K_{\Theta^{-1} }^{\alpha, \beta} \varphi ](k) =  \big( \Theta^{\alpha,
\beta}_{k}\big)^{-1} \hat{\varphi}(k) \;.
\end{aligned}
\end{equation*}
The latter operator is defined in view of the lower bound on the correlation
coefficients in Assumption~\ref{assu:non-deg-uniq}. In particular, we can
follow the same calculations as in \eqref{e:bd-op} to obtain
that $ \overline{G}^{\alpha, \alpha} (u)^{-1} \in \mathbf{L}
(H^{\gamma_{0}} (\TT^{d}; \RR ), L^{2}(\TT^{d}; \RR)) $. More precisely, we can bound:
\begin{equation*}
\begin{aligned}
\| \overline{G}^{\alpha, \alpha} (u)^{-1} \varphi \|_{L^{2}} \lesssim \|
\texttt{g}^{-1} (u) \|_{\infty} \| K^{\alpha, \alpha}_{ \Theta^{-1}} \varphi \|
\lesssim_{\Gamma} \| \varphi \|_{H^{\gamma_{0}}} \;,
\end{aligned}
\end{equation*}
since $ \| \texttt{g}(u) - 1 \|_{\infty} \leqslant C_{\mathrm{s}} \| \texttt{g}(u) -1
\|_{H^{\gamma_{0}}} \leqslant 1/2 $, by Sobolev
embedding and by our choice of $ \ve_{1} $ in \eqref{e:eps1}, and where in the
last step we used the lower bound on the noise coefficients in
Assumption~\ref{assu:non-deg-uniq}. Hence overall the nonlinearity $ \overline{G} $ that we
have constructed does indeed satisfy \eqref{eqn:prop-B-bar} (the
first requirement of \eqref{eqn:prop-B-bar}is satisfied by construction).

Now the proof of the strong Feller property follows closely the proof of
\cite[Theorem 7.1.1]{DaPratoZabczyk96Ergodicity}. In particular, we observe
that Hypothesis 7.1.(iv) regarding the Hilbert--Schmidt norm of the semigroup
$s \mapsto e^{-s (- \Delta)^{\a}} $
in the quoted book corresponds to our assumption $
\mathbf{a} > d/2$ (with the Hilbert space $ H $ being $ L^{2}
(\TT^{d}; \RR^{m}) $). 

\textit{Step 2: Properties of the localised dynamic.} Let us consider the
solution $ \overline{u} $ to the nonlinear and nonlocal equation
  \begin{equ}[eqn:localized]
    \ud \overline{u}_{t}^{\alpha}  = - \nu^{\alpha} (- \Delta)^{\a}
\overline{u}_{t}^{\alpha} \ud t + \left(  \overline{G} (
    \overline{u}_{t}) \cdot  \ud \xi\right)^{\alpha} \;, \quad \overline{u}_{0} \in
    H^{\gamma_{0}} \;, \quad \alpha \in \{ 1, \dots, m \} \;.
  \end{equ}
Here, as usual, we have defined $ ( \overline{G} ( \overline{u}) \cdot \ud \xi)^{\alpha} =
\sum_{\beta =1}^{m} \overline{G}^{\alpha, \beta} ( \overline{u}) \ud
\xi^{\alpha, \beta} $.
Our objective is to obtain some a priori bounds on the solution: in particular
we would like to guarantee that with high probability, at least for small
times, the solution remains close to $ 1 $, if it is started in its
neighbourhood. Such estimates are classical, therefore we do not enter into
details. Instead, since \eqref{eqn:localized} is not dissimilar from
\eqref{eqn:main}, we refer to the proof of Lemma~\ref{lem:wp} to obtain that,
since we are assuming crucially $ \a > d/2 $ (so that we can choose $ \gamma =
\gamma_{0} $ in Lemma~\ref{lem:wp}) there exist $C, \zeta^{\prime} > 0 $ such
that:
  \begin{equs}
     \Big( \EE \Big[  \sup_{0 \leqslant s \leqslant t} \| \overline{u}_{s} -1\|_{H^{\gamma_{0}}}^{2}\Big] 
\Big)^{\frac{1}{2}} & \leqslant  \|
    \overline{u}_{0} -1 \|_{H^{\gamma_{0}}} + C t^{\zeta^{\prime} },
  \end{equs}
  One consequence of this bound is that if we define $ \tau $ to be the
stopping time
\begin{equ}
    \tau = \inf \{ t \geqslant 0  \; \colon \; u_{t} \not\in B_{\ve_{1}} \}\;,
  \end{equ}
then for any $ \delta \in (0, 1) $, we find a (deterministic) $ t \in (0,
1) $ such that 
  \begin{equ}
    \sup \{ \PP_{u_{0}} ( \tau < s )  \; \colon \; u_{0} \in B_{\ve_{1}/2} \;, s
    \leqslant t \} \leqslant \delta / 2\;.
  \end{equ}
Now, for any $ u_{0} = \overline{u}_{0} \in B_{\ve_{1}/2} $ we have $ u_{t} =
  \overline{u}_{t} $ up to the stopping time $ \tau $.
  Therefore, uniformly over $s \leqslant t$ and $ u, v 
  \in B_{\ve}$, with $ \ve \in (0, \ve_{1}/2) $, we find that
  \begin{equ}
    \| \mQ_{s} (u , \cdot ) - \mQ_{s}(v, \cdot)  \|_{\mathrm{TV}, L^{2}} \leqslant \|
    \overline{\mQ}_{s}(u, \cdot) - \overline{\mQ}_{s}(v, \cdot)
    \|_{\mathrm{TV}, L^{2}} + \delta /2 \;.
  \end{equ}
  In particular, the lemma is now proven if we show that for any $s \in
  (0, t) $ there exists $
  \ve \in (0, 1/4) $ such that
  \begin{equ}[e:fnl]
    \sup_{u, v \in B_{\ve}} \| \overline{\mQ}_{s}(u, \cdot) -
    \overline{\mQ}_{s}(v, \cdot) \|_{\mathrm{TV}, L^{2}} \leqslant 1 - \f{3\delta}2 \;.
  \end{equ}
 
\textit{Step 3: The Bismut--Elworthy--Li formula.} Finally, we set out to prove
\eqref{e:fnl}, which follows if we can establish the strong Feller property of $
\overline{u}_{t} $. Here we follow a classical approach via the
  Bismut--Elworthy--Li formula, see for example \cite[Lemma
7.1.3]{DaPratoZabczyk96Ergodicity}, which reads
\begin{equs}
    D_{u} & \overline{\mQ}_{t} (u, \psi) h \\
& = \EE_{u} \left[
      \psi( \overline{u}_{t}) \frac{1}{t} \sum_{\alpha =1}^{m} \int_{0}^{t} \int_{\TT^{d}} \left[
[\overline{G}^{\alpha, \alpha} (
\overline{u}_{s})]^{-1}  \big(D_{u} \overline{u}_{s}^{\alpha} h \big) \right] (x)
\cdot \ud   \xi^{\alpha, \alpha} ( x, s)\right] \;. 
\end{equs}
Here we have written $ \EE_{u} $ for the expectation under the law of $
\overline{u}_{t} $ started in $ \overline{u}_{0} = u $. Furthermore, $
D_{u} \overline{u}^{\alpha} $ denotes the Fr\'echet derivative of the $
\alpha $-th component of the solution $
\overline{u}_{t} $ with respect to its initial condition.
  In particular, this formula allows us to bound
  \begin{equs}
    |  D_{u} \overline{\mQ}_{t}&(u, \psi) h |^{2} \\
& \lesssim  \| \psi
    \|_{\infty}^{2} t^{-2} \sum_{\alpha=1}^{m} \EE_{u} \left[\int_{0}^{t}
\int_{\TT^{d}} | [\overline{G}^{\alpha, \alpha} (
      \overline{u}_{s})]^{-1} \big(D_{u} \overline{u}_{s}^{\alpha} h\big)  |^{2}(x) \ud s \ud
    x \right] \\
    & \lesssim \| \psi \|_{\infty}^{2} t^{-2} \int_{0}^{t} \EE_{u} \|
    \big( D_{u} \overline{u}_{s} h \big)
\|_{H^{\gamma_{0}}(\TT^{d}; \RR^{m})}^{2} \ud s\;,
    \label{eqn:bel}
  \end{equs}
where in the last line we have used the second property of $ \overline{G} $ in
\eqref{eqn:prop-B-bar}.

  Hence to conclude we have to find a bound on the $ H^{\gamma_{0}}
(\TT^{d}; \RR^{m})$ norm of
  the differential of the flow to \eqref{eqn:localized}. If we set $
r_{t}^{\alpha} =  D_{u}
  \overline{u}_{t}^{\alpha} h $ we find that
  \begin{equ}
    \ud r_{t}^{\alpha} = - \nu^{\alpha} (- \Delta)^{\a} r_{t}^{\alpha} \ud t +
 \left( D \overline{G} ( \overline{u}_{t})
    r_{t} \cdot \ud \xi\right)^{\alpha}  \;, \qquad r_{0}^{\alpha} =
h^{\alpha} \in H^{\gamma_{0}}(\TT^{d}; \RR)\;,
  \end{equ}
where we have written as usual $ \left( D \overline{G} ( \overline{u}_{t})
    r_{t} \cdot \ud \xi\right)^{\alpha} = \sum_{\beta=1}^{m}
\big(D \overline{G}^{\alpha, \beta} ( \overline{u}_{t}) r_{t}\big) \ud
\xi^{\alpha, \beta}$.
  The equation for $ r $ can in turn can be rewritten via \eqref{e:dG} as
  \begin{equ}
    \ud r_{t}^{\alpha}  = - \nu^{\alpha} (- \Delta)^{\a}  r_{t}^{\alpha} \ud t
+ \sum_{\beta=1}^{m} M_{( D \texttt{g} ( \overline{u}_{t}) h)^{\beta} } \circ K_{\Theta}^{\alpha, \beta} \ud \xi^{\alpha, \beta}
     \;.
  \end{equ}
  Now, for simplicity let us define for $ \overline{u}_{0}, h \in H^{\gamma_{0}}$ fixed
    $v_{t} = D \texttt{g} ( \overline{u}_{t}) h$.
  Then, if we rewrite the equation for $ r_{t} $ in its mild formulation, we
can follow the same calculations as in Lemma~\ref{lem:wp} to obtain for some $
\zeta^{\prime} > 0 $ and $ \gamma_{1} < \gamma_{0} $
\begin{equ}
  \EE \| r_{t} \|_{H^{\gamma_{0}}}^{2} \lesssim \| r_{0}
  \|_{H^{\gamma_{0}}}^{2} + t^{\zeta} \sup_{0 \leqslant s \leqslant t} \EE \|
  v_{s} \|_{H^{\gamma_{1}}}^{2} \lesssim \|
  h \|_{H^{\gamma_{0}}}^{2}(1 +t^{\zeta^{\prime}})\;,
\end{equ}
where in the last step we have used \eqref{eqn:bdd-derivative} to bound $
v_{s} $. 
We are now ready to conclude our argument. From \eqref{eqn:bel} we deduce that
\begin{equ}
  \sup_{\| h \|_{H^{\gamma_{0}}} \leqslant 1} | D_{u}
  \overline{\mQ}_{t} (u , \psi) h | \lesssim \| \psi
  \|_{\infty}^{2} t^{- 1} (1 + t^{\zeta^{\prime} }) \;,
\end{equ}
so that for some $ C>0 $ and uniformly over $ u, v \in B_{\ve} $
\begin{equ}
  | \overline{\mQ}_{t} (u, A) - \overline{\mQ}_{t} (v, A) | \leqslant C t^{- 1} (1
  + t^{\zeta^{\prime} }) \| u - v \|_{H^{\gamma_{0}}} \leqslant C  t^{- 1} (1 +
  t^{\zeta^{\prime}}) \ve\;,
\end{equ}
from which we deduce the claim.
\end{proof}
The second ingredient used to establish the small set property is a
form of controllability.
\begin{lemma}\label{lem:control}
  In the setting of Theorem~\ref{thm:uniq}, and in particular under
Assumption~\ref{assu:non-deg-uniq}, for any $ \ve \in (0, 1
), R > 0 $ there exist $ s_{\star} > 0  $
  and $ \delta^{\prime} \in (0, 1) $ such that
  \begin{equ}
    \mP_{s_{\star}}([\pi] , B_{\ve}^{\mathrm{proj}}) \geqslant \delta^{\prime}
\;, \qquad  \forall \; [\pi] \in V_{R}\;.
  \end{equ}
\end{lemma}

\begin{proof}
  The result follows from the support theorem for \eqref{eqn:main} by solving a
control problem.
To this aim, for any $s_{\star} >0 $ (we will fix this parameter later on), we denote by $ \mathcal{CM}
\subseteq L^{2}(\TT^{d}; (\RR^{m})^{\otimes 2}) $ the
  Cameron--Martin space  associated to the noise $ \dot{W} $ on the time
  interval $ [0, s_{\star}] $ (for convenience, we omit the dependence on $
  s_{\star} $ in the notation), characterised by the norm
\begin{equ}
  \| h \|_{\mathcal{CM}}^{2} = \int_{0}^{s_{\star}} \sum_{\alpha, \beta
=1}^{m} \sum_{k \in \ZZ^{d}} (\Gamma^{\alpha, \beta}_{k})^{-1} |
  \hat{h}^{\alpha, \beta}_{s}(k) |^{2} \ud s \;.
\end{equ}
Then for $ h \in \mathcal{C M}$, we denote by $
\Phi_{t}^{h} \pi $ the flow of the following deterministic, controlled
equation, for any initial datum $ u \in L^{2} $:
\begin{align*}
  \partial_{t} \Phi_{t}^{h} u &  = \mL \Phi_{t}^{h}
u + h  \Phi_{t}^{h} u, \qquad  \Phi_{0}^{h}
u = u \;, \qquad \forall t \in [0, s_{\star}] \;.
\end{align*}
Now the support theorem, see for example \cite[Theorem
2.1]{BallyMilletSole1995Support} (the setting of this work is slightly
more complex than ours, since the noise is space-time white, but the 
key technical result \cite[Proposition 2.2]{BallyMilletSole1995Support} adapts
to our setting), guarantees that
\begin{equation*}
\begin{aligned}
\supp ( \mQ_{s_{\star}}(u, \cdot) ) = \overline{ \{
\Phi_{s_{\star}}^{h} u  \; \colon \; h \in \mathcal{CM} \}
}^{H^{\gamma_{0}}} \;, \qquad \forall
u \in H^{\gamma_{0}}(\TT^{d}; \RR^{m}) \;.
\end{aligned}
\end{equation*}
Here we indicate with $ \mQ_{s} (u , \cdot) $ the law in $
H^{\gamma_{0}} $ of the solution $ u_{s} $ to \eqref{eqn:main} started in $ u
$. Recall that the support of a positive measure $ \mu $ on a metric space
$ (\mX, d) $ is given by
\begin{equation*}
\begin{aligned}
\supp (\mu) = \{ x \in \mX  \; \colon \; \mu (B)> 0 \text{ for all open sets }
B \text{ such that } x \in B \} \;.
\end{aligned}
\end{equation*}
Now we observe that if we prove that for $ B^{\mathrm{sym}}_{\ve} = B_{\ve}
\cup (- B_{\ve}) $ we have
\begin{equation} \label{e:Qaim}
\begin{aligned}
B_{\ve/2}^{\mathrm{sym}} \cap \supp (\mQ_{s_{\star}}(\pi, \cdot)) \neq 0 \;,
\qquad \forall \,  \pi \in  V_{R} \;,
\end{aligned}
\end{equation}
then also $ B_{\ve/2}^{\mathrm{proj}}  \in
\supp(\mP_{s_{\star}}([\pi], \cdot)) $ for all $ \pi \in S \cap V_{R} $. In
particular, if we show that \eqref{e:Qaim} is true, then our
result holds in view of Lemma~\ref{lem:cont-ic}, which guarantees the lower
bound of the probability to be uniform over $ \pi \in V_{R} $ (since $
V_{R} $ is compact in $ S $ and $ B_{\ve}^{\mathrm{sym}}$ is open in $
H^{\gamma_{0}} $, so that $ \mP_{t}([\pi], B_{\ve}^{\mathrm{proj}}) $ is
lower-semicontinuous in $ [\pi] $ by the Portmanteau theorem). Therefore, our aim is now to prove
\eqref{e:Qaim}, and in particular it suffices to show that
there exists an $ s_{\star} (R) > 0 $ such that for each $ \pi
\in V_{R} $ we can find a control $ h $ satisfying
\begin{equs}[eqn:aim-ctrl-har]
  \Phi^{h}_{s_{\star}} (\pi) / \|
  \Phi^{h}_{s_{\star}} (\pi) \| &  \in
  B_{\ve/2}^{\mathrm{sym} } \;, \qquad & \forall \pi \in V_{R} \;.
\end{equs}

One possible construction of such $ h$ is as follows. 
We start by observing that from the definition of $
V_{R} $ we have 
\begin{equ}
 M =  M (\pi) \leqslant \lfloor \kappa_{0}^{-1} \log{R} \rfloor = n_{R} \in \NN \;,
\end{equ}
with $ \kappa_{0} $ as in Theorem~\ref{thm:lyap-func}. In particular, there
exists an $ \alpha_{0} \in \{ 1, \dots, m \} $ such that $ \|
\Pi^{\myl}_{M(\pi)} \pi^{\alpha} \| \geqslant \eta $, for some deterministic $
\eta > 0 $. Our argument then proceeds in two steps. The main point will be
to prove that there exists a time $ s_{\star}^{(1)} > 0 $ such that, for
appropriate control $ h $, at time $ s_{\star}^{(1)} $ we have $ \|
\Pi^{\myl}_{1} \pi_{s_{\star}^{(1)}}^{\alpha_{0}} \| \geqslant
\ve \| \Pi^{\mygeq}_{1} \pi_{s_{\star}^{(1)}}^{\alpha_{0}}  \| $. That is, at
least the $ \alpha_{0} $-th component is concentrated in the $ 0 $-th frequency
level. We then use this result to prove that by a time $ s_{\star}^{(2)} >
s_{\star}^{(1)} $ and for an appropriate choice of $ h $ we have $ \| \Pi^{\myl}_{1}
\pi_{s_{\star}^{(2)}} \| \geqslant \ve \| \Pi^{\mygeq}_{1}
\pi_{s_{\star}^{(2)}} \| $, so that all components are now close to the unit,
at least in $ L^{2} $ norm.
 
We start by constructing the control $ h $ on the first time step $ [0,
s_{\star}^{(1)}) $. Let us fix a time
horizon $ t_{\star} \in (0, 1) $ to be fixed later on. Then consider
\begin{equation*}
\begin{aligned}
h^{\alpha_{0}, \alpha_{0}}(t, x) = P_{t} \Pi_{M}^{\myl} \pi^{\alpha_{0}} (x)
\;, \quad h^{\alpha, \beta} (t, x) = 0 \;, \quad \forall (\alpha, \beta) \neq
(\alpha_{0}, \alpha_{0})\;, \ \  t \in [0, t_{\star})\;,
\end{aligned}
\end{equation*}
where $ P_{t} $ is the semigroup $ P_{t}= e^{- \nu^{\alpha_{0}} (- \Delta)^{\a}t } $.
Note that such $ h $ lies in the Cameron--Martin space $ \mC \mM $ by
Assumption~\ref{assu:non-deg-uniq}. Further, there exists a constant $
C(R) >0 $ such that $$ \sup_{0 \leqslant s \leqslant 1} \left\{  \| \Phi^{h}_{s} \pi \|
+ \| h_{s} \|_{\infty}\right\} \leqslant C(R) \;. $$  Then we can compute
\begin{equation*}
\begin{aligned}
\int_{\TT^{d}} \Phi^{h}_{t} \pi (x) \ud x = \int_{\TT^{d}} \pi (x) \ud x +
\int_{0}^{t} \int_{\TT^{d}} h_{s}(x) \left[ P_{s} \pi + \int_{0}^{s} P_{s -r}
(h_{r} u_{r}) \ud r \right] (x) \ud x \ud s \;.
\end{aligned}
\end{equation*}
We can estimate the last term as follows:
\begin{equation*}
\begin{aligned}
\int_{0}^{t_{\star}} & \int_{\TT^{d}} h_{s}(x) \left[ P_{s} \pi + \int_{0}^{s} P_{s -r}
(h_{r} u_{r}) \ud r \right] (x) \ud x \ud s \\
& \geqslant \int_{0}^{t_{\star}} \| P_{s} \Pi_{M}^{\myl} \pi \|^{2} \ud s -
\int_{0}^{t_{\star}} \| h_{s} \|_{\infty} \int_{0}^{s} \| h_{r} \|_{\infty} \|
u_{r}\| \ud r \ud s \\
& \geqslant e^{- 2  \zeta_{M} t_{\star}} t_{\star} \eta -
\frac{t^{2}_{\star}}{2} C^{3} \geqslant c_{0}(\eta, \delta, R) > 0
\end{aligned}
\end{equation*}
for some constant $ c_{0}(\eta, \delta, R) $, provided that $ \delta $ is chosen
sufficiently small. Now assume that $ \left\vert \int_{\TT^{d}} \pi (x) \ud x
\right\vert < c_{0}/2 $ (otherwise set, instead of the present choice of $
h $, the trivial $ h = 0 $ on \( [0, t_{\star}) \)). In this way we have
\begin{equation*}
\begin{aligned}
 \left\vert \int_{\TT^{d}}\Phi^{h}_{t_{\star}} \pi (x) \ud x \right\vert
\geqslant c_{0}/2 \;.
\end{aligned}
\end{equation*}
Now by setting $ h = 0 $ on $ [t_{\star}, s_{\star}^{(1)}) $, so that by
choosing $ s_{\star}^{(1)} $ sufficiently large we obtain the desired $
\| \Pi^{\myl}_{1} (\Phi^{h}_{s_{\star}^{(1)}} \pi )^{\alpha_{0}} \| \geqslant \ve
\| \Pi^{\mygeq}_{1} (\Phi^{h}_{s_{\star}^{(1)}} \pi )^{\alpha_{0}} \|$.
Therefore we find that at time $ s_{\star}^{(1)} $ (up to choosing a 
larger $ s_{\star}^{(1)} $)
\begin{equation*}
\begin{aligned}
\pi_{s_{\star}^{(1)}} = \frac{\Phi^{h}_{s_{\star}^{(1)}} \pi}{\|
\Phi^{h}_{s_{\star}^{(1)}} \pi \|}
\end{aligned}
\end{equation*}
satisfies $ \| \pi_{s_{\star}^{(1)}} - \kappa \| \leqslant \ve/2 $, where $
\kappa \in L^{2}(\TT^{d}; \RR^{m}) $ is the constant vector $ \kappa =
(\kappa_{\alpha})_{\alpha =1}^{m} $ and by our previous discussion there exists a deterministic $
\underline{\kappa} $ such that $ |\kappa^{\alpha_{0}}| \geqslant
\underline{\kappa}>0$.

Finally, by choosing any $ s_{\star}^{(2)} > s_{\star}^{(1)} $, and since $
\kappa $ is nonzero, we can construct a time-independent control $ h $ on $
[s_{\star}^{(1)}, s_{\star}^{(2)}) $ for the ODE
\begin{equation*}
\begin{aligned}
\partial_{t} y^{\alpha} = \sum_{\beta=1}^{m} y^{\beta} h^{\alpha, \beta}\;,
\qquad y^{\alpha}_{s_{\star}^{(1)}} = \kappa^{\alpha} \;,
\end{aligned}
\end{equation*}
such that $ y_{s_{\star}^{(2)}}^{\alpha} =
y_{s_{\star}^{(1)}}^{\alpha_{0}} $ for all $ \alpha \in \{ 1,
\dots,m\} $. Up to choosing $ \ve $ sufficiently small, this guarantees that
under such choice of control
\begin{equation*}
\begin{aligned}
\min \left\{ \| \Phi^{h}_{s_{\star}^{(1)} } \pi / \|  \Phi^{h}_{s_{\star}^{(1)}
} \pi \| - 1 \| \;,  \| \Phi^{h}_{s_{\star}^{(1)} } \pi / \|
\Phi^{h}_{s_{\star}^{(1)} } \pi \| + 1 \| \right\} \leqslant \ve \;.
\end{aligned}
\end{equation*}
Finally, by Schauder estimates for a sufficiently
large time $
s_{\star}> s_{\star}^{(2)} $ and $ h =0 $ on $ [s_{\star}^{(2)}, s_{\star}] $, we find $
\pi_{s_{\star}} \in B^{\mathrm{sym}}_{\ve} $ as desired. 
\end{proof}

\section{Basic properties of the angular component} \label{sec:basic}

The first result of this section establishes well-posedness and continuity with
respect to the initial data for \eqref{eqn:main}.

\begin{lemma}\label{lem:wp}
Under the assumptions of Theorem~\ref{thm:lyap-func}, for any $ u_{0} \in L^{2}
(\TT^{d}; \RR^{m}) $ there exists a unique mild solution to \eqref{eqn:main} for all $ t \geqslant
0 $. Furthermore, for every
$$  \gamma \in [0, \a + \gamma_{0} -d/2)\;, $$  and $ p \geqslant 2,
T > 0 $ there exists a $ C(T, \gamma, p ) $ such that
\begin{equation*}
\begin{aligned}
\EE \left[ \sup_{0 \leqslant s \leqslant T} \| u_{s}
\|_{H^{\gamma}}^{p} \right]^{\frac{1}{p}} \leqslant C(T, \gamma, p) \|
u_{0} \|_{H^{\gamma}} \;.
\end{aligned}
\end{equation*}
\end{lemma}
\begin{proof}[(sketch of)]
This result is entirely classical, so we refrain from proving it completely. We
only provide a short description of how the condition $ \gamma < \a +
\gamma_{0} -d/2 $ appears. We can represent the solution $ u_{t} $ in mild form
as
\begin{equation*}
\begin{aligned}
u_{t} = e^{t \mL} u_{0} + \int_{0}^{t} e^{(t -s) \mL}  u_{s} \cdot \ud W_{s} \;,
\end{aligned}
\end{equation*}
where $ e^{t \mL} \varphi^{\alpha} = \exp ( - \nu^{\alpha}(- \Delta)^{\a} t)
\varphi^{\alpha} $ for every $ \alpha \in \{ 1, \dots, m \} $. Then the crucial point of the
proof is to estimate the stochastic convolution appearing above. Here we use
the so-called ``factorisation method'' (see also the proof of
Lemma~\ref{lem:high-reg-disc}).
Namely, if for some $ \zeta \in (0, 1) $ we
define (assuming for the moment that the integral is defined, which will be
justified below) 
\begin{equation*}
\begin{aligned}
\mZ_{s} =\int_{0}^{s}(s-r)^{- \zeta }e^{(s -r) \mL} u_{r} \cdot \ud
W_{r} \;, 
\end{aligned}
\end{equation*}
then we find that for some constant $ c_{\zeta} \in (0, 1) $
\begin{equation*}
\begin{aligned}
\int_{0}^{s} e^{(s-r)\mL} u_{r} \cdot \ud W_{r} = c_{\zeta}\int_{0}^{s} (s -r)^{\zeta -1} e^{(s
- r) \mL} \mZ_{r} \ud r \;.
\end{aligned}
\end{equation*}
Therefore, by H\"older's inequality and Lemma~\ref{lem:reg-semi} we find that
\begin{equation*}
\begin{aligned}
 \sup_{0
\leqslant s \leqslant t} \Bigg\| &\int_{0}^{s} e^{(s-r)\mL} u_{r} \cdot \ud
W_{r} \Bigg\|_{H^{\gamma}}  \\
& \lesssim \sup_{0 \leqslant s \leqslant t}
\int_{0}^{s} (s -r)^{\zeta -1 - (\gamma - \gamma_{1}) /2\a} \|
\mZ_{r} \|_{H^{\gamma_{1}}} \ud r\\
& \lesssim  \sup_{0
\leqslant s \leqslant t}  \left( \int_{0}^{s} (s - r)^{(\zeta - 1 -
(\gamma - \gamma_{1}) /2\a) p} \ud r
\right)^{\frac{1}{p}} \cdot \left( \int_{0}^{s} \| \mZ_{r} \|_{H^{
\gamma_{1}}}^{q} \ud r \right)^{\frac{1}{q}} \;,
\end{aligned}
\end{equation*}
where $ \gamma_{1} < \gamma $ and $ p,q \in [1, \infty] $ are conjugate, so
that $ 1/p + 1/q =1$. 
Now, to close our argument we must choose appropriate parameters $
\gamma_{1}, \zeta $ and $ p $. First, we fix any arbitrary $ \gamma_{1} \in
(0, \gamma)$ such that
\begin{equation} \label{e:gamma1}
\begin{aligned}
 \gamma -  \a < \gamma_{1} < \gamma_{0} - d/2\;,
\end{aligned}
\end{equation}
which is possible in view of our assumption \(\gamma < \a + \gamma_{0} - d/2
\). Next we fix $ \zeta \in (0, 1/2) $ such
that $ \zeta - (\gamma - \gamma_{1})/2\a > 0 $. This choice is now possible
in view of the upper bound in \eqref{e:gamma1}. Finally, we choose $ p \in
(1, \infty) $ sufficiently close to $ 1 $ such that
\begin{equation*}
\begin{aligned}
\zeta^{\prime}-1 \eqdef (\zeta - 1 - (\gamma_{0} - \gamma_{1}) /2\a) p > -1 \;,
\end{aligned}
\end{equation*}
which is now possible in view of all our choices above. With these choices
taken, we can now conclude that
\begin{equation*}
\begin{aligned}
\EE \left[ \sup_{0
\leqslant s \leqslant t} \Bigg\| \int_{0}^{s} e^{(s-r)\mL} u_{r}\cdot \ud
W_{r} \Bigg\|_{H^{\gamma_{0}}}^{2}  \right] \lesssim t^{2 \zeta^{\prime}} \EE \left[
\left( \int_{0}^{t} \| \mZ_{s} \|_{H^{
\gamma_{1}}}^{q} \ud s \right)^{\frac{2}{q}}\right] \;.
\end{aligned}
\end{equation*}
The result therefore follows if we can prove that for any $ t > 0 $ and any $
q \in (2, \infty) $
\begin{equ}
\sup_{0 \leqslant s \leqslant t} \EE \| \mZ_{s} \|_{H^{
\gamma_{1}}}^{q} < \infty \;.
\end{equ}
Since $ \mZ_{s} $ is a stochastic integral, by BDG this reduces to bounding
its quadratic variation, as $ \EE \| \mZ_{s} \|_{H^{\gamma_{1}}}^{q} \lesssim 
\EE \langle \mZ^{s} \rangle^{\frac{q}{2}}_{s} $, where $ \langle
\cdot \rangle $ denotes the scalar quadratic variation of an $
H^{\gamma_{1}} $-valued martingale (see also the analogous argument leading to
\eqref{eqn:bdg}), and for $ 0 \leqslant r \leqslant s $ we write $
\mZ^{s}_{r}=\int_{0}^{r}(s-h)^{- \zeta }e^{(s -h) \mL}  u_{h} \cdot \ud W_{h} $.
In our setting, we can bound the quadratic variation as follows for any $ t > 0
$ and following the notation of Remark~\ref{rem:fourier}:
\begin{equs}
    \langle \mZ^{t} \rangle_{t} & \leqslant 
\sum_{k, l \in
      \ZZ^{d}} (1 + | k |)^{2 \gamma_{1}} \int_{0}^{t}
(t -s)^{- 2 \zeta}
       e^{- 2 \nu_{\mathrm{min}}(t -s)| k |^{2 \a}}
    | \hat{u}_{s}^{k- l}|^{2} 
  \overline{\Gamma}_{l} \ud s \\
  & \lesssim \int_{0}^{t} (t -s)^{- 2 \zeta}\sum_{k \in
    \ZZ^{d}} (1 + | k |)^{2 \gamma_{1}} \sum_{l \in \ZZ^{d}} \overline{\Gamma}_{l}^{2}
 | \hat{u}^{k- l}_{s} |^{2} \ud s \;.
  \end{equs}
Now we use that by assumption $
\overline{\Gamma}_{l} \lesssim (1 + | k |)^{- 2 \gamma_{0}} $ together with
   Lemma~\ref{lem:bound-weights}, so that
\begin{equation*}
\begin{aligned}
\langle \mZ^{t} \rangle_{t} & \lesssim_{\Gamma} t^{1 - 2 \zeta}\sup_{0 \leqslant s \leqslant t} \sum_{k , l \in \ZZ^{d}} (1 + | l
|)^{2(\gamma_{1}- \gamma_{0})}  (1 + | k-l
|)^{2\gamma_{1}} | \hat{u}_{s}^{k-l}|^{2} \\
& \lesssim_{\Gamma} t^{1 - 2 \zeta} \sup_{0 \leqslant s \leqslant t}\|
u_{s} \|_{H^{\gamma_{1}}}^{2}\;,
\end{aligned}
\end{equation*}
where  by \eqref{e:gamma1} we have that  $ 2(\gamma_{0}- \gamma_{1}) > d $.
From here on the study of well-posedness of the equation follows classical lines.
\end{proof}

The next result guarantees that the angular component is
defined for all positive times. This result is not completely obvious: as a matter
of fact we make use of the full strength of the results in the previous
sections (including the higher order regularity estimates in
Proposition~\ref{prop:higher-regularity}). We do however believe that this is overkill and
that one could get such a result as a consequence of a pathwise form of backwards uniqueness.
Unfortunately we were unable to find such a result in the existing literature, although
\cite{BackwardsUnique} covers the case $\a = 1$ and $m = 1$.

\begin{lemma}\label{lem:nonzero}
Under the assumptions of Theorem~\ref{thm:lyap-func}, for all $ u_{0}  \in L^{2}_{\star}(\TT^{d}; \RR^{m})$, the solution $
t \mapsto u_{t} $ to \eqref{eqn:main} almost surely satisfies $
u_{t} \neq 0 $ for all $ t \geqslant 0 $. In addition, for every $ \ve \in
(0, 1), t \geqslant 0 $ and $ R \geqslant  1 $ there exists a $ \delta (\ve, R, t) \in (0, 1) $ such that
\begin{equ}[e:ubd]
\sup_{u_{0} \in S^{\a}_{R}} \PP \Big( \sup_{0 \leqslant s \leqslant t} \| u_{s} \|
\leqslant \delta \Big) \leqslant \ve \;,
\end{equ}
with
\begin{equ}
S^{\a}_{R} = S \cap \{ \varphi \in H^{\a}  \; \colon \; \| \varphi
\|_{H^{\a}} \leqslant R \} \;.
\end{equ}
Finally, the process $ t \mapsto \pi_{t} = u_{t} / \| u_{t} \|  $ is a Markov process.
\end{lemma}

\begin{proof}

We restrict to proving the claim in \eqref{e:ubd}. The fact that for $
u_{0} \in L^{2}_{\star} $, almost surely  $ u_{t} \neq 0 $
for all $ t \geqslant 0 $, follows along similar lines: the only difference is
that when considering the skeleton median $ (M_{t})_{t \geqslant 0} $ and the
related stopping times we must stop at the potential hitting time $
\tau_{0} = \inf \{ t \geqslant 0  \; \colon \; \| u_{t} \| = 0 \}$, or else the
process defined in the previous sections might not be defined.

As for \eqref{e:ubd}, consider an initial condition $
u_{0} \in S^{\a}_{R}$. Our aim will be to use \eqref{e:r} via an upper bound on
the following quantity (see the proof of Corollary~\ref{cor:finite-lyap}
for a somewhat similar argument):
\begin{equ}
\EE \left[ \sup_{0 \leqslant s \leqslant t } \| \pi_{s} \|_{H^{\a}}^{2} \right]
\;,
\end{equ}
where $ t \mapsto \pi_{t} = u_{t} / \| u_{t} \| $ is the angular component
associated to $ u $.
To obtain an estimate on the quantity above we follow roughly the proof of Proposition~\ref{prop:higher-regularity}.
Indeed, for any $ i \in \NN $ consider the probability $ \mathbf{p}_{t} (i) = \PP ( t \in [
T_{i}, T_{i +1})) $. Then we can bound
\begin{equ}
\EE \left[ \sup_{0 \leqslant s \leqslant t } \| \pi_{s} \|_{H^{\a}}^{2} \right]
\leqslant \sum_{i \in \NN} \sqrt{\mathbf{p}_{t}(i)} \left( \EE \left[ \sup_{0
\leqslant s \leqslant T_{i+1} } \| \pi_{s} \|_{H^{\a}}^{4}\right]
\right)^{\frac{1}{2}} \;.
\end{equ}
Now, for every $ i \in \NN $ we have, similarly to \eqref{e:ref} in the proof
of Proposition~\ref{prop:higher-regularity}
\begin{equation*}
\begin{aligned}
\EE_{T_{i}} \left[  \sup_{T_{i} \leqslant s < T_{i+1}} \|
\pi_{s} \|_{H^{\a}}^{4}\right] & \lesssim M_{T_{i}}^{4 \a} + \EE_{T_{i}} \left[
\sup_{T_{i} \leqslant s < T_{i+1} } \| w_{s} \|_{\a, M_{T_{i}}}^{4}\right] \\
& \lesssim  \mf{F}(1/2, \bpi_{T_{i}} ) + \| w_{T_{i}} \|_{\a,
M_{T_{i}}}^{4} \;,
\end{aligned}
\end{equation*}
where the last bound follows from Lemma~\ref{lem:high-reg-disc}. Therefore, following verbatim the proof of
Proposition~\ref{prop:higher-regularity} we obtain that for every $ i \in \NN
\setminus \{ 0 \} $ and some constant $ \overline{C} (\a) $ (uniform over $
u_{0} $ and $ i $)
\begin{equation*}
\begin{aligned}
 \EE \left[  \sup_{T_{i} \leqslant s < T_{i+1}} \|
\pi_{s} \|_{H^{\a}}^{4}\right] \leqslant \overline{C} (\a) \mf{G} ( \pi_{0}) \;, \qquad
\EE \left[  \sup_{0 \leqslant s < T_{1}} \|
\pi_{s} \|_{H^{\a}}^{4}\right] \leqslant \overline{C} (\a) \| \pi_{0}
\|_{H^{\a}}^{4} \;.
\end{aligned}
\end{equation*}
Now, if we bound
\begin{equation*}
\begin{aligned}
 \EE \left[ \sup_{0
\leqslant s \leqslant T_{i+1} } \| \pi_{s} \|_{H^{\a}}^{4}\right]
\leqslant \sum_{j = 0}^{i}\EE \left[ \sup_{T_{j}
\leqslant s \leqslant T_{j+1} } \| \pi_{s} \|_{H^{\a}}^{4}\right] \;,
\end{aligned}
\end{equation*}
then collecting all our estimates we can then conclude via
Lemma~\ref{lem:stopping-density-bounds} that
\begin{equ}[e:bdfin]
\EE \left[ \sup_{0 \leqslant s \leqslant t } \| \pi_{s} \|_{H^{\a}}^{2} \right]
\lesssim e^{c \| u_{0} \|_{H^{\a}}^{2}} \sum_{i} \sqrt{
\mathbf{p}_{t} ( i) \cdot i} \lesssim_{t} e^{c \| u_{0} \|_{H^{\a}}^{2}} \;.
\end{equ}
Now to conclude we will make use of \eqref{e:r}, which allows us to represent
\begin{equs}
\| u_{t} \| = \| u_{0} \| \exp \left( \int_{0}^{t} \langle \pi_{s}, \mL \pi_{s} \rangle \ud
s + \int_{0}^{t} \langle \pi_{s}, \pi_{s} \cdot \circ \ud W_{s} \rangle \right)
\;.
\end{equs}
The results follows then immediately from \eqref{e:bdfin}, since the stochastic
integral has bounded It\^o--Stratonovich corrector and the It\^o integral has
bounded quadratic variation: see for example the proof of
Corollary~\ref{cor:finite-lyap}.

The fact that $ \pi_{t} $ is a Markov process follows for example from the
representation in \eqref{e:pi}, since by our argument above $
Q_{\mL}(\pi_{t}, \pi_{t}) $ is finite 
for all $ t > 0 $.
\end{proof}
The next result guarantees the Feller property of the semigroup associated to
the angular component.

\begin{lemma}\label{lem:cont-ic}
In the setting of Theorem~\ref{thm:lyap-func}, fix $ \gamma \in [\a, \a +
\gamma_{0} - d/2) $. For any $ \pi \in S \cap H^{\gamma}$, let $
\mP_{t}^{\gamma} (\pi, \cdot) $ be the
law in $ H^{\gamma} $ of the angular component $ \pi_{t} $ of
\eqref{eqn:main} started
in $ \pi_{0} = \pi $. Then the map $ \pi \mapsto \mP_{t}^{\gamma}
(\pi, \cdot) $ is continuous from $
S \cap H^{\gamma} $ to the space of probability measures on $
S \cap H^{\gamma} $ equipped with the topology of weak convergence.
\end{lemma}

\begin{proof}
It is sufficient to show that if $ \pi^{n} \to \pi $ in $ H^{\gamma}
$, then for every $ t > 0
$ one has $\pi_t^n \to \pi_t$ in probability in $H^\gamma$. 
We know that $
u^{n}_{t}  \to u_{t} $ in probability in $ H^{\gamma} $ by
Lemma~\ref{lem:wp} (in fact, we have convergence in $ L^{p}(\Omega) $ for any $ p
\geqslant 2 $), and the uniform
bound in Lemma~\ref{lem:nonzero} implies that in addition
$u^{n}_{t} / \| u^{n}_{t} \| \to u_{t} / \| u_{t} \|$ in probability, as required.
\end{proof}

\glsaddall
\printglossaries

\bibliographystyle{Martin}

\def\polhk#1{\setbox0=\hbox{#1}{\ooalign{\hidewidth
  \lower1.5ex\hbox{`}\hidewidth\crcr\unhbox0}}}

\end{document}